\documentclass[a4paper,twoside,10pt]{article}

\usepackage[a4paper,left=3cm,right=2.5cm, top=3cm, bottom=3cm]{geometry}
\usepackage[latin1]{inputenc}
\usepackage{mathrsfs}
\usepackage{enumerate}
\usepackage{graphicx}
\usepackage{epstopdf}
\usepackage{amsmath}
\usepackage{amsthm}
\usepackage{algorithm,algpseudocode}
\usepackage{amssymb}
\usepackage{hyperref}
\usepackage{stmaryrd}
\usepackage{subfigure}
\usepackage{float}
\usepackage{bigints}
\usepackage{cite}
\usepackage{color}
\usepackage[abs]{overpic}
\usepackage[font=footnotesize,labelfont=bf]{caption}
\usepackage{cases}
\usepackage[author={Lorenzo}]{pdfcomment}
\usepackage{soul}
\usepackage[table]{xcolor}
\usepackage{comment}

\restylefloat{table}
\theoremstyle{plain}
\newtheorem{thm}{Theorem}[section]
\newtheorem{cor}[thm]{Corollary}
\newtheorem{lem}[thm]{Lemma}
\newtheorem{prop}[thm]{Proposition}
\theoremstyle{definition}

\theoremstyle{remark}
\newtheorem{remark}[thm]{Remark}
\newtheorem{assumption}[thm]{Assumption}

\newcommand{\sm}{\sigma_m} 
\newcommand{\sv}{\sigma_v} 

\newcommand{\eremk}{\hbox{}\hfill\rule{0.8ex}{0.8ex}}
\newcommand{\h}{h} 
\newcommand{\PItD}{\mathcal P_{ItD}} 
\newcommand{\T}{\mathcal T} 
\newcommand{\A}{\mathcal A} 
\newcommand{\Aprime}{\mathcal A'} 
\newcommand{\B}{\mathcal B} 
\newcommand{\Ak}{\A_k} 
\newcommand{\Aprimek}{\Aprime_k} 
\newcommand{\Bk}{\B_k} 
\newcommand{\vtilde}{v^{ext}} 
\newcommand{\vext}{\vtilde} 
\newcommand{\vhtilde}{\vtilde_\h} 
\newcommand{\lambdah}{\lambda_\h} 
\newcommand{\uext}{u^{ext}} 
\newcommand{\uhext}{\uext_\h} 
\newcommand{\f}{f} 
\newcommand{\ftilde}{\widetilde \f} 
\renewcommand{\k}{k n} 
\renewcommand{\Re}{\operatorname{\mathbb {RE}}} 
\newcommand{\m}{m} 
\newcommand{\mh}{\m_\h} 
\newcommand{\xbold}{\mathbf x} 
\newcommand{\ybold}{\mathbf y} 
\newcommand{\nbold}{\mathbf n}
\newcommand{\Omegap}{\Omega^+} 
\newcommand{\nGamma}{\mathbf n_{\Gamma}} 
\newcommand{\V}{\mathcal V} 
\newcommand{\Vz}{\V_0} 
\newcommand{\Vtilde}{\widetilde \V} 
\newcommand{\Vk}{\V_k} 
\newcommand{\Vtildek}{\Vtilde_k} 
\newcommand{\Ntildek}{\Ncal} 
\newcommand{\Ncal}{\widetilde{\mathcal N}_k} 
\newcommand{\K}{\mathcal K} 
\newcommand{\Kz}{\K_0} 
\newcommand{\Kprime}{\K'} 
\newcommand{\Kprimez}{\Kprime_0} 
\newcommand{\Ktilde}{\widetilde \K} 
\newcommand{\W}{\mathcal W} 
\newcommand{\Wz}{\W_0} 
\newcommand{\Kk}{\K_k} 
\newcommand{\Kprimek}{\Kprime_k} 
\newcommand{\Ktildek}{\Ktilde_k} 
\newcommand{\Wk}{\W_k} 
\newcommand{\Pcalh}{\mathcal P_\h} 
\renewcommand{\S}{\mathcal S} 
\newcommand{\p}{p} 

\newcommand{\SV}{\S_{\V}}
\newcommand{\SK}{\S_{\K}}
\newcommand{\SKprime}{\S_{\Kprime}}
\newcommand{\SW}{\S_{\W}}
\newcommand{\SVtilde}{\widetilde{\S}_{\V}}
\newcommand{\SKtilde}{\widetilde{\S}_{\K}}

\newcommand{\AVtilde}{\widetilde{\A}_{\V}}
\newcommand{\AKtilde}{\widetilde{\A}_{\K}}

\newcommand{\gammaz}{\gamma_0} 
\newcommand{\gammazint}{\gammaz^{int}} 
\newcommand{\gammazext}{\gammaz^{ext}} 
\newcommand{\gammao}{\gamma_1} 
\newcommand{\gammaoint}{\gammao^{int}} 
\newcommand{\gammaoext}{\gammao^{ext}} 
\newcommand{\psim}{\psi_m} 
\newcommand{\psitilde}{\psi^{ext}} 
\newcommand{\psiext}{\psi^{ext}} 
\newcommand{\psih}{\psi_\h} 
\newcommand{\psimh}{\psi_{m\h}} 
\newcommand{\psitildeh}{\psi_\h^{ext}} 
\newcommand{\n}{\mathbf n} 
\newcommand{\Lcal}{\mathcal L} 
\newcommand{\Lcalk}{\Lcal_k} 
\newcommand{\C}{\mathcal C} 
\newcommand{\Ck}{\C_\k} 
\newcommand{\M}{\mathcal M} 

\newcommand{\Rm}{R_\m}
\newcommand{\Rext}{R^{ext}}
\newcommand{\Mk}{\M_\k} 
\newcommand{\U}{\mathcal U} 
\newcommand{\Uk}{\U_\k} 
\newcommand{\aaleph}{\aleph_k} 
\newcommand{\bbeth}{\beth_k} 
\newcommand{\ddaleth}{\daleth_k} 
\newcommand{\ddalethint}{\daleth^{int}_k} 
\newcommand{\ddalethext}{\daleth^{ext}_k} 
\newcommand{\Fcal}{\mathcal{F}} 
\newcommand{\uh}{u_h} 
\newcommand{\vh}{v_h} 
\newcommand{\nh}{n_\h} 
\newcommand{\vexth}{v_\h^{ext}} 
\newcommand{\Vh}{V_h} 
\newcommand{\Wh}{W_h} 
\newcommand{\Zh}{Z_h} 
\newcommand{\Calderon}{\text{Calder\'on}}
\newcommand{\cG}{c_G}
\newcommand{\rr}{r}
\renewcommand{\rm}{\rr_\m}
\newcommand{\rext}{\rr^{ext}}
\renewcommand{\div}{\operatorname{div}}
\newcommand{\Ad}{\mathfrak a}
\newcommand{\RHSLcalk}{RHS_{\Lcalk}}

\newcommand{\Pcal}{\mathcal P}
\newcommand{\Rcal}{\mathcal R}
\newcommand{\Zcal}{\mathcal Z}

\definecolor{ilariagreen}{rgb}{0 0.7 0.2}

\newcommand{\vertiii}[1]{{\left\vert\kern-0.25ex\left\vert\kern-0.25ex\left\vert #1  \right\vert\kern-0.25ex\right\vert\kern-0.25ex\right\vert}}

\numberwithin{equation}{section}
 
\begin{tiny}
\author{Lorenzo Mascotto \thanks{Fakult\"at f\"ur Mathematik, Universit\"at Wien, 1090 Vienna, Austria\newline E-mail: {\tt lorenzo.mascotto@univie.ac.at, ilaria.perugia@univie.ac.at, alexander.rieder@univie.ac.at}} \and 
Jens M.~Melenk \thanks{Institut f\"ur Analysis und Scientific Computing, TU Wien, 1040 Vienna, Austria\newline E-mail: {\tt melenk@tuwien.ac.at}} \and 
Ilaria Perugia \footnotemark[1] 
\and 
Alexander Rieder \footnotemark[1]}
\end{tiny}

\date{}

\title{\textbf{\normalsize{FEM-BEM mortar coupling for the Helmholtz problem in three dimensions}}}

\begin{document}
\maketitle

\begin{abstract}
\noindent
We present a FEM-BEM coupling strategy for time-harmonic acoustic scattering  in media with variable sound speed.
The coupling is realized with the aid of a mortar variable that is an impedance trace on the coupling boundary.
The resulting sesquilinear form is shown to satisfy a  G{\aa}rding inequality.
Quasi-optimal convergence is shown for sufficiently fine meshes.
Numerical examples confirm the theoretical convergence results.  

\end{abstract}

\section{Introduction}
We analyze a numerical method for acoustic scattering in media with variable  sound speed.
The speed of sound may be variable in a bounded domain~$\Omega$  whereas it is assumed to be constant in~$\Omegap:= {\mathbb R}^3 \setminus \overline{\Omega}$.
Mathematically, in the time-harmonic setting under consideration here, this problem is described by the Helmholtz equation
\begin{equation}
\label{eq:helmholtz} 
-\div (\Ad \nabla u )- (kn)^2 u = \ftilde \quad \mbox{
  in~${\mathbb R}^3$},
\end{equation}
where~$\Ad$ is a smooth diffusion parameter such that the
  support of~$1-\Ad$ is contained in~$\Omega$, $\ftilde$ is a 
  function with support contained in $\Omega$,
$k\ge k_0> 0$ is the wavenumber, and the function $n \in L^\infty({\mathbb R}^3)$ is the ratio of the local sound speed to the speed of sound in the homogeneous part~$\Omegap$. 
In particular, $n\equiv 1$ in~$\Omegap$. We also assume the
medium is fully penetrable, i.e., $\vert n(\xbold) \vert \ge c_0 > 0$
a.e. in~${\mathbb R}^3$.

One computational challenge for this problem class is that the problem is posed in the full space~${\mathbb R}^3$.
Since the medium is assumed homogeneous in the unbounded domain~$\Omegap$, boundary integral equations techniques can be employed to reduce  the problem in~$\Omegap$ to~$\partial\Omega$ 
and subsequently use a boundary element method (BEM) for the discretization. 
In the bounded region~$\Omega$, the physical situation may be more complex and a description by partial differential equations and a numerical treatment by the finite element method (FEM) is more appropriate.
These considerations make a FEM-BEM coupling attractive. 
In the present work, we analyze a FEM-BEM coupling that 
is based on three fields: the solution~$u$ in~$\Omega$, the exterior Dirichlet trace~$u^{ext}$ of the solution on~$\partial\Omega$, and  a mortar variable~$m$ on~$\partial\Omega$, which is the Robin trace $m = \partial_{\nbold} u + \operatorname{i} k u$ of the solution~$u$ on~$\partial\Omega$. 

For symmetric positive definite problems, various FEM-BEM coupling techniques 
have been presented and analyzed in the past such as the so-called
symmetric coupling of Costabel, \cite{costabel88a}, which uses both
equations of the Calder\'on system, and so-called one-equation 
couplings that employ only one of the two equations of the Calder\'on 
system, 
\cite{sayas09,steinbach11,aurada-feischl-fuehrer-karkulik-melenk-praetorius13}. 
Also so-called three-field techniques, which are similar to the coupling strategy
pursued here, have been analyzed in, e.g., \cite{carstensen-funken99,erath12} in the symmetric setting.
  
While coupling is reasonably well-understood for symmetric positive definite
problems, the situation is less developed for Helmholtz problems. 

One strategy to solve~\eqref{eq:helmholtz} 
is to rely on the exterior Dirichlet-to-Neumann (DtN) operator; see, e.g., the analysis in~\cite{melenk-sauter10}.
Numerically, the resulting system has a fully populated sub-block arising from the DtN operator, and the 
full system matrix has no easily identifiable invertible sub-blocks. 
From a computational point of view, however, methods that have identifiable invertible subproblems are of interest. In the present work, 
therefore, we study a coupling strategy 
where the subproblem for the unknown~$u$ (the solution on~$\Omega$) 
is a well-posed problem. Such a strategy cannot rely on the interior 
DtN or Neumann-to-Dirichlet (NtD) maps if~$k$ is an eigenvalue of the 
interior Dirichlet or Neumann problem;
see, e.g., ~\cite{Gatica_VEMBEM} and the discussion in~\cite[Sec.~{2}, {3}]{steinbach13}.
One technique to overcome this deficiency of the interior DtN or NtD operator 
is to employ---explicitly or implicitly---a subproblem
that corresponds to a Helmholtz problem with Robin boundary conditions. 
This route is taken in~\cite{steinbach-windisch11,steinbach13, hiptmair-meury06,hiptmair-meury08, kirsch-monk90,kirsch-monk94,kirsch-monk95,ganesh-morgenstern16}.
It is also the approach taken in our formulation in which, 
once the Robin trace~$m$ is known, the subproblems for the 
solution~$u$ on~$\Omega$ and the exterior
Dirichlet trace~$u^{ext}$ are well-posed. 

Several coupling techniques and approaches for solving full-space
problems such as~\eqref{eq:helmholtz} are available in the literature. 
The coupling techniques mentioned above fall in the class of 
single-trace-formulations (STF). More generally, multi-trace-formulations (MTF)  are also possible.
Here one works with separate approximations~$u$ and~$u^{ext}$ on~$\Omega$ and~$\Omegap$ and enforces continuity 
of both the values and the normal derivatives across~$\partial\Omega$; see~\cite{claeys-hiptmair-jerez-hanckes13, claeys-hiptmair13, claeys-hiptmair12}. 
The coupling methodologies discussed so far are nonoverlapping, in the sense that problems on two disjoint domains~$\Omega$ and~$\Omegap$ are considered and the coupling is performed on the common interface~$\partial\Omega$.
Recent alternatives are overlapping techniques, where for two bounded domains~$\Omega_1 \subset \Omega_2$ with~$\Omega_1^+ \cup \Omega_2 = {\mathbb R}^3$, two approximations~$u_1$ (defined on~$\Omega_1^+$) 
and~$u_2$ (defined on~$\Omega_2$) are sought and coupled by the condition that~$u_1 - u_2$ be small on~$\Omega_1^+ \cap \Omega_2$.
An attractive feature of this coupling approach is the freedom in the choice of~$\Omega_1$ and~$\Omega_2$.
While~$\Omega_2$ may be chosen to accommodate
a convenient FEM discretization, the exterior problem for~$u_1$ 
can be attacked by boundary integral equation techniques and, for smooth~$\partial\Omega_1$, rapidly convergent methods; we refer to~\cite{casati-hiptmair-smajic18,dominguez-ganesh-sayas19} and the references therein.
Another class of methods for~\eqref{eq:helmholtz} is based on the Lippmann-Schwinger equation volume-integral equation; we refer to~\cite{mckay-bruno05} and references therein.
For convex~$\Omega$, methods based on absorption such as the PML method of B\'erenger~\cite{berenger94} can be employed,  and have the attractive feature of being formulated in terms of differential operators, thus avoiding integral operators altogether;
we refer to~\cite{collino-monk98,bramble-pasciak07} and references therein.
A last class of methods worth mentioning is that based on ``infinite elements'' where the discretization of the unbounded domain is realized by nonpolynomial functions;
we mention~\cite{demkowicz-ihlenburg01} and in particular the method based on the so-called  ``pole condition'', \cite{schmidt-deuflhard95,hohage-schmidt-zschiedrich-I,hohage-schmidt-zschiedrich-II,hohage-nannen09,hohage-nannen15}. 

As already underlined, the Galerkin formulation analyzed in the present paper employs three fields, the solution~$u$ in~$\Omega$, the exterior trace~$u^{ext}$ to represent the solution in~$\Omegap$
and the Robin trace~$m = \partial_{\nbold} u + iku$.
For smooth~$\partial\Omega$, we show coercivity of the sesquilinear form up to a compact perturbation so that we are able to show (asymptotic) quasi-optimality of the Galerkin method
under the usual resolution condition that the discretization be sufficiently fine. 
Our FEM-BEM coupling is related to earlier 2D FEM-BEM coupling approaches in~\cite{kirsch-monk90,kirsch-monk94,ganesh-morgenstern16},  where the coupling is also realized through the same Robin trace~$m$.
Differently from our approach, \cite{kirsch-monk90} restricts to circular coupling boundaries so that 
the exterior Impedance-to-Dirichlet operator~${\mathcal P}_{ItD}$ (see Proposition~\ref{proposition:CFIE}) is explicitly  available and the use of the extra variable~$u^{ext}$ is obviated. 
References~\cite{kirsch-monk94,ganesh-morgenstern16} also avoid the explicit use of 
$u^{ext}$ by realizing the operator~${\mathcal P}_{ItD}$ 
by a rapidly convergent Nystr\"om method.  

The novelty of the present approach is in the choice of the coupling variable $m$. Choosing 
it as an impedance trace
leads to a system that has block structure with invertible subblocks for $u$ and $\uext$.
Computationally established FEM or BEM tools could be used for these subproblems. 
Stability (i.e., G{\aa}rding inequality) of the method is inherited from the 
stability assumption~\eqref{assumption:uniqueness} irrespective of the choice of 
coupling boundary $\Gamma$. This flexibility in the choice of $\Gamma$ can be exploited
to facilitate the meshing or the realization of relevant boundary integral operators. 

The present work refrains from a sharp wavenumber-\emph{explicit} theory 
as was done in \cite{melenk-sauter10, melenk-sauter11, MelenkParsaniaSauter_generalDGHelmoltz} for high order FEM, and in \cite{loehndorf-melenk11} 
for high order BEM. 
First steps towards such a goal are achieved in the appendix
by presenting a $k$-explicit G{\aa}rding inequality.
However, a sharp $k$-explicit analysis would require a more 
elaborate regularity theory of various dual problems than what is done in 
Section~\ref{subsection:regularity:dual:problem}.
Then, it would have to follow the path taken in~\cite{loehndorf-melenk11} for scattering problems and in~\cite{melenk-sauter20} for Maxwell's equations.

\medskip
\paragraph*{Notation.}
We employ standard Sobolev spaces as introduced in, e.g.,~\cite{mclean2000strongly}. 
For~$s \in {\mathbb N}_0$ and bounded Lipschitz domains~$D \subset
{\mathbb R}^3$, we define the norms $\|v\|_{s,D}$ for complex-valued
functions~$v$ by $\|v\|^2_{s,D} = \sum_{\alpha\colon|\alpha| \leq s}
\|D^\alpha v\|^2_{0,D}$, and seminorms $|v|^2_{s,D}=\sum_{\alpha\colon|\alpha| = s}
\|D^\alpha v\|^2_{0,D}$.
The Sobolev spaces~$H^s(D)$ are defined as the closure of~$\mathcal C^\infty (D)$
under the norm~$\|\cdot\|_{s,D}$; $H^0(D)=:L^2(D)$. The Sobolev spaces~$H^s_0(D)$  are
defined as the closure of~$\mathcal C^\infty_0(D)$ under this norm.
For positive noninteger~$s$, the spaces~$H^s(D)$ are defined by interpolation between~$H^{\lfloor s\rfloor}(D)$ and~$H^{\lceil s\rceil}(D)$, and the spaces~$H^s_0(D)$ by interpolation between~$H^{\lfloor s\rfloor}_0(D)$ and~$H^{\lceil s\rceil}_0(D)$.
For~$s > 0$, the space~$H^{-s}(D)$ is defined as the dual of~$H^s_0(D)$ with norm 
\[
\|v\|_{-s,D} = 
\sup_{w \in H^s_0(D)} \frac{ |  \langle w, v\rangle | }{\|w\|_{s,D}}.
\]
Here and in the following, $\langle \cdot,\cdot \rangle$ denotes the
duality pairing, which is taken to  to be linear in the first argument and
antilinear in the second argument; thus, it 
coincides with the $L^2(D)$ inner product in case both~$v$ and~$w \in L^2(D)$.
The Sobolev spaces~$H^s(D)$ are Hilbert spaces, and we
write~$(\cdot,\cdot)_{s,D}$ for the corresponding inner
product.

For closed, connected smooth 2-dimensional surfaces $\Gamma \subset {\mathbb R}^3$ and~$s \ge 0$,
we define the Sobolev spaces~$H^s(\Gamma)$ in terms of the eigenfunctions of the Laplace-Beltrami operator.
To this end, let $\{ \varphi_n, \lambda_n  \}_{n \in \mathbb N_0}$ be a sequence of eigenpairs of the Laplace-Beltrami operator on~$\Gamma$, which we assume to be normalized so that they form an $L^2(\Gamma)$-orthonormal basis.
For~$s \ge 0$, we define the norm $\|v\|^2_{s,\Gamma}:= \sum_{n} |v_n|^2 (1+\lambda_n)^s$, 
where $v = \sum_{n} v_n \varphi_n$ is expanded in the $L^2(\Gamma)$-orthogonal basis $\{\varphi_n\}_{n \in {\mathbb N}_0}$.
This norm is equivalent to the one obtained by using local charts as described in ~\cite{mclean2000strongly}; see, e.g., \cite[Sec.~{5.4}]{nedelec01}.
Negative order Sobolev spaces are defined by duality and equipped with the norm
\begin{equation}
\label{eq:dual-Gamma}
\|v\|_{-s,\Gamma} = 
\sup_{w \in H^s(\Gamma)} \frac{| \langle w,v\rangle | }{\|w\|_{s,\Gamma}}.
\end{equation}
Again, $\langle \cdot,\cdot\rangle$ is the duality pairing, with the understanding to coincide with the $L^2(\Gamma)$ inner product if both arguments are in~$L^2(\Gamma)$.
The mapping $v \mapsto (v_n)_{n \in {\mathbb N}_0}$, with $v_n = (v,\varphi_{n})_{L^2(\Gamma)}$, is an isometric isomorphism between~$H^s(\Gamma)$ and the sequence space $\{(v_n)_{n \in {\mathbb N}_0}\,|\, \sum_n |v_n|^2 (1+\lambda_n)^s <\infty\}$.
Likewise, the dual space~$H^{-s}(\Gamma)$ of $H^s(\Gamma)$ can be identified with the sequence space $\{(v_n)_{n \in {\mathbb N}_0}\,|\, \sum_n |v_n|^2 (1 + \lambda_n)^{-s}<\infty\}$,
and the norm~\eqref{eq:dual-Gamma} is given by $\| v\|^2_{-s,\Gamma} = \sum_{n} |v_n|^2(1+\lambda_n)^{-s}$.
When identifying the spaces~$H^s(\Gamma)$ and~$H^{-s}(\Gamma)$ with sequence spaces in this way, the duality pairing~$\langle w,v\rangle$ takes the form
\begin{equation}
\label{eq:duality-pairing-Gamma}
\langle w,v\rangle = \sum_{n \in \mathbb N_0} w_n \overline{v_n}. 
\end{equation}
For $s \in  {\mathbb R}$, the spaces~$H^{s}(\Gamma)$ are Hilbert spaces endowed with the inner product~$(\cdot,\cdot)_{s,\Gamma}$, which takes the form
\begin{equation} \label{negative:Sobolev:sesquilinear:form}
(w,v)_{s ,\Gamma} = \sum_{n \in \mathbb N_0} w_n \overline{v_n} (1+\lambda_n)^{s},
\end{equation}
when the elements of~$H^{s}(\Gamma)$ are characterized by sequences through the above mentioned isometric isomorphisms. 
The seminorm~$|\cdot|_{\frac12,\Gamma}$
in~$H^{\frac12}(\Gamma)$ is defined by factoring out constant functions:
$|v|_{\frac12,\Gamma} =\inf_{c \in  {\mathbb C}}  \|v -
  c\|_{\frac12,\Gamma}$.

For bounded linear operators $K:H^s(\Gamma) \rightarrow
H^{s'}(\Gamma)$, we will require the adjoint $K^\ast:H^{-s'}(\Gamma) \rightarrow H^{-s}(\Gamma)$
defined by $\langle K^\ast w,v\rangle = \langle w,Kv\rangle$, where
the  duality pairings are between the appropriate spaces. By taging
the conjugate, this also
implies $\langle Kv,w\rangle = \langle v , K^\ast w\rangle$.

Given two positive quantities~$a$ and~$b$, we write~$a \lesssim b$
and~$a \gtrsim b$ whenever there exists a positive constant~$c$
such that~$a \le c \, b $ and~$a \ge c\,b $, respectively (the dependence of the constant is specified at each occurrence).

\paragraph*{Outline of the paper.}
In Section~\ref{section:problem}, we introduce the problem we are interested in,
and we recall basic tools from the theory of boundary integral operators.
This problem is formulated in variational formulation in Section~\ref{section:continuous:mortar:coupling}.
Here, the auxiliary mortar variable is introduced.
The well-posedness and the regularity of the solution to the dual problem, including the proof of a G\r arding inequality, are investigated as well.
Section~\ref{section:FEM-BEM:mortar} is devoted to the construction
of a FEM-BEM mortar coupling for the discretization of the continuous problem.
Quasi-optimality of the Galerkin method is established under a
resolution condition on the approximation spaces.
Numerical results and comments on the implementation of the method are
presented in Section~\ref{section:numerical:results}.
Finally, conclusions are drawn in
Section~\ref{section:conclusions}. In the appendix, we present wavenumber-explicit
continuity estimates and a G{\aa}rding inequality for the case of an analytic coupling surface $\Gamma$.

\section{The continuous problem and the functional setting} \label{section:problem}
In the present section we introduce the target problem, the functional setting, some notation, and recall useful definitions and properties 
of boundary integral potentials and operators.

Given~$\Omega$ a bounded domain in~$\mathbb R^3$, we define $\Omegap:= \mathbb R^3 \setminus \overline \Omega$ and $\Gamma := \overline \Omega \cap \overline{\Omegap}$, and we associate with~$\Gamma$ the normal vector~$\nGamma$ pointing outwards~$\Omega$.
Let~$\Ad=\Ad(\xbold)$ be a diffusion parameter in~$C^{\infty}(\mathbb R^3)$ such that~$0 < \alpha_* \le \Ad \le \alpha^*$ in~$\mathbb R^3$, for two constants~$\alpha_*$ and~$\alpha^*$.
The forthcoming analysis extends also to the case of a positive definite diffusion tensor~$\Ad$. For the sake of exposition, we stick to the scalar case.
Consider the wavenumber function~$\k(\xbold)$, where~$k \in \mathbb
R$, $k\ge k_0>0$, denotes the angular frequency, and the
function~$n\in L^{\infty}(\mathbb R^3, \mathbb C)$ such that $\vert n(\xbold) \vert \ge c_0 > 0$ denotes the material refraction index.
We assume that
\begin{equation}
\label{eq:assumption-on-a} 
\text{$\Ad \equiv 1$ and $n\equiv 1$ a.e.\ on~$\Omegap \cup \mathcal N(\Gamma)$ and $\mathcal N(\Gamma) = $ an open neighborhood of~$\Gamma$.}
\end{equation}
For source terms of the form 
\[
\widetilde f (\xbold)= \begin{cases}
\f (\xbold) & \text{in } \Omega,\\
0 & \text{elsewhere,}
\end{cases}
\]
with $f\in L^2(\Omega)$ with compact support,
we consider the three dimensional Helmholtz problem: find $u:\mathbb
R^3 \rightarrow \mathbb C$ such that
\begin{equation} \label{Helmholtz:problem:on:unbounded:domain}
\begin{cases}
-\div(\Ad \nabla  u) -(\k)^2 u=\widetilde  \f \quad \quad  \text{in } \mathbb R^3, \\
\displaystyle{\lim_{\vert \xbold \vert\rightarrow +\infty}}\vert \xbold
  \vert\left( \partial_{\vert\xbold\vert} u  - i k u   \right) =0,
\end{cases}
\end{equation}
which can be also rewritten as a transmission problem.
To this end, we introduce the jumps of functions and of their normal derivatives across the interface~$\Gamma$.
Given $\varphi \in H^1(\mathbb R^3\setminus \Gamma)$, we denote by~$\gammazint(\varphi)$ and~$\gammazext(\varphi)$ the Dirichlet traces over~$\Gamma$ of the restrictions of~$\varphi$ over~$\Omega$ and~$\Omegap$, respectively,
whereas, we denote by~~$\gammaoint(\varphi)$ and~$\gammaoext(\varphi)$ the Neumann traces over~$\Gamma$ of the restrictions of~$\varphi$ to~$\Omega$ and~$\Omegap$, respectively.
For any $v \in H^1(\mathbb R^3 \setminus \Gamma)$, we
set
\[
  \llbracket v \rrbracket_\Gamma := 
  \gammazint (v) - \gammazext (v),\quad
  \llbracket \partial_{\nGamma} v \rrbracket_\Gamma :=
  \gammaoint(v) - \gammaoext (v).
\]

The transmission problem equivalent to~\eqref{Helmholtz:problem:on:unbounded:domain} consists in finding $u:\mathbb R^3\rightarrow \mathbb C$ such that
\begin{equation}\label{Helmholtz:transmission:problem}
\begin{cases}
-\div(\Ad  \nabla u) -(\k)^2 u= \f \quad\text{in } \Omega, \\
-\Delta u  - k^2 
u= 0 \quad\text{in } \Omegap, \\
\llbracket u \rrbracket_\Gamma =0,\quad \llbracket \partial_{\nGamma} u
\rrbracket_\Gamma =0, \\
\displaystyle{\lim_{\vert \xbold \vert\rightarrow +\infty}}\vert \xbold
  \vert\left( \partial_{\vert\xbold\vert}  u  - i k u   \right) =0.
\end{cases}
\end{equation}
In Section~\ref{section:continuous:mortar:coupling}, we introduce an additional formulation based on adding a mortar variable, which is investigated numerically in Section~\ref{section:FEM-BEM:mortar}.
The remainder of this section is devoted to recall definitions and tools stemming from the theory of boundary integral operators.
\medskip

We assume henceforth the following uniqueness assertion:
\begin{equation} \label{assumption:uniqueness}
\begin{cases}
-\div( \Ad \nabla u) -(\k)^2 u= 0 \quad\text{in } \Omega \cup \Omegap, \\
\llbracket u \rrbracket_\Gamma =0,\quad \llbracket \partial_{\nGamma} u
\rrbracket_\Gamma =0, \\
\displaystyle{\lim_{\vert \xbold \vert\rightarrow +\infty}}\vert \xbold
  \vert\left( \partial_{\vert\xbold\vert}  u  - i k u   \right) =0
\end{cases}
\quad\quad\quad \text{implies} \quad\quad\quad u=0.
\end{equation}
Note that this assumption corresponds to the uniqueness of the solution to problem~\eqref{Helmholtz:transmission:problem}, and is used in the following to apply the Fredholm theory.

\begin{remark}
It is well-known that the uniqueness assertion~\eqref{assumption:uniqueness} holds true for~$\Ad=1$ and~$n=1$.
Indeed, in this case a solution~$u$ is smooth (in fact, analytic) and then it follows from~\cite[(2.10)]{colton-kress98} (by letting the set $D$ shrink to a point there) that~$u \equiv 0$.
The solution is then given by the Newton potential~$u = \Ntildek \f$ defined in~\eqref{Newton:potential}.
For Helmholtz problems with variable coefficients we refer to the recent  works~\cite{burq02, GPS_Helmholtz_het_wellposedness, MoiolaSpence_acoustictrans_wavenumberexplicit, grahamSauter_Helmholtz, chaumont2018wavenumber}
and the references therein for discussions regarding uniqueness as well as the dependence of the solution operator on the wavenumber~$k$.
\eremk
\end{remark}

\subsection{Boundary integral operators for the 3D Helmholtz problem in a nutshell} \label{subsection:BIO:for:Helmholtz}
In this section, we define integral operators for functions on~$\Gamma$, and recall some of their properties. As we assume that the refraction index~$n \equiv 1$ 
in a neighborhood of~$\Gamma$, the wavenumber that
enters these definitions is simply~$k$.

Given
\[
G_k(\xbold, \ybold) = \frac{e^{ik\vert \xbold-\ybold \vert}}{4\pi \vert \xbold - \ybold \vert}
\]
the fundamental solution to the 3D Helmholtz problem, we define the single and double layer potentials by 
\begin{align}
&  \Vtildek \varphi (\xbold) = \int_\Gamma G_k (\xbold - \ybold) \varphi(\ybold) ds(\ybold)  \quad \forall \xbold \in \mathbb R^3 \setminus \Gamma,\, \forall \varphi \in H^{-\frac{1}{2}}(\Gamma), \label{single:layer:potential}\\
&  \Ktildek \varphi (\xbold) = \int_\Gamma \partial_{\nGamma(\ybold)}G_k (\xbold - \ybold) \varphi(\ybold) ds(\ybold)  \quad \forall \xbold \in \mathbb R^3 \setminus \Gamma,\, \forall \varphi \in H^{\frac{1}{2}}(\Gamma). \label{double:layer:potential}
\end{align} 
These two potentials satisfy the Helmholtz equation with wavenumber~$k$ in~$\Omega \cup \Omegap$.
We also define the Newton potential by 
\begin{equation} \label{Newton:potential}
\Ntildek g (\xbold) := \int_\Omega G_k(\xbold - \ybold) g(\ybold)
d\ybold \quad \forall \xbold \in {\mathbb R}^3,\, \forall g \in L^2(\Omega).
\end{equation}
Starting from~\eqref{single:layer:potential} and~\eqref{double:layer:potential}, we define the four boundary integral operators. 
The properties of such operators are detailed in several textbooks and papers; we refer here for instance to~\cite{mclean2000strongly, SauterSchwab_BEMbook, steinbach_BEMbook, costabel}.
\medskip

Firstly, we introduce the single layer operator $\Vk : H^{-\frac{1}{2}}(\Gamma) \rightarrow H^{\frac{1}{2}}(\Gamma)$:
\begin{equation} \label{single layer operator}
\Vk \varphi := \gammazint (\Vtildek \varphi) \quad \forall \varphi \in H^{-\frac{1}{2}}(\Gamma).
\end{equation}
This operator extends to an operator~$\Vk : H^{-1+s}(\Gamma) \rightarrow H^{s}(\Gamma)$ for all~$s\in[0,1]$ on Lipschitz boundaries~$\Gamma$.
Besides, $\Vtildek$ satisfies the jump relation
\begin{equation}
\label{eq:jump-rel-V-1}
\gammazext(\Vtildek \varphi) = \gammazint(\Vtildek \varphi) \quad \forall \varphi \in H^{-\frac{1}{2}}(\Gamma) \quad \text{and} \quad \llbracket \Vtildek \varphi \rrbracket_\Gamma =0 \quad \forall \varphi \in H^{-\frac{1}{2}}(\Gamma).
\end{equation}
Next, we define the double layer operator $\Kk: H^{\frac{1}{2}}(\Gamma) \rightarrow H^{\frac{1}{2}}(\Gamma)$:
\begin{equation} \label{double:layer:operator}
\left( -\frac{1}{2} + \Kk  \right) \varphi := \gammazint ( \Ktildek \varphi) \quad \forall \varphi \in H^{\frac{1}{2}}(\Gamma).
\end{equation}
This operator extends to an operator~$\Kk: H^s(\Gamma) \rightarrow H^s(\Gamma)$ for all~$s \in [0,1]$ on Lipschitz boundaries~$\Gamma$.
Moreover, the following jump condition holds true:
\begin{equation}
\label{eq:jump-rel-K-1}
\llbracket \Ktildek \varphi \rrbracket_\Gamma = -\varphi \quad \forall \varphi \in H^{\frac{1}{2}}(\Gamma) \quad \text{and} \quad \gammazext (\Ktildek \varphi) = \left( \frac{1}{2} + \Kk  \right) \varphi \quad \forall \varphi \in H^{\frac{1}{2}}(\Gamma).
\end{equation}
The so-called adjoint double layer operator $\Kprimek: H^{-\frac{1}{2}}(\Gamma) \rightarrow H^{-\frac{1}{2}}(\Gamma)$ is specified by
\begin{equation} \label{adjoint:double:layer:operator}
\left( \frac{1}{2} + \Kprimek  \right) \varphi := \gammaoint ( \Vtildek \varphi) \quad \forall \varphi \in H^{-\frac{1}{2}}(\Gamma).
\end{equation}
This operator extends to an operator~$\Kprimek: H^{-s}(\Gamma) \rightarrow H^{-s}(\Gamma)$ for all~$s \in [0,1]$ on Lipschitz boundaries~$\Gamma$.
Moreover, the following jump condition holds true
\begin{equation}
\label{eq:jump-rel-V-2}
\llbracket \partial_{\nGamma} \Vtildek \varphi \rrbracket_\Gamma = \varphi \quad \forall \varphi \in H^{-\frac{1}{2}}(\Gamma) \quad \text{and} \quad \gammaoext (\Vtildek \varphi) = \left( -\frac{1}{2} + \Kprimek  \right) \varphi \quad \forall \varphi \in H^{-\frac{1}{2}}(\Gamma).
\end{equation}
The hypersingular boundary integral operator $\Wk: H^{\frac{1}{2}}(\Gamma) \rightarrow H^{-\frac{1}{2}}(\Gamma)$ is given by
\begin{equation} \label{hypersingular:operator}
-\Wk \varphi := \gammaoint(\Ktildek \varphi) \quad \forall \varphi \in H^{\frac{1}{2}}(\Gamma).
\end{equation}
This operator extends to an operator~$\Wk: H^{s}(\Gamma) \rightarrow H^{-1+s}(\Gamma)$ for all~$s \in [0,1]$ on Lipschitz boundaries~$\Gamma$.
The following jump condition holds true:
\begin{equation}
\label{eq:jump-rel-K-2}
\llbracket \partial_{\nGamma} \Ktildek \varphi \rrbracket_\Gamma = 0 \quad \forall \varphi \in H^{\frac{1}{2}}(\Gamma) \quad \text{and} \quad \gammaoext (\Ktildek \varphi) = -\Wk \varphi \quad \forall \varphi \in H^{-\frac{1}{2}}(\Gamma).
\end{equation}
Additionally, we highlight a feature of the single layer and the hypersingular operators for zero wavenumber, namely, 
there exist constants
$c_1,c_2,c_3,c_4>0$
depending on $\Omega$ such that 
\begin{equation}
\label{coercivity:operators:Laplace}
\begin{split}
c_1 \|\varphi\|^2_{-\frac{1}{2},\Gamma} &\leq \langle \varphi , \Vz \varphi \rangle \leq c_2 \Vert \varphi \Vert_{-\frac{1}{2}, \Gamma}^2 
\quad \forall \varphi \in H^{-\frac{1}{2}} (\Gamma), \\
c_3|\psi|^2_{\frac{1}{2},\Gamma} & \leq 
\langle \Wz \psi , \psi \rangle \leq
c_4\vert \psi \vert_{\frac{1}{2}, \Gamma}^2 \quad \forall \psi
\in H^{\frac{1}{2}} (\Gamma)/{\mathbb C}.
\end{split}
\end{equation}
Moreover, 
\begin{equation}
\label{eq:adjoint-laplace}
\Vz^* = \Vz, 
\qquad \Kz^* = \Kz^\prime, 
\qquad \Wz^* = \Wz.
\end{equation}
Assuming sufficient smoothness of the interface~$\Gamma$, the
  range of the parameter~$s$ in the mapping properties of the four
  boundary operators that are recalled after formulas~\eqref{single layer operator},
  \eqref{double:layer:operator},
  \eqref{adjoint:double:layer:operator},
  and~\eqref{hypersingular:operator}, can be arbitrarily enlarged.

\begin{prop} \label{proposition:properties:compact:operators}
Let $\Gamma$ be $C^\infty$ and let~$\Vk$, $\Kk$, $\Kprimek$, and~$\Wk$ be defined in~\eqref{single
  layer operator}, \eqref{double:layer:operator},
\eqref{adjoint:double:layer:operator},
and~\eqref{hypersingular:operator}, respectively.
Then, for all~$s \in \mathbb R$, the following maps are bounded linear operators: 
\begin{equation} \label{eq:smoothGamma}
\begin{split}
& \Vk: H^{-1+s}(\Gamma) \rightarrow H^{s}(\Gamma), \quad \quad
\Kk: H^{s}(\Gamma) \rightarrow H^{s}( \Gamma),\\
& \Kprimek: H^{-s}(\Gamma) \rightarrow H^{-s}(\Gamma), \quad \quad\quad
\Wk: H^{s}(\Gamma) \rightarrow H^{-1+s}( \Gamma).\\
\end{split}
\end{equation}
Moreover, for $s \ge 0$, the operators
\begin{equation} \label{compact:part}
\begin{split}
& \Vk-\Vz: H^{-\frac{1}{2}+s}(\Gamma) \rightarrow H^{\frac{5}{2}+s} (\Gamma), \quad \quad \Kk - \Kz: H^{\frac{1}{2}+ s}(\Gamma) \rightarrow H^{\frac{5}{2} +s}( \Gamma),\\
& \Kprimek-\Kprimez: H^{-\frac{1}{2}+s}(\Gamma) \rightarrow H^{\frac{3}{2}+s} (\Gamma), \quad \quad \Wk - \Wz: H^{\frac{1}{2}+ s}(\Gamma) \rightarrow H^{\frac{3}{2} +s}( \Gamma),
\end{split}
\end{equation}
are bounded linear operators.
\end{prop}
\begin{proof}
For the proof of~\eqref{eq:smoothGamma}, we refer to \cite[Thm.~{7.2}]{mclean2000strongly}.
The enhanced shift properties of the differences in~\eqref{compact:part} as compared to the individual terms expresses a compactness property,
which is well-known; see, e.g., 
\cite{steinbach_BEMbook, SauterSchwab_BEMbook, hiptmair-meury06, BuffaHiptmair2005regularized}.
For analytic $\Gamma$, the mapping properties of~\eqref{compact:part} 
could be extracted from 
the potential estimates in
\cite[Thms.~{5.3}--{5.4}]{melenk2012mapping} by taking appropriate
traces; cf.~Appendix~\ref{appendix:garding}. 
In the interest of readability and to be able to connect with
Remark~\ref{remark:nonsmooth-interface} below, 
we sketch the argument.  For $\varphi \in H^{-1/2+s}(\Gamma)$, the function 
$u:= \Vtildek \varphi - \widetilde{\mathcal V}_0 \varphi$ satisfies 
\begin{equation}
\label{eq:proof-mapping-properties}
-\Delta u - k^2 u = - k^2 \widetilde {\mathcal V}_0 \varphi \quad \mbox{ in ${\mathbb R}^3\setminus \Gamma$}, 
\qquad \llbracket{u}\rrbracket_\Gamma = 0, \quad 
\qquad \llbracket{\partial_{\nGamma} u} \rrbracket_\Gamma = 0. \quad 
\end{equation}
We note that $\widetilde{\mathcal V}_0  \varphi \in H^{1+s}(\Omega) \cap H^{1+s}(\Omega^+\cap B_R(0))$ (for a sufficiently large ball $B_R(0)$) 
by the mapping properties of $\widetilde{\mathcal V}_0$, \cite[Thm.~{6.13}]{mclean2000strongly}. Since the interface is smooth, elliptic 
regularity for transmission problems gives $u \in H^{s+3}(\Omega) \cap H^{s+3}(\Omega^+ \cap B_R(0))$. Taking the trace and the conormal
derivative on $\Gamma$ proves the mapping properties for $\Vk - \Vz$ and $\Kprimek-\Kprimez$. The mapping properties
of $\Kk - \Kz$ and $\Wk-\Wz$ are seen similarly by studying the potential 
$u:= \Ktildek \psi - \widetilde{\mathcal K}_0  \psi$ for $\psi \in H^{1/2+s}(\Gamma)$.
\end{proof}
 
\begin{remark} \label{remark:dependence:constant:on:k}
The continuity constants of the mappings in~\eqref{eq:smoothGamma} and~\eqref{compact:part} depend on the wavenumber~$k$.
Some wavenumber-explicit control of the constants in~\eqref{compact:part} 
is possible using the refined $k$-explicit regularity given in \cite[Thms.~{5.1}--{5.4}]{melenk2012mapping}; cf.~also Appendix~\ref{appendix:garding}.
While this kind of regularity is a key ingredient of a $k$-explicit analysis of high order method, a sharp $k$-explicit analysis 
as done for acoustic scattering problems in \cite{loehndorf-melenk11} 
or for Maxwell problems in \cite{melenk-sauter20} is beyond the scope of the present work as it requires 
a much more elaborate analysis of various dual problems than what is done in Section~\ref{subsection:regularity:dual:problem}. 

Some $k$-explicit bounds for integral operators are available in the literature: 
besides \cite{melenk2012mapping}, we mention the bounds for the operators $\Vk$, $\Kk$, $\Kprimek$, $\Wk$ 
in \cite[Sec.~{1.2.3}]{graham-loehndorf-melenk-spence15}, \cite[Thm.~{6.4}]{GalkowskiPhDthesis} 
and \cite[Sec.~5]{actaBEMhelmholtz}, \cite{baskin-spence-wunsch16}
for the combined field operators.
\eremk
\end{remark}
\medskip

Let~$u$ be a (near $\Gamma$ sufficiently regular) solution to the 
Helmholtz equation $-\Delta u - k^2 u = 0$ in $\Omega^+$
satisfying the 
Sommerfeld radiation condition at infinity. Then, we have the following 
representation~\cite{mclean2000strongly,SauterSchwab_BEMbook}:
\[
u(\xbold) = -\Vtildek \partial_\n u (\xbold) + \Ktildek u (\xbold) \quad \forall \xbold \in \Omegap.
\]
Taking the trace and the trace of the normal derivative yields the following two equations, known
as the 
exterior \Calderon\,system for homogeneous Helmholtz problems with solution~$u$:
\begin{equation} \label{exterior:Calderon:system:Helmholtz}
\begin{cases}
\gammazext u =  \left( \frac{1}{2} + \Kk  \right) (\gammazext u) - \Vk (\gammaoext u) \\
\gammaoext u = -\Wk (\gammazext u) + \left( \frac{1}{2} - \Kprimek  \right) (\gammaoext u).\\
\end{cases}
\end{equation}

\section{Mortar coupling and analysis} \label{section:continuous:mortar:coupling}
The aim of the present section is to rewrite the transmission problem~\eqref{Helmholtz:transmission:problem} by adding an additional ``intermediate'' unknown called \emph{mortar variable}, which is the impedance trace of the solution on the interface~$\Gamma$.
This new formulation is analyzed in the remainder of the section, and is the target of the numerical analysis of Section~\ref{section:FEM-BEM:mortar}.

The problem with mortar coupling we are interested in reads as follows: Find the functions $u: \Omega\rightarrow \mathbb C$, $\m: \Gamma\rightarrow \mathbb C$, and $\uext: \Gamma\rightarrow \mathbb C$ such that 
\begin{align} 
& \begin{cases}
-\div(\Ad \nabla  u) - (\k)^2 u = \f & \text{in } \Omega,\\
\partial _{\nGamma} u + i k  u - \m =0 & \text{on } \Gamma,\\
\end{cases} \label{first:system:Helmholtz}\\
& \begin{cases}
\uext = \PItD \m & \text{on } \Gamma
\end{cases},\label{second:system:Helmholtz}\\
& \begin{cases}
u - \left[ \left( \frac{1}{2} + \Kk  \right) \uext - \Vk (\m - i k \uext)   \right] = 0 & \text{on } \Gamma.
\end{cases} \label{third:system:Helmholtz}
\end{align}
In the boundary conditions in problem~\eqref{first:system:Helmholtz}, the term~$\partial_{\nGamma}$ is indeed equal to~$\Ad\, \partial_{\nGamma}$, as we have assumed that~$\Ad$ is equal to~$1$ in a neighborhood of the interface~$\Gamma$.

The operator $\PItD: H^{-\frac{1}{2}}(\Gamma) \rightarrow H^{\frac{1}{2}}(\Gamma)$ appearing in~\eqref{second:system:Helmholtz} maps the impedance mortar variable~$\m$ to the Dirichlet trace~$\uext$ of the solution to the exterior problem.
A description of such operator can be found in~\cite[pag. 124--126]{actaBEMhelmholtz}. 

Equation~\eqref{first:system:Helmholtz} represents the problem in the interior domain~$\Omega$, i.e., a Helmholtz problem endowed with impedance boundary condition provided by the mortar variable.
On the other hand, equation~\eqref{second:system:Helmholtz} relates the Dirichlet trace of the solution in the exterior domain~$\Omegap$ with the mortar variable.
Finally, equation~\eqref{third:system:Helmholtz} couples the three unknowns altogether, i.e., connects the mortar variable~$\m$ with the Dirichlet traces of~$u$ and~$\uext$, the solutions in the interior and exterior domains, respectively.

Equations~\eqref{second:system:Helmholtz} and~\eqref{third:system:Helmholtz} follow from the exterior \Calderon\,system~\eqref{exterior:Calderon:system:Helmholtz}
and the coupling condition~$\llbracket u \rrbracket_\Gamma =0$ in~\eqref{Helmholtz:transmission:problem}.

\begin{remark} \label{remark:transmission:continuous}
The mortar formulation~\eqref{first:system:Helmholtz}-\eqref{second:system:Helmholtz}-\eqref{third:system:Helmholtz} is equivalent to a transmission formulation.
This can be easily seen by inverting the operator~$\PItD$ formally, solving~\eqref{second:system:Helmholtz} in terms of~$\m$ as
\[
\m = \PItD^{-1} \uext,
\]
and inserting~$\m$ in~\eqref{first:system:Helmholtz} and~\eqref{third:system:Helmholtz}.
The counterpart of this in the discrete setting is discussed in Remark~\ref{remark:transmission:discrete}.
\eremk
\end{remark}

For ease of reading, we recall and prove the following equivalent formulation of~\eqref{second:system:Helmholtz}.
\begin{prop} \label{proposition:CFIE}
Equation~\eqref{second:system:Helmholtz} is equivalent to
\begin{equation} \label{formula:Chandler-Wilde}
\Bk  \uext + i k \Aprimek ( \uext) -   \Aprimek \m = 0,
\end{equation}
where
\begin{equation} \label{definition:Bk:and:Akprime}
\Bk := -\Wk - i k \left( \frac{1}{2} - \Kk \right),\quad \Aprimek := \frac{1}{2} + \Kprimek + i k \Vk,
\end{equation}
are combined integral operators.
\end{prop}
\begin{proof}
The exterior \Calderon\,system~\eqref{exterior:Calderon:system:Helmholtz} also reads
\[
\begin{cases}
(\Kk-\frac{1}{2}) (\gammazext u) - \Vk (\gammaoext u)=0\\
\Wk (\gammazext u) + (\Kprimek + \frac{1}{2}) (\gammaoext u) =0.\\
\end{cases}
\]
Owing to the second equation of~\eqref{first:system:Helmholtz} and to the transmission conditions in~\eqref{Helmholtz:transmission:problem}, we get
\begin{equation} \label{rewriting:exterior:Calderon:Helmholtz:part:2}
\begin{cases}
(\Kk-\frac{1}{2}) (\uext) + i k \Vk (\uext) - \Vk (\m) = 0\\
- \Wk (\uext) + i k (\Kprimek + \frac{1}{2}) ( \uext) - (\Kprimek + \frac{1}{2})\m = 0.\\
\end{cases}
\end{equation}
Multiplying the first equation in \eqref{rewriting:exterior:Calderon:Helmholtz:part:2} by~$i k$, we deduce
\begin{equation} \label{rewriting:exterior:Calderon:Helmholtz:part:3}
\begin{cases}
-i k (\frac{1}{2}- \Kk) (\uext) + i  (i k \Vk) (k \uext) -  i k \Vk(\m) = 0 \\
- \Wk (\uext) + i k (\Kprimek + \frac{1}{2}) ( \uext) - (\Kprimek + \frac{1}{2})\m = 0.\\
\end{cases}
\end{equation}
We obtain \eqref{formula:Chandler-Wilde} by summing the two equations in \eqref{rewriting:exterior:Calderon:Helmholtz:part:3}, whence the name \emph{combined} for the operators in~\eqref{definition:Bk:and:Akprime}.
To conclude the proof, it suffices to show that~\eqref{formula:Chandler-Wilde} implies~\eqref{second:system:Helmholtz}.
To this aim, we use that the operator~$\Bk + i \Aprimek$ is invertible; see~\cite[Theorem~{2.27}]{actaBEMhelmholtz}.
\end{proof}

\begin{remark}
\label{remk:Aprimek-invertible}
The operators $\Bk$, $\Aprimek$, $\Bk+i \Aprimek$ are invertible
by \cite[Thm.~{2.27}]{actaBEMhelmholtz}. Wavenumber-explicit estimates
for the operator $\Aprimek$ and the exterior Dirichlet-to-Neumann operators are available in \cite{baskin-spence-wunsch16} 
for so-called nontrapping domains $\Omega$. 
The use of so-called combined field integral equations that are well-posed for all wavenumbers~$k$ goes back 
at least to \cite{BurtonMiller, BrakhageWerner}; we refer to~\cite{actaBEMhelmholtz} for a more detailed discussion.
\eremk
\end{remark}

Next, having at our disposal Proposition~\ref{proposition:CFIE}, we write the mortar coupling of the interior and exterior Helmholtz problems in weak form, which reads
\begin{equation} \label{weak:formulation:mortar:Helmholtz}
\begin{cases}
\text{find } (u,\m, \uext) \in H^1(\Omega) \times H^{-\frac{1}{2}}(\Gamma) \times H^{\frac{1}{2}}(\Gamma) \text{ such that}\\
(\Ad \nabla u, \nabla v)_{0,\Omega} - ((\k)^2 \,u, v)_{0,\Omega} + i k( u,v)_{0,\Gamma} - \langle \m, v \rangle = (\f, v)_{0,\Omega}\quad \forall v \in H^1(\Omega)\\
\langle  (\Bk + i k  \Aprimek)\uext  - \Aprimek \m, \vtilde  \rangle = 0 \quad \forall \vtilde \in H^{\frac{1}{2}}(\Gamma)\\
\langle u,\lambda \rangle - \langle  (\frac{1}{2} + \Kk) \uext - \Vk (\m - i k\uext), \lambda \rangle =  0 \quad \forall \lambda \in H^{-\frac{1}{2}}(\Gamma).
\end{cases}
\end{equation}
Problem~\eqref{weak:formulation:mortar:Helmholtz} is equivalent to the following problem:
\begin{equation} \label{equivalent:weak:formulation}
\begin{cases}
\text{find } (u, \m, \uext) \in H^1(\Omega) \times H^{-\frac{1}{2}}(\Gamma) \times H^{\frac{1}{2}}(\Gamma)\text{ such that}\\
\T( (u,\m, \uext) , (v,\lambda,\vtilde) ) = (\f, v)_{0,\Omega} \quad \forall (v,\lambda,\vtilde) \in H^1(\Omega) \times H^{-\frac{1}{2}}(\Gamma) \times H^{\frac{1}{2}}(\Gamma),
\end{cases}
\end{equation}
where we have set, by a proper linear combination of the three equations in~\eqref{weak:formulation:mortar:Helmholtz},
\begin{equation} \label{form:T:for:Helmholtz}
\begin{split}
\T( (u, \m, \uext) , 	&(v,\lambda,\vtilde) ) = (\Ad \nabla u, \nabla v )_{0,\Omega} -  ( (\k)^2u, v)_{0,\Omega} + i k( u, v)_{0,\Gamma} - \langle \m, v \rangle \\
				& - \langle  (-\Wk- i k ( \frac{1}{2} - \Kk  )  + i k   ( \frac{1}{2} +\Kprimek + i k \Vk ) )\uext 
-  ( \frac{1}{2} +\Kprimek + i k \Vk ) \m, \vtilde   \rangle  
\\&     +  \langle u,\lambda \rangle -  \langle (\frac{1}{2} + \Kk) \uext - \Vk (\m - i k \uext), \lambda  \rangle .\\
\end{split}
\end{equation}

\begin{remark}
\label{remk:three-field} 
For the special case $k = 0$, the sesquilinear form $\T$ reduces
to the one that is used in the ``three-field'' coupling strategy
for Poisson problems; see, e.g., \cite[Def.~{4.2}]{erath12}.
A direct calculation using~\eqref{eq:adjoint-laplace} shows that, for $k= 0$, one has 
$\T\bigl( (u,m,\uext),(u,m,\uext) \bigr)
 = \|\Ad^{\frac{1}{2}} \nabla u \|^2_{L^2(\Omega)} 
   + \langle \Wz \uext,\uext\rangle  + \langle \Vz \m,\m\rangle$, which is nonnegative. \eremk
\end{remark}
\medskip

In Theorem~\ref{theorem:Garding:inequality} below, we prove that $\T(\cdot, \cdot)$ satisfies a G{\aa}rding inequality, which then allows us to prove the following existence and uniqueness result:
\begin{thm} \label{theorem:well-posedness:of:weak:mortar:formulation}
Assuming \eqref{assumption:uniqueness} and that the interface~$\Gamma$ is smooth, problem~\eqref{weak:formulation:mortar:Helmholtz} and, consequently, problem~\eqref{equivalent:weak:formulation}, admit a unique solution.
In particular, the sesquilinear form $\T(\cdot, \cdot)$ satisfies a
positive inf-sup condition with a wavenumber-dependent constant.
\end{thm}
\begin{proof}
By Fredholm theory, the assertion follows from a combination of Theorem~\ref{theorem:Garding:inequality}, i.e., the G{\aa}rding inequality, and the uniqueness assumption~\eqref{assumption:uniqueness}.
\end{proof}

\subsection{G{\aa}rding inequality} \label{subsection:Garding:inequality}
In this section, we prove a G{\aa}rding inequality for~$\T(\cdot, \cdot)$ defined in~\eqref{form:T:for:Helmholtz}.
Such an inequality is one of the two lynchpins for the proof of Theorem~\ref{theorem:well-posedness:of:weak:mortar:formulation}, 
as well as for the analysis of the convergence of the method carried
out in Section~\ref{section:FEM-BEM:mortar}.
\begin{thm} [G{\aa}rding inequality]
\label{theorem:Garding:inequality}
Let~$\T(\cdot, \cdot)$ be defined as in~\eqref{form:T:for:Helmholtz},
and assume that the interface~$\Gamma$ is smooth. Fix $k_0 > 0$. 
There there is $c > 0$ (depending only on $k_0$ and $\Omega$) and,
for each $k \ge k_0$, there is a positive constant~$\cG(k)$ depending
on~$k$ and $\Omega$ such that
for all $(v,\lambda,\vext)\in H^1(\Omega) \times H^{-\frac{1}{2}}(\Gamma) \times H^{\frac{1}{2}} (\Gamma)$
\begin{equation} \label{final:Garding}
\begin{split}
\Re ( 	& \T( (v,\lambda,\vext),   (v,\lambda,\vext)   )  ) \\
& \ge c \left\{ 
\| \Ad^{\frac{1}{2}} \nabla v \|^2_{0,\Omega}
+ \Vert \lambda \Vert^2_{-\frac{1}{2},\Gamma}
+ \Vert \vext
\Vert^2_{\frac{1}{2},\Gamma} 
\right\}  
- \left\{k^2 \Vert n\,v \Vert^2_{0,\Omega} 
+\cG(k) \left(\Vert \lambda \Vert_{-\frac{5}{2}, \Gamma}^2  + \Vert \vext \Vert_{-\frac{3}{2}, \Gamma}^2\right) \right\}.
\end{split}
\end{equation}
\end{thm}
\begin{proof}
  The estimate~\eqref{final:Garding} is not $k$-explicit
  since the constant $\cG(k)$ 
depends on $k$. 
Notwithstanding, we highlight all the occurrences where the estimates
depend on the
wavenumber~$k$.

Given the positivity results for $k=0$
  in~\eqref{coercivity:operators:Laplace}, see also Remark~\ref{remk:three-field}, 
we write 
$\Vk = \Vz + (\Vk-\Vz)$, 
$\Kk = \Kz + (\Kk-\Kz)$, 
$\Kprimek = \Kz^\prime + (\Kprimek-\Kz^\prime)$, 
$\Wk = \Wz + (\Wk-\Wz)$, and split the sesquilinear form $\T(\cdot,\cdot)$ accordingly: 
\begin{align}
\nonumber 
 \T( (&u, \m, \uext) , 	(v,\lambda,\vtilde) )  \\
\nonumber 
=& (\Ad \nabla u, \nabla v )_{0,\Omega} -  ( (\k)^2u, v)_{0,\Omega} + i k( u, v)_{0,\Gamma} - \langle \m, v \rangle 
			\\
\nonumber 
& 
- \langle  (-\Wk- i k ( \frac{1}{2} - \Kk  )  + i k   ( \frac{1}{2} +\Kprimek + i k \Vk ) )\uext 
				 -  ( \frac{1}{2} +\Kprimek + i k \Vk ) \m, \vtilde   \rangle  \\
\nonumber 
&     +  \langle u,\lambda \rangle 
-  \langle (\frac{1}{2} + \Kk) \uext - \Vk (\m - i k \uext),\lambda \rangle  \\
\nonumber 
=& 
\Bigl\{ (\Ad \nabla u, \nabla v )_{0,\Omega} 
+ \langle \Wz\uext,\vext\rangle 
+ k^2 \langle \Vz \uext,\vext\rangle  
+ \langle \Vz m, \lambda \rangle 
\Bigr\} 
\\ 
\nonumber 
&
+ \Bigl\{ i k ( u, v)_{0,\Gamma} 
- i k \langle \Kz \uext,\vext\rangle 
- i k \langle \Kz^\prime \uext,\vext\rangle 
+\langle (\frac{1}{2}+\Kz^\prime) m,\vext\rangle 
\\ 
\nonumber 
& 
+ i k \langle \Vz m,\vext\rangle 
- ik\langle \Vz \uext, \lambda\rangle 
- \langle (\frac{1}{2} + \Kz) \uext, \lambda\rangle 
- \langle \m, v \rangle 
 +  \langle u,\lambda \rangle 
\Bigr\} 
\\ 
\nonumber 
&+ 
\Bigl\{ 
\langle (\Wk-\Wz) \uext,\vext\rangle 
- ik \langle   (  \Kk- \Kz  )  \uext,\vext\rangle 
- ik \langle    (\Kprimek -\Kz^\prime)\uext,\vext\rangle  
\\ 
\nonumber 
& 
+ k^2 \langle (\Vk -\Vz) \uext ,\vext \rangle 
+ \langle (\Kprimek -\Kz^\prime) m,\vext\rangle 
+ i k \langle  (\Vk -\Vz) \m, \vtilde   \rangle   
\\ 
\nonumber 
& 
-  \langle (\Kk-\Kz) \uext ,\lambda \rangle 
+ \langle (\Vk - \Vz)\m, \lambda\rangle
     - ik \langle   (\Vk - \Vz) \uext,\lambda \rangle
     -  ( (\k)^2 u, v)_{0,\Omega} 
\Bigr\}  \\
\label{eq:T1T2T3}
 =:& T_1(\cdot,\cdot) + T_2(\cdot,\cdot) + T_3(\cdot,\cdot), 
\end{align}
where the sesquilinear forms $T_1$, $T_2$, $T_3$ correspond
to the three expressions in $\{\cdots\}$. 

\emph{1.~step:} We show that
\begin{equation}\label{eq:GardingStep1}
  \Re\left(T_1\bigl((v,\lambda,\vext),(v,\lambda,\vext)\bigr)\right)=
  T_1\bigl((v,\lambda,\vext),(v,\lambda,\vext)\bigr)
\gtrsim \| \Ad^{\frac{1}{2}} \nabla v \|^2_{0,\Omega}
+ \Vert \lambda \Vert^2_{-\frac{1}{2},\Gamma} + \Vert \vext \Vert^2_{\frac{1}{2},\Gamma},
\end{equation}
with implied constant dependent on $k_0$ but independent of $k$. In view of the 
positivity assertions for $\Vz$ and $\Wz$
(cf.~\eqref{coercivity:operators:Laplace}), in order to
prove~\eqref{eq:GardingStep1}, it suffices to ascertain
that 
  \begin{equation}\label{eq:seminorm}
|\vext|^2_{\frac{1}{2},\Gamma} +k^2 \|\vext\|^2_{-\frac{1}{2},\Gamma}
\gtrsim
\|\vext\|^2_{\frac{1}{2},\Gamma},
\end{equation}
with implied constant 
dependent on $k_0$ but independent of $k$.
Let $\overline v:= |\Gamma|^{-1} (\vext,1)_{0,\Gamma}$ be the~$L^2(\Gamma)$-projection of~$\vext$ onto~${\mathbb C}$. 
We note that $|\vext|^2_{\frac{1}{2},\Gamma} \sim \inf_{c\in\mathbb
    C}\|\vext-c\|^2_{\frac{1}{2},\Gamma}=\|\vext-\overline v\|^2_{\frac{1}{2},\Gamma}$. Next, 
$ |\Gamma| |\overline{v}|  = |(\vext,1)_{0,\Gamma} | 
 = |\langle \vext,1\rangle| \leq \|\vext\|_{-\frac{1}{2},\Gamma} 
 \| 1\|_{\frac{1}{2},\Gamma}  \lesssim 
 \|\vext\|_{-\frac{1}{2},\Gamma}$. 
Hence, 
  \[
    \|\vext\|^2_{\frac{1}{2},\Gamma}\le \|\vext-\overline v\|^2_{\frac{1}{2},\Gamma}+\|\overline v\|^2_{\frac{1}{2},\Gamma}=
    \|\vext-\overline v\|^2_{\frac{1}{2},\Gamma}+|\overline
    v|^2 \|1\|_{\frac{1}{2},\Gamma}  \lesssim |\vext|^2_{\frac{1}{2},\Gamma} 
+ \|\vext\|^2_{-\frac{1}{2},\Gamma}.
    \]
Since $k\ge k_0>0$, we get $|\vext|^2_{\frac{1}{2},\Gamma} +
\|\vext\|^2_{-\frac{1}{2},\Gamma}
\lesssim|\vext|^2_{\frac{1}{2},\Gamma} +k^2
  \|\vext\|^2_{-\frac{1}{2},\Gamma}$
  (with hidden constant depending on
  $k_0$ but independent of $k$), and~\eqref{eq:seminorm} follows.

\emph{2.~step:} Using $\Re(\overline z) = \Re( z)$ for all $z \in \mathbb C$
and~\eqref{eq:adjoint-laplace},
we see that 
\begin{equation}\label{eq:GardingStep2}
\Re \left(T_2\bigl( (v,\lambda,\vext),(v,\lambda,\vext)  \bigr)\right) = 0. 
\end{equation}

\emph{3.~step:} Using the continuity of the operators in~\eqref{compact:part}, we bound each of the first nine terms in $T_3\bigl(v,\lambda,\vext),(v,\lambda,\vext)\bigr)$ as follows:  
  \begin{equation}
\label{eq:garding-foobar}
\begin{split}
|\langle \lambda, (\Vk - \Vz)\lambda\rangle| & \leq \|\lambda\|_{-\frac{5}{2},\Gamma} \|(\Vk - \Vz) \lambda\|_{\frac{5}{2},\Gamma} 
\lesssim
\|\lambda\|_{-\frac{5}{2},\Gamma} \|\lambda\|_{-\frac{1}{2} ,\Gamma}, \\[0.2cm] 
|\langle \vext, (\Wk - \Wz)\vext\rangle| & \leq \|\vext\|_{-\frac{3}{2},\Gamma} \|(\Wk - \Wz) \vext\|_{\frac{3}{2},\Gamma} 
\lesssim
\|\vext\|_{-\frac{3}{2},\Gamma} \|\vext\|_{\frac{1}{2} ,\Gamma}, \\[0.2cm]  
|\langle (\Vk - \Vz)\vext,\vext\rangle| & \leq \|(\Vk - \Vz) \vext\|_{\frac{7}{2},\Gamma} \|\vext\|_{-\frac{7}{2} ,\Gamma} 
\lesssim
\|\vext\|_{\frac{1}{2},\Gamma} \|\vext\|_{-\frac{7}{2} ,\Gamma}, \\[0.2cm]  
|\langle (\Kprimek - \Kprimez)\lambda,\vext\rangle| & \leq \|(\Kprimek - \Kprimez) \lambda\|_{\frac{3}{2},\Gamma} \|\vext\|_{-\frac{3}{2} ,\Gamma} 
\lesssim
\|\lambda\|_{-\frac{1}{2},\Gamma} \|\vext\|_{-\frac{3}{2} ,\Gamma}, \\[0.2cm]  
|\langle (\lambda,(\Kk - \Kz)\vext\rangle| & \leq \|\lambda\|_{-\frac{5}{2},\Gamma} \|(\Kk - \Kz) \vext\|_{\frac{5}{2} ,\Gamma} 
\lesssim
\|\lambda\|_{-\frac{5}{2},\Gamma} \|\vext\|_{\frac{1}{2} ,\Gamma}, \\[0.2cm]  
|\langle ((\Vk - \Vz)\lambda, \vext\rangle| & \leq \|(\Vk -\Vz)\lambda\|_{\frac{5}{2},\Gamma} \|\vext\|_{-\frac{5}{2} ,\Gamma} 
\lesssim
\|\lambda\|_{-\frac{1}{2},\Gamma} \|\vext\|_{-\frac{5}{2} ,\Gamma}, \\[0.2cm]  
|\langle \lambda, (\Vk - \Vz) \vext\rangle| & \leq \|\lambda\|_{-\frac{7}{2},\Gamma} \|(\Vk - \Vz) \vext\|_{\frac{7}{2} ,\Gamma} 
\lesssim
\|\lambda\|_{-\frac{7}{2},\Gamma} \|\vext\|_{\frac{1}{2} ,\Gamma}, \\[0.2cm]  
|\langle (\Kprimek - \Kprimez) \vext,\vext\rangle| & \leq \|(\Kprimek - \Kprimez) \vext\|_{\frac{5}{2} ,\Gamma} \|\vext\|_{-\frac{5}{2},\Gamma} 
\lesssim
\|\vext\|_{\frac{1}{2},\Gamma} \|\vext\|_{-\frac{5}{2} ,\Gamma}, \\[0.2cm]  
|\langle \vext,(\Kk - \Kz) \vext,\vext\rangle| & \leq \|\vext\|_{-\frac{5}{2},\Gamma} \|(\Kk - \Kz) \vext\|_{\frac{5}{2} ,\Gamma} 
\lesssim
\|\vext\|_{-\frac{5}{2},\Gamma} \|\vext\|_{\frac{1}{2} ,\Gamma},
\end{split}
\end{equation}                                                 
where the hidden constants depend on $k$. The last term is simply
  \begin{equation}\label{eq:GardingLast}
((kn)^2v,v)_{0,\Omega}=k^2\|n\,v\|_{0,\Omega}^2.
    \end{equation}
Collecting~\eqref{eq:GardingStep1}, \eqref{eq:GardingStep2}, \eqref{eq:GardingLast},
and~\eqref{eq:garding-foobar},
and using Young's
inequality show the claim.
\end{proof}

\begin{remark} \label{remark:nonsmooth-interface}
The G{\aa}rding inequality in Theorem~\ref{theorem:Garding:inequality} relies on the 
compactness properties of~$\Vk-\Vz$, $\Kk-\Kz$, $\Kprimek-\Kprimez$, and~$\Wk-\Wz$ 
as employed in~\eqref{eq:garding-foobar}. Such compactness properties are still valid for Lipschitz domains. 
For Lipschitz domains, the spaces $H^{s}(\Gamma)$, $|s| \leq 1$, can be defined using
local (Lipschitz) charts; see, e.g., \cite{mclean2000strongly}. We claim that, for $\varepsilon \in (0,1)$,
\begin{subequations}
\label{eq:mapping-lipschitz}
\begin{align}
\label{eq:mapping-lipschitz-a}
\Vk - \Vz \colon H^{-1/2}(\Gamma) & \rightarrow H^{1-\varepsilon}(\Gamma), 
& 
\Kprimek - \Kprimez \colon H^{-1/2}(\Gamma) & \rightarrow L^{2}(\Gamma), \\
\label{eq:mapping-lipschitz-b}
\Kk - \Kz \colon H^{1/2}(\Gamma) & \rightarrow H^{1-\varepsilon}(\Gamma), 
& 
\Wk - \Wz \colon H^{1/2}(\Gamma) & \rightarrow L^{2}(\Gamma). 
\end{align}
\end{subequations}
To see~\eqref{eq:mapping-lipschitz-a}, consider for 
$\varphi \in H^{-1/2}(\Gamma)$
the potential $u = \Vtildek \varphi - \widetilde \V_0 \varphi$. It is in $H^1(\Omega) \cap H^1(\Omega^+ \cap B_R(0))$ 
and satisfies~\eqref{eq:proof-mapping-properties}. We note that $k^2 \widetilde \V_0 \varphi \in H^1(\Omega) \cap H^1(\Omega^+ \cap B_R(0))$ 
by \cite[Thm.~{1}]{costabel}. Therefore, $\widetilde \V_0 \varphi$ is in $L^2(B_R(0))$, 
and then~\eqref{eq:proof-mapping-properties} implies $u \in H^2_{loc}({\mathbb R}^3)$. 
For Lipschitz domains, the trace operator is continuous $H^{s}(\Omega)  \rightarrow H^{s-1/2}(\Gamma)$
for $1/2 < s < 3/2$, \cite[Thm.~{3.38}]{mclean2000strongly}. This implies that 
$\gamma_0 u \in H^{1-\varepsilon}(\Gamma)$ for any $\varepsilon  \in (0,1]$, and this gives 
the first mapping property in~\eqref{eq:mapping-lipschitz-a}. The conormal derivative of $u$ on $\Gamma$ 
is $\partial_{\nGamma} u = \nabla u \cdot \nGamma $. 
Since $\nabla u \in H^{1}(\Omega)$ and $ \nGamma\in L^\infty(\Gamma)$, 
we infer $\partial_{\nGamma} u \in L^2(\Gamma)$. This gives the second mapping property in~\eqref{eq:mapping-lipschitz-a}. 
The mapping properties in~\eqref{eq:mapping-lipschitz-b} are shown by similar arguments using, 
for $\psi \in H^{1/2}(\Gamma)$,  the potential  $u = \Ktildek \psi - \widetilde \Kz \psi$, which 
is in $H^1(\Omega) \cap H^1(\Omega^+ \cap B_R(0))$ by \cite[Thm.~{1}]{costabel}. 

Inserting the mapping properties~\eqref{eq:mapping-lipschitz} in~\eqref{eq:garding-foobar} yields, for any chosen 
$\varepsilon  \in (0,1/2)$ and all $(v,\lambda,\vext) \in H^1(\Omega) \times H^{-1/2}(\Gamma) \times H^{1/2}(\Gamma)$ ,
the G{\aa}rding inequality 
\begin{equation*} 
\begin{split}
\Re \bigl( \T\bigl( (v,\lambda,\vext),   (v,\lambda,\vext)   \bigr) \bigr ) 
		 \gtrsim&  \Vert\Ad^{\frac12}\nabla v
                 \Vert^2_{0,\Omega} + \Vert \lambda
                 \Vert^2_{-\frac{1}{2},\Gamma} + \| \vext
                 \|^2_{\frac{1}{2},\Gamma} \\
                 &- \cG(k) \left( \Vert n\,v \Vert^2_{0,\Omega}  + \Vert \lambda \Vert_{-1+\varepsilon, \Gamma}^2  + \Vert \vext \Vert_{0, \Gamma}^2 \right).  
\end{split}
\end{equation*}
\eremk
\end{remark}
\medskip

The compact perturbation in the G{\aa}rding inequality of 
Theorem~\ref{theorem:Garding:inequality} essentially arises from the 
the differences 
$\Vk - \Vz$, $\Kk - \Kz$, $\Kprimek - \Kz'$, $\Wk - \Wz$. 
For analytic $\Gamma$, a very good description of these differences 
is provided in~\cite{melenk2012mapping}. In Appendix~\ref{appendix:garding}, based on that characterization of the difference operators
(see~\eqref{eq:differences}), a $k$-explicit G{\aa}rding inequality for analytic $\Gamma$ is proven (see Theorem~\ref{thm:k-explicit-garding}).

\subsection{Regularity of solutions to the dual problem of~\eqref{weak:formulation:mortar:Helmholtz}} \label{subsection:regularity:dual:problem}
In this section, we analyze a problem dual to problem~\eqref{weak:formulation:mortar:Helmholtz}, or equivalently to problem~\eqref{equivalent:weak:formulation}.
The regularity of the solution of this dual problem is crucial for the convergence analysis in Section~\ref{section:FEM-BEM:mortar}. 

The dual problem we are interested in reads
\begin{equation} \label{dual:problem:1}
\begin{cases}
\text{find } (\psi, \psim, \psitilde) \in H^1(\Omega)\times H^{-\frac{1}{2}}(\Gamma) \times H^{\frac{1}{2}}(\Gamma) \text{ such that} \\
\T ( (v, \lambda, \vtilde),  (\psi, \psim, \psitilde)) =\Fcal(v,\lambda,\vtilde) \quad \forall (v,\lambda,\vtilde) \in H^1(\Omega)\times H^{-\frac{1}{2}}(\Gamma) \times H^{\frac{1}{2}}(\Gamma),\\
\end{cases}
\end{equation}
where
\[
\Fcal (v,\lambda,\vtilde) = [(v,\rr)_{0,\Omega} +  (\lambda, \rm)_{-\sm,\Gamma} +   (\vtilde, \rext)_{-\sv,\Gamma} ]
\]
for a given $(\rr, \rm, \rext )\in L^2(\Omega) \times H^{-\sm}(\Gamma) \times H^{-\sv}(\Gamma)$ 
and $\sm = \frac{5}{2}$, $\sv = \frac{3}{2}$. 
We recall that the Sobolev inner products $(\cdot,\cdot)_{-\sigma,\Gamma}$ on $H^{-\sigma}(\Gamma)$ are defined in~\eqref{negative:Sobolev:sesquilinear:form}.

More explicitly, we consider: 
\begin{equation} \label{dual:problem:2}
\begin{cases}
\text{find } (\psi, \psim, \psitilde) \in H^1(\Omega)\times H^{-\frac{1}{2}}(\Gamma) \times H^{\frac{1}{2}}(\Gamma) \text{ such that} \\
(\Ad \nabla v, \nabla \psi )_{0,\Omega} -  ((\k)^2 v, \psi)_{0,\Omega} + i k ( v, \psi)_{0,\Gamma} - \langle \lambda, \psi \rangle \\
\quad\quad -\langle  ( \Bk + i k \Aprimek) \vtilde -  \Aprimek
\lambda, \psitilde   \rangle
+ \langle v,\psim\rangle
- \langle (\frac{1}{2} + \Kk) \vtilde - \Vk (\lambda - i k \vtilde) , \psim  \rangle \\
=  \left ((v,\rr)_{0,\Omega} +  (\lambda, \rm)_{-\sm,\Gamma} +  (\vtilde, \rext)_{-\sv,\Gamma} \right)\\
\quad \quad \quad\forall (v,\lambda,\vtilde) \in H^1(\Omega)\times H^{-\frac{1}{2}}(\Gamma) \times H^{\frac{1}{2}}(\Gamma).\\
\end{cases}
\end{equation}
\subsubsection{Riesz representations}
In the following, we need a few technical results.

\begin{lem} \label{lemma:on:adjoint:operators}
The following identities are valid: For all $\varphi\in H^{-\frac{1}{2}}(\Gamma)$ and for all~$\psi \in H^{\frac{1}{2}}(\Gamma)$
\begin{align}
& \Vk^* \varphi = \overline {\Vk \overline \varphi}, \quad \Kk^*\varphi = \overline{\Kk' \overline{\varphi}}, \label{adjoint:of:blo:1}\\
& (\Kk')^*\psi= \overline{\Kk \overline{\psi}}, \quad \Wk^* \psi= \overline {\Wk \overline \psi}, \label{adjoint:of:blo:2}
\end{align}
where we recall that $\cdot^*$ denotes the adjoint operator.
Moreover,
for all~$\varphi\in H^{-\frac{1}{2}}(\Gamma)$ and~$\psi \in H^{\frac{1}{2}}(\Gamma)$,
it holds true that
\begin{align} 
& \langle \varphi , (\Ak')^* \psi \rangle =  \langle \varphi , \overline{\frac{1}{2} \overline{\psi} + \Kk \overline{\psi}+i k \Vk \overline{\psi}} \rangle, \label{adjoint:Aprime}\\
  & \langle \varphi, ( \Bk + i k \Aprimek )^*\psi\rangle
    = -\langle \varphi, \overline {
    (\Wk + k^2 \Vk)
    \overline{\psi}}
   + i k	 \overline{(\Kprimek + \Kk)
    \overline{\psi}}
    \rangle.
   \label{adjoint:Aprime:2}
\end{align}
\end{lem}
\begin{proof}
The identities in~\eqref{adjoint:of:blo:1} are proven in~\cite[equation (2.38)]{actaBEMhelmholtz}. We limit ourselves to prove here the first identity in~\eqref{adjoint:of:blo:2}, since the second one can be dealt with similarly to~\cite{actaBEMhelmholtz}.
For all~$\varphi \in H^{-\frac{1}{2}}(\Gamma)$ and~$\psi
\in H^{\frac{1}{2}}(\Gamma)$, we have
\[
  \langle \varphi, (\Kprimek)^*\psi\rangle=
  \overline{\langle \overline{\Kprimek\varphi},
    \overline{\psi}\rangle}
  \overset{\eqref{adjoint:of:blo:1}}{=}
  \overline{\langle \Kk^*\overline{\varphi},
    \overline{\psi}\rangle}
  =\overline{\langle \overline{\varphi},
    \Kk\overline{\psi}\rangle}
=\langle \varphi , \overline{\Kk\overline{\psi}}\rangle.
\]
As far as the proof of~\eqref{adjoint:Aprime} is concerned, recalling the definition of~$\Aprimek$ from~\eqref{definition:Bk:and:Akprime}, we have
\[
\begin{split}
   \langle  \varphi, (\Aprimek)^*\psi \rangle
   & = \langle \varphi,  (\frac{1}{2} + \Kprimek + i k \Vk)^* \psi \rangle = \langle \varphi  , \frac{1}{2} \psi + (\Kprimek)^* \psi- i k \Vk^* \psi \rangle\\
& \overset{\eqref{adjoint:of:blo:1},\eqref{adjoint:of:blo:2}}{=}
\langle \varphi , \frac{1}{2} \overline{\overline{\psi}} + \overline{\Kk \overline{\psi}} - i k \overline{\Vk {\overline{\psi}}}  \rangle  =
\langle \varphi, \overline{\frac{1}{2}\psi + \Kk \overline{\psi} + i k \Vk \overline{\psi}  }  \rangle,
\end{split}
\]
whence follows the claim. The proof of~\eqref{adjoint:Aprime:2} is dealt with similarly.
\end{proof}
The second technical result we need reads as follows.
\begin{lem} [Riesz representation]
\label{lemma:Laplace:Beltrami}
Let $s \in {\mathbb R}^+_0$ and $\sm \ge 1/2$, $\sv \ge 0$. 
Given~$\rm \in H^{s-\sm}(\Gamma)$ and~$\rext \in H^{s-\sv}(\Gamma)$, there 
exist~$\Rm\in H^{s+\sm}(\Gamma)$ and~$\Rext \in H^{s+\sv} (\Gamma)$ such that
\[
\Vert \Rm \Vert_{s+\sm, \Gamma} = \Vert \rm \Vert_{s-\sm, \Gamma},\quad \Vert \Rext \Vert_{s+\sv, \Gamma} = \Vert \rext \Vert_{s-\sv, \Gamma},
\]
and
\begin{equation} \label{regularity:Melenk:trick}
\begin{split}
&\langle \mu, \Rm \rangle= (\mu, \rm)_{-\sm, \Gamma} \quad \forall \mu \in H^{-\frac{1}{2}}(\Gamma),\\
&( \vext , \Rext ) _{0,\Gamma}= (\vext,\rext)_{-\sv, \Gamma} \quad \forall \vext \in L^2(\Gamma).
\end{split}
\end{equation}
\end{lem}
\begin{proof}
We show the assertion for~$\Rm$ only, since the case of~$\Rext$ can be proven analogously.
We construct~$\Rm$ in terms of the eigenpairs $\{ \varphi_n, \lambda_n  \}_{n \in \mathbb N}$ of the Laplace-Beltrami operator explicitly.
Recall that we identify elements of positive order Sobolev spaces~$H^t(\Gamma)$ and negative order Sobolev spaces~$H^{-t}(\Gamma)$ with sequence spaces, where, for~$t \ge 0$, the identification is simply the~$L^2(\Gamma)$-orthogonal expansion~$u = \sum_n u_n \varphi_n$. 
It is notationally convenient to realize the isomorphisms between the  spaces~$H^t(\Gamma)$ and  sequence spaces by simply writing~$u = \sum_{n} u_n \varphi_n$. Recall the realization of the duality pairing~\eqref{eq:duality-pairing-Gamma}
and the realization~\eqref{negative:Sobolev:sesquilinear:form} of the inner products.

Given $\rm\in H^{s-\sm}(\Gamma)$ we write it as  
\[ 
\rm = \sum_{n\in \mathbb N} (\rm)_n \varphi_n
\]
and define $\Rm$ by 
\[
\Rm =\sum_{n\in \mathbb N} (\Rm)_n  \varphi_n := \sum_{n\in \mathbb N} (\rm)_n (1+\lambda_{n})^{-\sm} \varphi_n.
\]
We have
\[
\Vert \Rm \Vert^2_{s+\sm, \Gamma} 
\stackrel{\eqref{negative:Sobolev:sesquilinear:form}}{=} 
\sum_{n\in \mathbb N} \vert (\rm)_n\vert^2(1+\lambda_n)^{s-\sm}
\stackrel{\eqref{negative:Sobolev:sesquilinear:form}}{=} 
\Vert \rm \Vert^2_{s-\sm, \Gamma},
\]
which entails~$\Rm\in H^{s+\sm}(\Gamma)$.
Moreover, for any~$\mu \in H^{-\sm}(\Gamma)$, which we express as 
$\mu=\sum_{n\in \mathbb N} \mu_n \varphi_n$, we get from~\eqref{eq:duality-pairing-Gamma}
\[
\langle \mu, \Rm \rangle = \sum_{n\in \mathbb N} \mu_n \overline{(\rm)_n} (1+\lambda_n)^{-\sm} = (\mu, \rm)_{-\sm, \Gamma}.
\qedhere
\]
\end{proof}
\subsubsection{A shift theorem for the dual problem}
In order to study the regularity of the solutions to~\eqref{dual:problem:2}, we rewrite the dual problem in an equivalent formulation by using Lemmata~\ref{lemma:on:adjoint:operators} and~\ref{lemma:Laplace:Beltrami}.
\begin{lem} \label{lemma:rewriting:dual:problem:strong:formulation}
Let $\sm \ge 1/2$ and $\sv \ge 0$. 
Let $(\rr, \rm, \rext )\in L^2(\Omega) \times H^{-\sm}(\Gamma) \times H^{-\sv}(\Gamma)$ and
let~$\Rm$ and~$\Rext$ be the representers of $\rm$ and $\rext$, respectively, constructed 
in Lemma~\ref{lemma:Laplace:Beltrami}.  
Then, problem~\eqref{dual:problem:1} is equivalent to the three following coupled problems:
find~$(\psi, \psim, \psitilde) \in H^1(\Omega) \times
H^{-\frac{1}{2}}(\Gamma) \times H^{\frac{1}{2}}(\Gamma)$ such that, in
strong form,
\begin{align}
&\begin{cases}
-\div(\Ad \nabla \overline  \psi ) - (\k)^2 \overline \psi =
\overline r  \quad \mbox{in }\Omega,\\
\nabla \overline\psi \cdot \nGamma + i k \overline\psi + \overline{\psim} = 0 \quad \mbox{on } \Gamma,\\
\end{cases}\label{dual:problem:first:equation:strong}\\
& \begin{cases}  -\overline \psi + (\frac{1}{2} + \Kk +i k \Vk) \overline{\psitilde} + \Vk \overline{\psim} = \overline \Rm\quad \mbox{on } \Gamma, \end{cases}\label{dual:problem:second:equation:strong} \\
& \begin{cases} (\Wk + i k (\frac{1}{2} - \Kprimek) - i k (\frac{1}{2} + \Kk + i k\Vk)) \overline{\psitilde} - ( (\frac{1}{2} + \Kprimek) + i k \Vk)  \overline{\psim} = \overline{\Rext} \quad \mbox{on } \Gamma.
\end{cases}\label{dual:problem:third:equation:strong:3}
\end{align}
\end{lem}
\begin{proof}
By selecting~$\lambda=0$ and~$\vtilde=0$ in~\eqref{dual:problem:2}, we get
\[
\begin{cases}
\text{find } ( \psi,\psim) \in H^1(\Omega)\times H^{-\frac{1}{2}}(\Gamma) \text{ such that}\\
( \Ad \nabla v, \nabla \psi)_{0,\Omega} - ((\k)^2 v, \psi)_{0,\Omega} + i k (v, \psi)_{0,\Gamma} + \langle v,\psim\rangle =  (v,\rr)_{0,\Omega} \quad \forall \psi \in H^1(\Omega),\\
\end{cases}
\]
which entails~\eqref{dual:problem:first:equation:strong} after an integration by parts.
\medskip

Next, by choosing~$v=0$ and~$\vtilde=0$ in~\eqref{dual:problem:2} and using~\eqref{regularity:Melenk:trick}, we obtain
\[
\begin{cases}
\text{find }  (\psi,\psim,\psitilde) \in H^1(\Omega)\times H^{-\frac{1}{2}}(\Gamma) \times H^{\frac{1}{2}}(\Gamma)   \text{ such that}\\
- \langle \lambda, \psi \rangle + \langle \Aprimek \lambda, \psitilde \rangle + \langle  \Vk \lambda, \psim \rangle = (\lambda, \rm)_{-\sm, \Gamma} =  \langle \lambda, \Rm \rangle \quad \forall \lambda\in H^{-\frac{1}{2}}(\Gamma).
\end{cases}
\]
In order to get~\eqref{dual:problem:second:equation:strong}, it is enough to use the definition of the adjoint of an operator, and to apply the 
identities~\eqref{adjoint:of:blo:1} and~\eqref{adjoint:Aprime} when dealing with~$\Vk$ and~$(\Ak')^*\psi$, respectively.
\medskip

Finally, we observe that, by taking~$v=0$ and~$\lambda=0$ and by using~\eqref{regularity:Melenk:trick},
\[
\begin{cases}
\text{find }  (\psi,\psim,\psitilde) \in H^1(\Omega)\times H^{-\frac{1}{2}}(\Gamma) \times H^{\frac{1}{2}}(\Gamma)   \text{ such that}\\
- \langle  ( \Bk + i k  \Aprimek) \vtilde, \psitilde \rangle  - \langle (\frac{1}{2} + \Kk) \vtilde + i k \Vk\vtilde, \psim \rangle \\
\quad\quad\quad\quad\quad\quad\quad\quad = (\vext,\rext)_{-\sv,\Gamma}   = (\vtilde, \Rext)_{0,\Gamma} \quad \forall \psiext \in H^{\frac{1}{2}} (\Gamma).
\end{cases}
\]
By using the definition of the adjoint of an operator and the identities~\eqref{adjoint:of:blo:1} and~\eqref{adjoint:Aprime:2}, we rewrite the previous equation as
\[
\langle \vext, \overline{(\Wk  + k^2 \Vk)\overline{\psiext}} + i k \overline{(\Kprimek + \Kk ) \overline{\psiext}} \rangle - \langle \vext, \frac{1}{2}\overline{\overline{\psim}} + \overline{ \Kprimek \overline{\psim}} - i k \overline{\Vk \overline{\psim}})\rangle = \langle \vext, \Rext \rangle,
\]
whence follows~\eqref{dual:problem:third:equation:strong:3}.
\end{proof}

We conclude this section by proving that, assuming some smoothness of the coefficients of the differential operator and smoothness of the interface~$\Gamma$,
then the solution operator to the dual problem 
\eqref{dual:problem:first:equation:strong}--\eqref{dual:problem:third:equation:strong:3} satisfies a shift theorem.

\begin{thm} \label{theorem:regularity:duality-new}
Let~$\Gamma$ be smooth and assume that~$\Ad$ is smooth, and that $\Ad$
and $n$ satisfy~\eqref{eq:assumption-on-a}. Assume~\eqref{assumption:uniqueness}.
Then: 
\begin{enumerate}[(i)]
\item 
\label{item:theorem regularity duality-i-new}
Fix~$s\in \mathbb R^+_{0}$ and let 
the refraction index~$n \in \mathcal C^{\infty}(\mathbb R^3, \mathbb C)$. 
Let, for 
\[
\rr \in H^{s}(\Omega),\quad\quad \Rm \in H^{\frac{3}{2}+s}(\Gamma), \quad\quad \Rext \in H^{\frac{1}{2}+s}(\Gamma),
\]
the triple $(\psi,\psim, \psitilde)$ be the solution to~\eqref{dual:problem:first:equation:strong}--\eqref{dual:problem:third:equation:strong:3}. Then~$(\psi,\psim, \psitilde)$ satisfy
\[
\psi \in H^{s+2}(\Omega), \quad\quad \psim \in H^{s+\frac{1}{2}}(\Gamma), \quad\quad \psitilde \in H^{s+\frac{3}{2}}(\Gamma),
\]
together with the {\sl a priori} estimates
\begin{equation}\label{bounds:norms:solutions:dual-new}
\Vert \psi \Vert_{s+2,\Omega} + 
 \Vert \psim \Vert_{s+\frac{1}{2}, \Gamma} + 
 \Vert \psitilde \Vert_{s+\frac{3}{2}, \Gamma} 
   \lesssim
   \left(\Vert \rr \Vert_{s,\Omega} + \Vert \Rm \Vert_{s + \frac{3}{2}, \Gamma} + \Vert \Rext \Vert_{s+\frac{1}{2}, \Gamma} \right),
\end{equation}
where the hidden constant depends on $k$, $\Omega$, $\Ad$, and $n$.
\item 
\label{item:theorem:regularity:duality-ii-new}
If the refraction index~$n$ is in~$L^\infty(\Omega, \mathbb C)$ on~$\Omega$, then the bounds in~\eqref{bounds:norms:solutions:dual-new} hold true with~$s=0$.
\end{enumerate}
\end{thm}
\begin{proof}
\emph{Proof of~\eqref{item:theorem regularity duality-i-new}:}
The proof identifies a potential $\ddaleth$ (which depends on 
$\psiext$, $\psim$) such that the function $\Lcalk$ defined
in~\eqref{Lcalk} below satisfies an elliptic transmission problem, for which
a shift theorem is available. The regularity of $\Lcalk$ then will allow us
to infer the regularity of $\psi$, $\psiext$, and $\psim$.   

We will only consider the case of integer $s \in {\mathbb N}_0$, as the 
general case is then obtained by interpolation. 

\emph{1.~step (a priori estimate):} 
By Theorem~\ref{theorem:well-posedness:of:weak:mortar:formulation}, the sesquilinear form 
${\mathcal T}$ satisfies an inf-sup condition, so that we have the ($k$-dependent)  
{\sl a priori} bound:
\begin{equation} \label{a:priori:bound}
\vert \psi \vert_{1,\Omega} + \Vert \psim \Vert_{-\frac{1}{2}, \Gamma} + \Vert \psiext \Vert_{\frac{1}{2},\Gamma}   \lesssim \Vert  r \Vert_{\left(H^{1}(\Omega)\right)^\prime} 
+ \Vert {\Rm} \Vert_{\frac{1}{2},\Gamma} + \Vert {\Rext} \Vert_{-\frac{1}{2},\Gamma}.
\end{equation}
\emph{2.~step (the potentials $\aaleph$ and $\bbeth$):} 
We introduce the two following potentials:
\begin{equation} \label{potentials:for:some:psi:objects}
\aaleph := \Ktildek \overline {\psitilde} + i k\Vtildek \overline{\psitilde}, \quad \quad \quad \bbeth := \Vtildek \overline{\psim}.
\end{equation}
By noting that
\begin{equation} \label{Neumann:lambda}
\begin{split}
&\gammaoint \bbeth = (\frac{1}{2} + \Kprimek) \overline{\psim},\quad\quad \gammaoext \bbeth = (-\frac{1}{2} + \Kprimek) \overline{\psim},\quad\quad \gammazint \bbeth = \gammazext \bbeth = \Vk \overline{\psim},
\end{split}
\end{equation}
we also have
\begin{equation} \label{Neumann_vprime}
\gammaoint \aaleph = -\Wk \overline{\psitilde} + ik (\frac{1}{2} + \Kprimek) \overline {\psitilde}, \quad\quad\quad \gammaoext \aaleph = -\Wk \overline{\psitilde} + i k (-\frac{1}{2} + \Kprimek) \overline {\psitilde},
\end{equation}
and
\begin{equation} \label{Dirichlet:vprime}
\gammazint \aaleph = (-\frac{1}{2} + \Kk) \overline{\psitilde} + i k \Vk \overline{\psitilde}, \quad\quad\quad \gammazext \aaleph = (\frac{1}{2} + \Kk) \overline{\psitilde} + i k \Vk \overline{\psitilde}.
\end{equation}
By using~\eqref{Neumann:lambda}, \eqref{Neumann_vprime} and~\eqref{Dirichlet:vprime}, we can rewrite~\eqref{dual:problem:third:equation:strong:3} in terms of the two auxiliary potentials~$\aaleph$ and~$\bbeth$:
\begin{equation} \label{couple:of:equations:related:to:the:third:1}
-\gammaoext \aaleph - i k \gammazext \aaleph - \gammaoint \bbeth - i k \gammazint \bbeth =  \overline {\Rext}.
\end{equation}
Since~\eqref{Neumann_vprime} and~\eqref{Dirichlet:vprime} imply that
\[
- \gammaoext \aaleph - i k \gammazext \aaleph = - \gammaoint \aaleph - i k \gammazint \aaleph,
\]
we deduce
\begin{align} 
& - \gammaoint \aaleph - i k \gammazint \aaleph - \gammaoint \bbeth  - i k \gammazint \bbeth =  \overline{\Rext} .\label{couple:of:equations:related:to:the:third:2}
\end{align}
It is also possible to reshape~\eqref{dual:problem:second:equation:strong} as
\begin{equation} \label{dual:problem:second:equation:strong:rewritten}
-\overline{\psi}  +  \gammazext \aaleph+ \gammazext \bbeth  = \overline{\Rm}.
\end{equation}
\medskip
\emph{3.~step (the potential $\ddaleth$):} 
Introduce the combined potential
\begin{equation} \label{definition:of:z}
\ddaleth := \aaleph + \bbeth,
\end{equation}
where $\aaleph$ and $\bbeth$ are defined in~\eqref{potentials:for:some:psi:objects}.
For future use, we record from~\ref{Neumann:lambda}, \eqref{Dirichlet:vprime} the jump relation 
\[
\llbracket \ddaleth \rrbracket_\Gamma = -\overline{\psiext}. 
\]
In view of the mapping properties of the 
double and single layer potentials, see, e.g., \cite{steinbach_BEMbook , SauterSchwab_BEMbook},
\begin{equation} \label{mapping:properties:potentials}
\Ktildek: H^{s'+\frac{1}{2}}(\Gamma) \rightarrow H^{s'+1}(\Omega \cup (\Omega^+\cap B_R(0))), 
\quad\quad \Vtildek : H^{s'-\frac{1}{2}}(\Gamma) \rightarrow H^{s'+1}(\Omega \cup (\Omega^+\cap B_R(0))),
\end{equation}
for any fixed $R > 0$, we deduce for $s' \ge 0$ such that the right-hand
side is finite that 
\[
\|\ddaleth \|_{H^{1+s'}(\Omega \cup (\Omega^+\cap B_R(0)))} \lesssim 
\|\psim \|_{s'-\frac{1}{2},\Gamma} + \|\psiext\|_{s'+\frac{1}{2},\Gamma}.   
\]
In particular, for $s' = 0$, we have in view of~\eqref{a:priori:bound}
\begin{equation}
\label{eq:ddaleth-H1}
\|\ddaleth \|_{H^{1}(\Omega \cup (\Omega^+\cap \B_R(0)))} \lesssim 
\|r\|_{\left(H^1(\Omega)\right)^\prime} + \|\Rm\|_{\frac{1}{2},\Gamma} + \|\Rext\|_{-\frac{1}{2},\Gamma}. 
\end{equation}
Thanks to~\eqref{couple:of:equations:related:to:the:third:2}, \eqref{couple:of:equations:related:to:the:third:1}, and~\eqref{Neumann:lambda}, $\ddaleth$ satisfies
\begin{equation} \label{problem:solved:by:z}
\begin{cases}
-\Delta \ddaleth - k^2 \ddaleth = 0 & \text{in } \mathbb R^3\setminus \Gamma\\
\gammaoint \ddaleth + i k \gammazint \ddaleth  =  -  \overline{\Rext} & \text{on } \Gamma \\
\gammaoext \ddaleth + i k \gammazext \ddaleth = - \overline{\Rext} - \overline{\psim} & \text{on } \Gamma.\\
\end{cases}
\end{equation}
Regularity of $\ddaleth$ on $\Omega$ and $\Omega^+$ beyond~\eqref{eq:ddaleth-H1} can be inferred from~\eqref{problem:solved:by:z} by elliptic regularity. 
Indeed, since $\ddaleth{}_{|\Omega}$ satisfies an elliptic equation with impedance boundary conditions, 
we get from the smoothness of $\Gamma$ and $\Rext \in H^{s+\frac{1}{2}}(\Gamma)$
that~$\ddaleth \in H^{s+2}(\Omega)$; see, e.g., \cite[Theorem 6, Section 6.3]{evansPDE} 
or \cite[Lemma~{6.5}]{MelenkParsaniaSauter_generalDGHelmoltz}.
Taking the trace on $\Gamma$ we get for $\ddalethint := \gammazint (\ddaleth)$ 
\begin{equation}
\label{eq:ddalethint}
\|\ddaleth\|_{s+2,\Omega} + \|\ddalethint\|_{s+\frac{3}{2},\Gamma} \lesssim \|\Rext\|_{s+\frac{1}{2},\Gamma}. 
\end{equation}
Combining \eqref{dual:problem:second:equation:strong:rewritten} and \eqref{definition:of:z}, we also have
\[
\ddalethext:= \gammazext (\ddaleth) =  \overline{\Rm} +\overline \psi
\]
As~$\overline {\psi}_{|\Gamma} = \gammazint(\overline \psi)$ we arrive at 
\begin{equation} \label{combining:first:and:second:properties}
\gammazint(\overline \psi + \ddaleth) - \gammazext (\ddaleth) = \ddalethint - \overline{\Rm} 
\in H^{s+ \frac{3}{2}}(\Gamma). 
\end{equation}
\emph{4.~step (the function $\Lcalk$ as the solution of a transmission problem):}
We introduce the function~$\Lcalk$ 
in~$\mathbb R^3 \setminus \Gamma$, which will be see to satisfy 
the transmission problem defined by~\eqref{contazzi:malefici}, \eqref{jump:Lcalk}, \eqref{jump:partial_n-Lcalk} below:

\begin{equation} \label{Lcalk}
\Lcalk =
\begin{cases}
\overline \psi + \chi \ddaleth 	& \text{in } \Omega,\\
\ddaleth					& \text{in } \Omegap,\\
\end{cases}
\end{equation}
where~$\chi$ is a smooth cut-off function such that~$\chi = 1$ in~$\mathcal N(\Gamma)$, which we recall denotes a sufficiently small neighborhood of~$\Gamma$, such that~$\Ad$ and~$n$ are equal to~$1$ on the support of~$\chi$.

The jump relation 
\begin{equation} \label{jump:Lcalk}
\llbracket \Lcalk \rrbracket_\Gamma = \ddalethint - \overline{\Rm} 
\in H^{s+3/2}(\Gamma), 
\end{equation}
is an immediate consequence of \eqref{combining:first:and:second:properties}. Moreover, we observe that
\begin{equation} \label{jump:partial_n-Lcalk}
\begin{split}
\llbracket \partial_{\nGamma} \Lcalk \rrbracket_\Gamma 	& = \gammaoint( \overline {\psi}) + \llbracket \partial_{\nGamma} \ddaleth \rrbracket_\Gamma = \gammaoint(\overline {\psi}) + \gammaoint \aaleph + \gammaoint \bbeth - \gammaoext \aaleph - \gammaoext \bbeth\\
								& \overset{\eqref{Neumann:lambda},\,\eqref{Neumann_vprime}}{=} \gammaoint( \overline \psi) + i k \overline{\psitilde} + \overline{\psim} \overset{\eqref{dual:problem:first:equation:strong}}{=} i k (\overline{\psiext} - \overline \psi).
\end{split}
\end{equation}
Next, we write an equation solved by~$\Lcalk$ in~$\Omega$:
\[
-\div(\Ad \nabla \Lcalk) \overset{\eqref{Lcalk}}{=} -\div(\Ad \nabla \overline \psi) - \div( \Ad \nabla (\chi \ddaleth)) \overset{\eqref{dual:problem:first:equation:strong}}{=} (\k)^2 \overline \psi + \overline \rr - \left\{ \nabla \Ad \cdot \nabla (\chi \ddaleth) + \Ad \Delta (\chi \ddaleth)    \right\}.
\]
We study the two terms in the curly braces $\{\cdots\}$ on the right-hand side separately.
The first one belongs to~$H^{s+1}(\Omega)$, since we are assuming smoothness of~$\Ad$ and we have that~$\ddaleth \in H^{s+2}(\Omega)$.
As far as the second one is concerned, we note that
\[
\begin{split}
\Ad \Delta(\chi \ddaleth) 	& = \Ad \{ \Delta \chi \ddaleth 	+ 2 \nabla \chi \cdot \nabla \ddaleth + \chi \Delta \ddaleth  \} \overset{\eqref{problem:solved:by:z}}{=} \Ad \{ \underbrace{\Delta \chi \ddaleth}_{\in H^{s+2}(\Omega)}	+ \underbrace{2 \nabla \chi \cdot \nabla \ddaleth}_{\in H^{s+1}(\Omega)} 
																							- \underbrace{k^2 \chi \ddaleth}_{\in H^{s+2} (\Omega)}  \},
\end{split}
\]
where we have used again that~$\ddaleth \in H^{s+2}(\Omega)$.

We infer that
\begin{equation} \label{contazzi:malefici}
\begin{split}
-&\div	(\Ad \nabla \Lcalk) 	 - (\k)^2 \Lcalk \\
	& = (\k)^2 \overline \psi + \overline {\rr} - (\k)^2 \overline \psi - (\k)^2 \chi \ddaleth - \left\{ \nabla \Ad \cdot \nabla (\chi \ddaleth) + \Ad   (\Delta \chi \ddaleth + 2\nabla\chi \cdot \nabla\ddaleth  + \chi \Delta \ddaleth)     \right\}\\
	& = \underbrace{\overline {\rr}}_{\in H^s(\Omega)} - \underbrace{(\k)^2 \chi \ddaleth}_{\in H^{s+2}(\Omega)} - \underbrace{\nabla \Ad \cdot \nabla (\chi \ddaleth)}_{\in H^{s+1}(\Omega)} 
					- \Ad \{  \underbrace{\Delta \chi \ddaleth}_{\in H^{s+2}(\Omega)} + \underbrace{2\nabla\chi \cdot \nabla\ddaleth}_{\in H^{s+1}(\Omega)}  - \underbrace{k^2 \chi \ddaleth }_{\in H^{s+2}(\Omega)}   \} \\
	& =: \RHSLcalk \in H^{s} (\Omega).
\end{split}
\end{equation}
\emph{5.~step (bootstrapping regularity):} 
In the following step~6, we will establish a shift theorem
for $\Lcalk$. The {\sl a priori} estimate \eqref{a:priori:bound} provides $\psi \in H^1(\Omega)$ and $\psiext\in H^{\frac{1}{2}}(\Gamma)$ 
so that $\llbracket \partial_{\n_\Gamma} \Lcalk \rrbracket_\Gamma 
\in H^{\frac{1}{2}}(\Gamma)$; cf.~\eqref{jump:partial_n-Lcalk}.
The shift theorem of the 6.~step then will provide
$\Lcalk \in H^{\min\{2,s+2\}}(\Omega \cup (\Omega^+ \cap B_R(0)))$ 
with  
\begin{equation}
\label{eq:iteration:Lcalk}
\|\Lcalk \|_{\min\{2,s+2\},\Omega} + 
\|\Lcalk \|_{\min\{2,s+2\},\Omega^+\cap B_R(0)} 
\lesssim \|\Rext\|_{s+\frac{1}{2},\Gamma} + \|\Rm\|_{s+\frac{3}{2},\Gamma} 
+ \|\psi\|_{1,\Omega} 
+ \|\psiext\|_{\frac{1}{2},\Gamma}. 
\end{equation}
The definition of $\Lcalk$ in~\eqref{Lcalk} and the estimate~\eqref{eq:ddalethint} provide $\psi \in H^{\min\{2,s+2\}}(\Omega)$ with 
\begin{equation}
\label{eq:iteration:psi}
\|\psi\|_{\min\{2,s+s\},\Omega} 
\lesssim \|\Rext\|_{s+\frac{1}{2},\Gamma} + \|\Rm\|_{s+\frac{3}{2},\Gamma} 
+ \|\psi\|_{1,\Omega} 
+ \|\psiext\|_{\frac{1}{2},\Gamma}. 
\end{equation}
Thanks to~\eqref{definition:of:z}, \eqref{Neumann:lambda}, and~\eqref{Dirichlet:vprime}, it holds
\begin{equation} \label{nice:property:jump:Dirichlet:Lcalk}
\llbracket \Lcalk \rrbracket_\Gamma \overset{\eqref{combining:first:and:second:properties}}{=} \gammazint(\ddaleth + \overline \psi) - \gammazext\ddaleth = \overline{\psi} - \overline{\psitilde} \quad \Longrightarrow \quad \overline{\psitilde} = \overline{\psi} - \llbracket \Lcalk \rrbracket_\Gamma .
\end{equation}
Inserting~\eqref{eq:iteration:Lcalk} and~\eqref{eq:iteration:psi} in~\eqref{nice:property:jump:Dirichlet:Lcalk} provides 
$\psiext \in H^{\min\{\frac{3}{2},s+\frac{3}{2}\}}(\Gamma)$. 
Hence, we get from~\eqref{jump:partial_n-Lcalk} that 
$\llbracket \partial_{\n_\Gamma} \Lcalk \rrbracket_\Gamma 
\in H^{\frac{3}{2}}(\Gamma)$.  
That is, we have improved the regularity
of $(\psi,\psiext)$ from $H^1(\Omega) \times H^{\frac{1}{2}}(\Gamma)$ 
to $H^2(\Omega) \times H^{\frac{3}{2}}(\Gamma)$ and in turn 
of 
$\llbracket \partial_{\n_\Gamma} \Lcalk \rrbracket_\Gamma $ from 
$H^{\frac{1}{2}}(\Gamma)$ to 
$H^{\frac{3}{2}}(\Gamma)$. The arguments can be repeated with this improved
regularity to infer 
$(\psi,\psiext) \in H^{3}(\Omega) \times H^{\frac{5}{2}}(\Gamma)$, 
$(\psi,\psiext) \in H^{4}(\Omega) \times H^{\frac{7}{2}}(\Gamma)$ 
until we have reached
$(\psi,\psiext) \in H^{s+2}(\Omega) \times H^{\frac{3}{2}+s}(\Gamma)$. 
Finally, equation~\eqref{dual:problem:second:equation:strong} yields~$\psim \in H^{\frac{1}{2}+s}(\Gamma)$. 

\noindent \emph{6.~step:} 
For $s' \ge 0$ let the scalar function $U$ satisfy the transmission problem
\[
\begin{split}
&-\div (\Ad \nabla U) - (\k)^2 U = 
\begin{cases} g \in H^{s'}(\Omega), & \mbox{in $\Omega$}, \\
              0                     & \mbox{ in $\Omega^+$}
\end{cases},\\
&\llbracket U\rrbracket_\Gamma = g_1 \in H^{s'+\frac{3}{2}}(\Gamma),  \quad \llbracket \partial_{\n_\Gamma}U\rrbracket_\Gamma = g_2 \in H^{s'+\frac{1}{2}}(\Gamma), 
\end{split}
\]
We claim that for fixed $R>0$
\[
\|U\|_{s'+2,\Omega} + \|U\|_{s'+2,\Omega^+\cap B_R(0)} \lesssim \|g\|_{s',\Omega} + \|g_1\|_{s'+\frac{3}{2},\Gamma} +  \|g_2\|_{s'+\frac{1}{2},\Gamma}.
\]
To see this shift theorem, we first remove the jumps across $\Gamma$. 
We define 
\[
\Pcal := - \Ktildek g_1  + \Vtildek g_2 + \Ntildek g .
\]
Using the smoothing properties of 
the double and single layer potentials given in~\eqref{mapping:properties:potentials},  
we deduce that
\begin{equation}
\label{eq:estimate-Pcal}
\|\Pcal \|_{s'+2,\Omega} + 
\|\Pcal \|_{s'+2,\Omega^+\cap B_R(0)} 
\lesssim \|g_1 \|_{s'+\frac{3}{2},\Gamma} + 
\|g_2 \|_{s'+\frac{1}{2},\Gamma} + \|g\|_{s',\Omega}. 
\end{equation}
The jump relations~\eqref{eq:jump-rel-V-1}, \eqref{eq:jump-rel-V-2}, \eqref{eq:jump-rel-K-1}, \eqref{eq:jump-rel-K-2} imply 
$\llbracket \Pcal \rrbracket_\Gamma = g_1$ and 
$\llbracket \partial_{\n_\Gamma} \Pcal \rrbracket_\Gamma = g_2$.   
Furthermore, we have 
\begin{equation}
\label{eq:Pcal}
-\Delta \Pcal - k^2 \Pcal = 
\begin{cases} g & \mbox{ on $\Omega$}, \\
              0 & \mbox{ on $\Omega^+$}. 
\end{cases}
\end{equation}
At this point, we define the additional function~$\Zcal$ as 
\[
\Zcal:= U  - \chi \Pcal 
\]
with~$\chi$ being the same cut-off function as the one used in~\eqref{Lcalk}.
By construction, we observe that
\[
\llbracket \Zcal \rrbracket_\Gamma = 0 = \llbracket \partial_{\nGamma} \Zcal \rrbracket_\Gamma. 
\]
A calculation reveals  on $\Omega$ and $\Omega^+$ 
\begin{align}
\nonumber 
\div (\Ad \nabla (\chi \Pcal))  + \k^2 \chi\Pcal & = 
\Ad \chi \Delta \Pcal + 
2 \Ad \nabla \chi \cdot \nabla \Pcal + \chi \nabla \Ad \cdot \nabla \Pcal 
+ \Ad \Delta \chi \Pcal + (\k)^2 \chi \Pcal \\
& 
=: \chi( \Ad \Delta \Pcal + (\k)^2 \Pcal) + \Rcal 
\label{eq:calculate-P}
\end{align}
with 
\begin{equation*}
\|\Rcal \|_{s'+1,\Omega} + 
\|\Rcal \|_{s'+1,\Omega^+\cap B_R(0)} 
\stackrel{\eqref{eq:estimate-Pcal}}{ \lesssim }
\|g_1 \|_{s'+\frac{3}{2},\Gamma} + 
\|g_2 \|_{s'+\frac{1}{2},\Gamma} + \|g\|_{s',\Omega} 
\end{equation*}
and $\Rcal \equiv 0$ on a neighborhood $\mathcal{N}^\prime(\Gamma)$ of 
$\Gamma$. 

Upon writing $\widetilde g$ for the zero extension of $g$ to ${\mathbb R}^3$ 
we get on $\Omega \cup \Omega^+$ 
from~\eqref{eq:calculate-P}, \eqref{eq:Pcal}
\begin{align}
\label{eq:Z}
- \div (\Ad \nabla (\chi \Zcal))  - \k^2 \chi\Zcal & = 
(1 - \Ad \chi) \widetilde g + \chi k^2 (n^2 - \Ad)\Pcal + \Rcal 
=:\Rcal^\prime. 
\end{align}
We note that $\Rcal$ vanishes near $\Gamma$ and 
$\|\Rcal^\prime \|_{s',\Omega} + 
\|\Rcal^\prime \|_{s',\Omega^+\cap B_R(0)} 
\lesssim 
\|g_1 \|_{s'+\frac{3}{2},\Gamma} + 
\|g_2 \|_{s'+\frac{1}{2},\Gamma} + \|g\|_{s',\Omega}$. 
In view of the jump conditions satisfied by $\Zcal$, the 
equation~\eqref{eq:Z} holds on ${\mathbb R}^3$. Standard elliptic 
regularity then gives, for any $R'< R$, that
$\Zcal \in H_{loc}^{s'+2}({\mathbb R}^3)$ and 
$
\|\Zcal\|_{s'+2,B_{R'}(0)} \lesssim \|\Rcal^\prime\|_{s', B_R(0)}, 
$
which leads to the desired claim. 


\emph{Proof of~\eqref{item:theorem:regularity:duality-ii-new}:}
The case of a refraction index in~$L^{\infty}$ follows along the same
lines as the smooth case, with the difference that the shift result
of Step~6 is only valid for $s' = 0$. 
Therefore, all the estimates are valid substituting~$s$ with~$0$.
\end{proof}

Finally, we prove the shift theorem for the adjoint variational problem~\eqref{dual:problem:2}.
\begin{thm} \label{theorem:regularity:duality}
Let~$\Gamma$ be smooth and assume that~$\Ad$ is smooth, and that $\Ad$
and $n$ satisfy~\eqref{eq:assumption-on-a}. Assume~\eqref{assumption:uniqueness}. Let $\sm \ge \frac{1}{2}$ and $\sv \ge 0$. 
Then: 
\begin{enumerate}[(i)]
\item 
\label{item:theorem regularity duality-i}
Let the refraction index~$n \in \mathcal C^{\infty}(\mathbb R^3, \mathbb C)$. 
Assume $s_1$, $s_2$, $s_3 \ge 0$, $s_2 \ge \frac{3}{2} - \sm$, and $s_3 \ge \frac{1}{2}-\sv$.
Set $s:=\min\{s_1, s_2-\frac{3}{2}+\sm, s_3-\frac{1}{2}+\sv\}\ge 0$. 
Let, for 
\[
\rr \in H^{s_1}(\Omega),\quad\quad \rm \in H^{-\sm+s_2}(\Gamma), \quad\quad \rext \in H^{-\sv+s_3}(\Gamma),
\]
the triple $(\psi,\psim, \psitilde)$ be the solution 
to~\eqref{dual:problem:2}. 
Then~$(\psi,\psim, \psitilde)$ satisfy
\[
\psi \in H^{s+2}(\Omega), \quad\quad \psim \in H^{s+\frac{1}{2}}(\Gamma), \quad\quad \psitilde \in H^{s+\frac{3}{2}}(\Gamma),
\]
together with the {\sl a priori} estimates
\begin{equation}\label{bounds:norms:solutions:dual}
\Vert \psi \Vert_{s+2,\Omega} + 
 \Vert \psim \Vert_{s+\frac{1}{2}, \Gamma} + 
 \Vert \psitilde \Vert_{s+\frac{3}{2}, \Gamma} 
\lesssim \left(\Vert \rr \Vert_{s_1,\Omega} + \Vert \rm \Vert_{-\sm + s_2, \Gamma} + \Vert \rext \Vert_{-\sv+s_3, \Gamma} \right),
\end{equation}
where the hidden constant depends on $k$, $\Omega$, $\Ad$, and $n$.
\item 
\label{item:theorem:regularity:duality-ii}
If the refraction index~$n$ is in~$L^\infty(\Omega, \mathbb C)$ on~$\Omega$, then the bounds in~\eqref{bounds:norms:solutions:dual} hold true with~$s=s_1=0$
and $s_2$, $s_3$ as in~\eqref{item:theorem regularity duality-i}. 
\end{enumerate}
\end{thm}
\begin{proof}
\emph{Proof of~\eqref{item:theorem regularity duality-i}:} 
{}Lemma~\ref{lemma:Laplace:Beltrami} provides representers 
$\Rm \in H^{\sm+s_2}(\Gamma)$ and $\Rext \in H^{\sv+s_2}(\Gamma)$ such that 
$(\cdot,\rm)_{-\sm,\Gamma} = \langle \cdot,\Rm\rangle$ and 
$(\cdot,\rext)_{-\sv,\Gamma} = \langle \cdot,\Rext\rangle$. 
In view of $s + \frac{3}{2} \leq \sm+ s_2 $ and 
in view of $s + \frac{1}{2} \leq \sv+ s_3 $, we get the result from 
Theorem~\ref{theorem:regularity:duality-new}. 

\emph{Proof of~\eqref{item:theorem:regularity:duality-ii}:} 
For $s_1 = 0$ we have $s = 0$. The assumptions on $s_2$, $s_3$ imply
that the representers $\Rm \in H^{\frac{3}{2}}(\Gamma)$, 
$\Rext \in H^{\frac{1}{2}}(\Gamma)$. 
Theorem~\ref{theorem:regularity:duality-new}, \eqref{item:theorem:regularity:duality-ii-new} then proves the claim.  
\end{proof}

\begin{remark} \label{remark:diffusion}
The analysis in Theorem~\ref{theorem:regularity:duality-new} 
has been performed assuming that~$\Ad$ is globally smooth.
The smoothness assumption on~$\Ad$ can be relaxed to piecewise smooth in the following sense: Let~$\Omega_i$, $i=1,\ldots,N$, be Lipschitz domains whose closures are pairwise disjoint and~$\overline{\Omega}_i \subset \Omega$.
Let~$\Ad$ be smooth on each~$\overline{\Omega}_i$ and smooth on~$\Omega \setminus \cup_i \Omega_i$.
Assume that the following shift theorem holds: There is~$s_0 \in (0,1]$ such that, for any~$g \in H^{-1+s_0}(\Omega)$, the solution~$v \in H^1_0(\Omega)$ of
$$
-\div(\Ad \nabla v) = g \quad \mbox{ in $\Omega$}, 
$$
satisfies $\sum_i \|v\|_{H^{1+s_0}(\Omega_i)} + \|v\|_{H^{1+s_0}(\Omega \setminus \cup_i \overline{\Omega_i})} \lesssim \|g\|_{H^{-1+s_0}(\Omega)}$.
Then, for ~$r \in \left(H^{1-s_0}(\Omega)\right)^\prime$, 
$\Rm \in H^{-1+s_0+3/2}(\Gamma)$, and~$\Rext \in H^{-1+s_0+1/2}(\Gamma)$, 
the solution $(\psi,\psim,\psitilde)$ satisfies 
$\sum_i \|\psi\|_{H^{1+s_0}(\Omega_i)} +
\|\psi\|_{H^{1+s_0}(\Omega\setminus \cup_i\overline{\Omega_i})} +
\|\psim\|_{H^{1+s_0-3/2}(\Gamma)} +
\|\psitilde\|_{H^{1+s_0-3/2}(\Gamma)} \lesssim \|r
\|_{H^{-1+s_0}(\Omega)} + \|\rm \|_{H^{-1+s_0-3/2}(\Gamma)} \      
+ \|\rext \|_{H^{-1+s_0-1/2}(\Gamma)}$. 
This regularity assertion is sufficient to perform the convergence analysis of Section~\ref{section:FEM-BEM:mortar} on meshes that are aligned with subdomains~$\Omega_i$, $i=1,\ldots,N$. 
\eremk
\end{remark}

\section{Analysis of the FEM-BEM mortar coupling} \label{section:FEM-BEM:mortar}
In this section, we discuss a FEM-BEM mortar coupling tailored to the approximation of the solution to~\eqref{weak:formulation:mortar:Helmholtz}, and equivalently to~\eqref{equivalent:weak:formulation}.

Given a triple of finite dimensional spaces $V_h \times \Wh \times \Zh \subset H^1(\Omega) \times H^{-\frac{1}{2}}(\Gamma) \times H^{\frac{1}{2}}(\Gamma)$, 
the discretization of~\eqref{equivalent:weak:formulation} reads
\begin{equation} \label{FEM-BEM}
\begin{cases}
\text{Find } (\uh,\mh,\uhext) \in \Vh\times \Wh \times \Zh \text{ such that}\\
(\Ad \nabla \uh,  \nabla \vh )_{0,\Omega} - ((\k)^2 \uh,  \vh)_{0,\Omega} + i k (\uh,  \vh)_{0,\Gamma} - \langle \mh, \vh \rangle = (\f, \vh)_{0,\Omega}\quad \forall \vh \in \Vh,\\
\langle  (\Bk + i k \Aprimek) \uhext  - \Aprimek \mh, \vhtilde   \rangle = 0 \quad \forall \vhtilde \in \Zh,\\
\langle \uh, \lambdah \rangle - \langle  (\frac{1}{2} + \Kk) \uhext - \Vk (\mh - i k \uhext), \lambdah\rangle  =  0 \quad \forall \lambdah \in \Wh.\\
\end{cases}
\end{equation}
We will show unique solvability of~\eqref{FEM-BEM}, as well as a quasi-optimality result using the 
Schatz argument~\cite{schatz1974}, which  relies on 
($i$) the G{\aa}rding inequality~\eqref{final:Garding}, ($ii$) the regularity of the solution to the dual problem in Theorem~\ref{theorem:regularity:duality}, 
and ($iii$) an approximation property of the space $\Vh \times \Wh \times \Zh$. For the latter we make the following assumption:
\begin{assumption}[Approximation property] 
\label{assumption:approximation}
Given $s_0 > 0$, the family $\{\Vh \times \Wh \times \Zh\}_{h > 0} \subset H^1(\Omega) \times H^{-\frac{1}{2}}(\Gamma) \times H^{\frac{1}{2}}(\Gamma)$
of spaces has the following approximation property: For any $\varepsilon > 0$ there is $h_0 > 0$ such that
for all $h \in (0,h_0]$ there holds for all $(\psi,\psim,\psitilde) \in 
H^{1+s_0}(\Omega) \times H^{-\frac{1}{2}+s_0}(\Gamma) \times H^{\frac{1}{2}+s_0}(\Gamma)$ 
with $\|\psi\|_{H^{1+s_0}(\Omega)} + \|\psim\|_{H^{-\frac{1}{2}+s_0}(\Gamma)} + \|\psitilde\|_{H^{\frac{1}{2}+s_0}(\Gamma)} \leq 1$ 
\[ 
\inf_{\psih \in \Vh} \|\psi - \psih\|_{H^1(\Omega)} + 
\inf_{\psimh \in \Wh} \|\psim - \psimh\|_{H^{-\frac{1}{2}}(\Gamma)} + 
\inf_{\psitildeh \in \Zh} \|\psitilde - \psitildeh\|_{H^{\frac{1}{2}}(\Gamma)}  \leq \varepsilon.
\]
\end{assumption} 

\begin{remark}
For domains $\Omega$ with smooth boundary $\Gamma$, families $\Vh$, $\Wh$, $\Zh$ with some approximation
properties can be constructed as spaces of piecewise mapped polynomials. In order to resolve the boundary 
$\Gamma$, curved elements have to employed in the construction of $\Vh$, using, e.g., the technique 
of ``transfinite blending'', \cite{Gordon71,GordonHall73,GordonHall73b}. A specific family of spaces of 
piecewise polynomials of degree $p$ on meshes of size $h$ is given in \cite[Example~{5.1}]{melenk-sauter10}; that
family has the expected approximation properties in terms of the mesh size $h$ and the polynomial degree $p$. 
The spaces $\Wh$ and $\Zh$ can also be constructed as spaces of piecewise mapped polynomials of degree $p$
on a mesh on $\Gamma$ using the parametrization(s) of~$\Gamma$. We refer to \cite[Sec.~{4.1}]{SauterSchwab_BEMbook}
for details.
\eremk
\end{remark}
\begin{thm} \label{theorem:abstract:error:analysis}
Let $(u,\m,\uext) \in H^1(\Omega) \times H^{-\frac{1}{2}}(\Gamma) \times H^{\frac{1}{2}}(\Gamma)$ be the solution to 
problem~\eqref{weak:formulation:mortar:Helmholtz}.
Let the family of space $\{\Vh \times \Wh \times \Zh\}_{h >0}$ satisfy Assumption~\ref{assumption:approximation} with $s_0 = 1$. 
Then there is $h_0 > 0$ such that for all $h \in (0,h_0]$ there is a unique solution~$(\uh,\mh,\uhext) \in \Vh\times \Wh \times \Zh$ 
of~\eqref{FEM-BEM}. Moreover, there is $C > 0$, which depends on $k$, $\Omega$, $\Ad$, and $n$, such that for $h \in (0,h_0]$ 
\[
\begin{split}
  & \Vert\Ad^{\frac12}\nabla( u - \uh) \Vert_{1,\Omega} 
+k\Vert n(u-\uh)\Vert_{0,\Omega}
  + \Vert \m - \mh \Vert_{-\frac{1}{2}, \Gamma} + \Vert \uext - \uhext \Vert_{\frac{1}{2}, \Gamma}\\
& \quad\quad\quad \quad\quad \leq C \Bigl(  \Vert u -\vh \Vert_{1,\Omega}
+ \Vert \m - \nh \Vert_{-\frac{1}{2}, \Gamma} + \Vert \uext - \vexth \Vert_{\frac{1}{2}, \Gamma}\Bigr) \quad \forall (\vh, \nh, \vexth)\in \Vh \times \Wh \times \Zh. 
\end{split}
\]
\end{thm}
\begin{proof}
We follow the classical Schatz argument~\cite{schatz1974}.
We apply the G{\aa}rding inequality~\eqref{final:Garding} and get
\[
\begin{split}
&\Vert\Ad^{\frac12}\nabla( u - \uh) \Vert^2_{0,\Omega}  + \Vert \m - \mh \Vert^2_{-\frac{1}{2}, \Gamma} + \Vert \uext - \uhext \Vert_{\frac{1}{2}, \Gamma}^2\\
& \quad +  k^2\Vert n (u -\uh) \Vert_{0,\Omega}^2 + \cG(k)\left(\Vert \m - \mh \Vert^2_{-\frac{5}{2}, \Gamma}  +  \Vert \uext - \uhext \Vert^2_{-\frac{3}{2},\Gamma}\right)   \\
& \lesssim \Re \Bigl(\T \bigl((u-\uh, \m-\mh, \uext - \uhext), (u-\uh, \m-\mh, \uext - \uhext)\bigl)\Bigr)\\
& \quad + 2 k^2\Vert n (u -\uh) \Vert_{0,\Omega}^2 + 2\cG(k) \left(\Vert \m - \mh \Vert^2_{-\frac{5}{2}, \Gamma} + \Vert \uext - \uhext \Vert_{-\frac{3}{2}, \Gamma}^2  \right),\\
\end{split}
\]
where $\cG(k)$ is the constant appearing in~\eqref{final:Garding}.

In the following, we understand that the implied constants in $\lesssim$ depend on $k$, $\Omega$, $\Ad$, and $n$.

By considering the dual problem~\eqref{dual:problem:1} with~$\rr =2(\k)^2 (u - \uh)$, $\rm = 2\cG(k) (\m - \mh)$, and~$\rext = 2\cG(k) (\uext- \uhext)$, we can write
\[
\begin{split}
&\Vert\Ad^{\frac12}\nabla( u - \uh) \Vert^2_{0,\Omega}  + \Vert \m - \mh \Vert^2_{-\frac{1}{2}, \Gamma} + \Vert \uext - \uhext \Vert_{\frac{1}{2}, \Gamma}^2\\
& \quad +  k^2\Vert n (u -\uh) \Vert_{0,\Omega}^2 + \cG(k)\left(\Vert \m - \mh \Vert^2_{-\frac{5}{2}, \Gamma}  +  \Vert \uext - \uhext \Vert^2_{-\frac{3}{2},\Gamma}\right)   \\
& \lesssim \Re \Bigl(\T \bigl((u-\uh, \m-\mh, \uext - \uhext), (u-\uh, \m-\mh, \uext - \uhext)\bigr)\Bigr)\\
& \quad +  \T\bigl((u-\uh, \m-\mh, \uext - \uhext), (\psi, \psim, \psitilde)\bigr), \\
\end{split}
\]
where, by the stability estimate~\eqref{bounds:norms:solutions:dual} of 
Theorem~\ref{theorem:regularity:duality-new} with the parameters 
$s_1=s_2=s_3=0$, $\sigma_m=\frac52$, $\sigma_v=\frac32$ and thus $s=0$ there holds, 
\begin{align}
\nonumber 
\|\psi\|_{2,\Omega} + \|\psim\|_{\frac{1}{2},\Gamma} + \|\psitilde\|_{\frac{3}{2},\Gamma} 
& \lesssim k^2 \|u - \uh\|_{0,\Omega} + \cG(k) \|\m - \mh\|_{-\frac{5}{2},\Gamma} + \cG(k) \|\uext - \uhext\|_{-\frac{1}{2},\Gamma}, \\
\label{eq:stability-dual}
& \lesssim \|u - \uh\|_{1,\Omega} + \|\m - \mh\|_{-\frac{1}{2},\Gamma} + \cG(k) \|\uext - \uhext\|_{\frac{1}{2},\Gamma}. 
\end{align}

Applying the Galerkin orthogonality, we get for all $\psih\in\Vh$, $\psimh \in \Wh$, and~$\psitildeh \in \Zh$
\begin{equation} \label{splitting:into:I+II}
\begin{split}
&\Vert\Ad^{\frac12}\nabla( u - \uh) \Vert^2_{0,\Omega}  + \Vert \m - \mh \Vert^2_{-\frac{1}{2}, \Gamma} + \Vert \uext - \uhext \Vert_{\frac{1}{2}, \Gamma}^2\\
& \quad +  k^2\Vert n (u -\uh) \Vert_{0,\Omega}^2 + \cG(k)\left(\Vert \m - \mh \Vert^2_{-\frac{5}{2}, \Gamma}  +  \Vert \uext - \uhext \Vert^2_{-\frac{3}{2},\Gamma}\right)   \\
& \lesssim \Re \Bigl(\T \bigl((u-\uh, \m-\mh, \uext - \uhext), (u-\uh, \m-\mh, \uext - \uhext)\bigr)\Bigr)\\
& \quad +  \T\bigl((u-\uh, \m-\mh, \uext - \uhext), (\psi  - \psih  , \psim - \psimh, \psitilde - \psitildeh)\bigr) = I +II. \\
\end{split}
\end{equation}
We estimate the two terms $I$, $II$ on the right-hand side
of~\eqref{splitting:into:I+II} separately, starting with the term~$I$.
  Using again Galerkin orthogonality and the definition of
    $\T(\cdot,\cdot)$ and of the combined integral
    operators,
  we get for all $\vh\in \Vh$, $\nh\in \Wh$, and~$\vexth\in \Zh$
\[
\begin{split}
I 	& = \Re \T \bigl((u-\uh, \m-\mh, \uext - \uhext) , (u-\vh, \m - \nh, \uext - \vexth)  \bigr)\\
	& = \Re \big( (\Ad\nabla (u-\uh),  \nabla (u-\vh))_{0,\Omega} -  ((\k)^2 (u-\uh), u-\vh)_{0,\Omega} + i k (u-\uh, u-\vh)_{0,\Gamma}\\
	& \quad  - \langle \m-\mh, u -\vh \rangle-\langle (\Bk + i k \Aprimek) (\uext- \uhext) - \Aprimek (\m - \mh) , \uext - \vexth \rangle\\
	& \quad + \langle u - \uh,\m-\nh\rangle - \langle (\frac{1}{2} + \Kk) (\uext-\uhext) - \Vk ((\m - \mh) - i k (\uext - \uhext)), \m - \nh \rangle \big).
\end{split}
\]
As (see, e.g., \cite[proof of Lemma 8.1.6]{melenk_phdthesis})
\[
k (u-\uh, u-\vh)_{0,\Gamma} \lesssim k^2 \Vert u - \uh \Vert_{0,\Omega} \Vert u - \vh \Vert_{0,\Omega} + \vert u - \uh \vert_{1,\Omega} \vert u -\vh \vert_{1,\Omega},
\]
we have the straightforward bound
\[
\begin{split}
  \vert I \vert 	& \lesssim
  \Vert u-\uh \Vert_{1,\Omega}
   \Vert u-\vh \Vert_{1,\Omega} 
 + \Vert \m-\mh \Vert_{-\frac{1}{2}, \Gamma} \Vert u -\vh \Vert_{\frac12,\Gamma}\\
			& \quad + \Vert \Bk (\uext- \uhext) \Vert_{-\frac{1}{2}, \Gamma} \Vert \uext - \vexth \Vert_{\frac{1}{2},\Gamma} + k \Vert \Aprimek (\uext- \uhext) \Vert_{-\frac{1}{2}, \Gamma} \Vert \uext - \vexth \Vert_{\frac{1}{2},\Gamma}\\
			& \quad + \Vert \Aprimek (\m - \mh) \Vert_{-\frac{1}{2},\Gamma} \Vert \uext - \vexth \Vert_{\frac{1}{2}, \Gamma}\\
			& \quad + \Vert \m-\nh\Vert_{-\frac{1}{2}, \Gamma} \Vert u - \uh \Vert_{\frac{1}{2}, \Gamma} + \Vert \m -\nh \Vert_{-\frac{1}{2}, \Gamma} \Vert  (\frac{1}{2} + \Kk) (\uext-\uhext) \Vert_{\frac{1}{2}, \Gamma}\\
			& \quad + \Vert \m -\nh \Vert_{-\frac{1}{2}, \Gamma}  \Vert  \Vk (\m - \mh - i \k (\uext - \uhext)) \Vert_{\frac{1}{2}, \Gamma}.\\
\end{split}
\]
Simple calculations, based on the mapping properties
of the trace operator:
    $H^1(\Omega)\rightarrow H^{\frac12}(\Gamma)$ and 
of boundary
integral operators and combined integral operators
(in particular, we use
  $\Bk: H^{\frac12}(\Gamma) \rightarrow H^{-\frac12}(\Gamma)$,
  $\Aprimek, \Kk: H^{\frac12}(\Gamma) \rightarrow H^{\frac12}(\Gamma)$,
  $\Vk: H^{-\frac12}(\Gamma) \rightarrow H^{\frac12}(\Gamma)$),
lead to
\[
\begin{split}
  \vert I \vert 	& \lesssim
  \Vert u-\uh \Vert_{1,\Omega}
\Vert u-\vh \Vert_{1,\Omega} 
  + \Vert \m - \mh \Vert_{-\frac{1}{2},\Gamma} \Vert u -\vh \Vert_{1,\Omega}  \\
			& \quad + \Vert \uext - \uhext \Vert_{\frac{1}{2}, \Gamma} \Vert \uext - \vexth \Vert_{\frac{1}{2}, \Gamma} + \Vert \uext - \uhext \Vert_{-\frac{1}{2}, \Gamma}  \Vert \uext - \vexth \Vert_{\frac{1}{2}, \Gamma}\\
			& \quad + \Vert \m - \mh \Vert_{-\frac{1}{2}, \Gamma} \Vert \uext - \vexth\Vert_{\frac{1}{2}, \Gamma}   \\
			& \quad + \Vert \m -\nh \Vert_{-\frac{1}{2}, \Gamma} \Vert u - \uh \Vert_{1,\Omega}  + \Vert  \m - \nh \Vert_{-\frac{1}{2}, \Gamma} \Vert \uext - \uhext \Vert_{\frac{1}{2},\Gamma}\\
			& \quad + \Vert \m - \nh \Vert_{-\frac{1}{2}, \Gamma} \Vert \m - \mh \Vert_{-\frac{1}{2}, \Gamma} + \Vert \m - \nh \Vert_{-\frac{1}{2}, \Gamma} \Vert \uext -\uhext \Vert_{-\frac{1}{2}, \Gamma}.
\end{split}
\]
An~$\ell^2$ Cauchy-Schwarz inequality, together with $\Vert \uext -\uhext \Vert_{-\frac{1}{2}, \Gamma}\le \Vert \uext - \uhext \Vert_{\frac{1}{2}, \Gamma}$, entails
\begin{equation} \label{estimate:on:I:abstract}
\begin{split}
  \vert I \vert 	& \lesssim
  \left(   \Vert u-\uh \Vert_{1,\Omega}^2+ \Vert \m - \mh
    \Vert^2_{-\frac{1}{2}, \Gamma}
    + \Vert \uext - \uhext \Vert_{\frac{1}{2}, \Gamma}^2  \right) ^{\frac{1}{2}}   \\
			& \quad\quad  \cdot \left(   \Vert u-\vh \Vert_{1,\Omega}^2  + \Vert \m - \nh \Vert_{-\frac{1}{2}, \Gamma}^2  +  \Vert \uext - \vexth \Vert_{\frac{1}{2}, \Gamma}^2 \right) ^{\frac{1}{2}} .\\
\end{split}
\end{equation}
For the term~$II$ appearing on the right-hand side
of~\eqref{splitting:into:I+II}, we proceed similarly as for the
  term~$I$ and deduce 
\begin{equation} \label{estimate:on:II:abstract:partial}
\begin{split}
\vert II \vert	& \lesssim \left(   \Vert u-\uh
    \Vert_{1,\Omega}^2 + \Vert \m - \mh \Vert^2_{-\frac{1}{2},
  \Gamma}
+ \Vert \uext - \uhext \Vert_{\frac{1}{2}, \Gamma}^2  \right) ^{\frac{1}{2}} \\
			& \quad  \cdot \left(   \Vert \psi-\psih \Vert_{1,\Omega}^2 + \Vert \psim - \psimh \Vert_{-\frac{1}{2}, \Gamma}^2+  \Vert \psitilde - \psitildeh \Vert_{\frac{1}{2}, \Gamma}^2   \right) ^{\frac{1}{2}}. \\
\end{split}
\end{equation}
Assumption~\ref{assumption:approximation} with $s_0$ and the regularity assertion~\eqref{eq:stability-dual} provide for each $\varepsilon >0$ an $h_0 = h_0(\varepsilon)$ 
such that for $h \in (0,h_0]$ we have 
\begin{align}\label{estimate:on:psi}
\begin{split}
  &\Vert \psi -\psih \Vert^2_{1,\Omega} 
  +\Vert \psim - \psimh \Vert_{-\frac{1}{2}, \Gamma}^2
  +\Vert \psitilde - \psitildeh \Vert_{\frac{1}{2}, \Gamma}^2\\
  &\qquad\qquad\qquad \lesssim \varepsilon^2 \left(
    \Vert\psi \Vert^2_{2,\Omega}
    +\Vert \psim \Vert^2_{\frac{1}{2}, \Gamma}
    +\Vert \psitilde \Vert^2_{\frac{3}{2}, \Gamma}
  \right)\\
 &\qquad\qquad\qquad \lesssim \varepsilon^2\left(
    \Vert u-\uh \Vert^2_{1,\Omega}
    +\Vert \m-\mh \Vert^2_{-\frac{1}{2},\Gamma}
    + \Vert \uext - \uhext \Vert^2_{\frac{1}{2}, \Gamma}  
  \right).
\end{split}
\end{align}
Inserting~\eqref{estimate:on:psi}
into~\eqref{estimate:on:II:abstract:partial} produces 
\begin{equation} \label{estimate:on:II:abstract}
\begin{split}
\vert II \vert 	& \lesssim  \varepsilon \left(  \Vert u -\uh
    \Vert_{1,\Omega}^2 + \Vert \m - \mh \Vert^2_{-\frac{1}{2},
  \Gamma}
+ \Vert \uext -\uhext\Vert_{\frac{1}{2}, \Gamma}^2  \right). \\
\end{split}
\end{equation}
We insert \eqref{estimate:on:I:abstract} 
and~\eqref{estimate:on:II:abstract} into~\eqref{splitting:into:I+II} 
and take $\varepsilon$ sufficiently small so as to kick the term $II$ back to the left-hand side 
of \eqref{splitting:into:I+II}. The desired quasi-optimality result is obtained.  

The well-posedness of method~\eqref{FEM-BEM} follows as in~\cite{schatz1974}.
\end{proof}

\begin{remark} \label{remark:smooth-vs-polyhedron}
In the numerical experiments of Section~\ref{section:numerical:results}, we employ a polyhedral domain~$\Omega$.
This case is not directly covered by
Theorem~\ref{theorem:abstract:error:analysis}.
Nevertheless, Remark~\ref{remark:nonsmooth-interface} indicates that a G{\aa}rding inequality is valid also for Lipschitz  domains.
The shift theorem for the dual problem in Theorem~\ref{theorem:regularity:duality-new} relies 
on (i) a shift theorem for the iterior impedance problem 
(cf.~\eqref{eq:ddalethint} in Step~3 of the proof of Theorem~\ref{theorem:regularity:duality-new})
and (ii) a shift theorem for a transmission problem (cf.\ Steps~4 and 5 of the proof of Theorem~\ref{theorem:regularity:duality-new}). 
Some shift theorem is also available for both problems for polyhedral domains. 

\eremk
\end{remark}

\section{Numerical results} \label{section:numerical:results}
partitions of $\Omega$ by tetrahedra with straight faces, with $h$ denoting the mesh size. 
The two sequences of meshes~$\{\Pcalh^2\}$ and~$\{\Pcalh^3\}$ on $\Gamma$ are obtained by intersecting $\Gamma$ 
with the tetrahedra in~$\Pcalh^1$. For $\ell\in\mathbb N_0$,
  we denote by ${\mathcal
    P}_\ell$ the space of polynomials of degree at most $\ell$.
For a polynomial degree
$p\ge1$, we set 
\begin{subequations}
\label{approximation:spaces}
\begin{align}
\Vh& :=S^{p,1}(\Omega,\Pcalh^1):= \{v \in H^1(\Omega)\,|\, v_{|_K} \in {\mathcal P}_p \ \forall K \in \Pcalh^1\}, \\
\Wh& :=S^{p-1,0}(\Gamma,\Pcalh^2):= \{v \in L^2(\Gamma)\,|\, v_{|_K} \in {\mathcal P}_{p-1} \ \forall K \in \Pcalh^2\}, \\
\Zh& :=S^{p,1}(\Gamma,\Pcalh^3):= \{v \in H^1(\Gamma)\,|\, v_{|_K} \in {\mathcal P}_{p} \ \forall K \in \Pcalh^3\}. 
\end{align}
\end{subequations}

The numerical experiments for an $h$-version are based on quasi-uniform mesh refinements of a coarse initial triangulation 
and polynomial degrees, $\p=1$, $2$, and~$3$.  We also study the $p$-version on fixed meshes. 
We investigate the behavior of the following relative errors:
\begin{equation} \label{computed:quantities}
\frac{\Vert u - \uh \Vert_{0,\Omega}}{\Vert u \Vert_{0,\Omega}},\quad\quad \frac{\vert u - \uh \vert_{1,\Omega}}{\vert u \vert_{1,\Omega}},
			\quad\quad \h^{\frac{1}{2}}\frac{\Vert \m - \mh \Vert_{0,\Gamma}}{ \Vert \m \Vert_{0,\Gamma}}, \quad\quad\h^{-\frac{1}{2}} \frac{\Vert \uext - \uhext \Vert_{0,\Gamma}}{\Vert \uext\Vert_{0,\Gamma}}.
\end{equation}
Note that the two last quantities scale, in terms of $h$, like the corresponding~$H^{-\frac{1}{2}}(\Gamma)$ and the~$H^{\frac{1}{2}}(\Gamma)$ relative errors, respectively, but are computationally more easily accessible. 

\paragraph*{Test case 1: $\h$-version.}
Let the diffusion coefficient~$\Ad\equiv 1$, and prescribe 
the exact solution to be 
\begin{equation} \label{testcase1:solution}
u(x,y,z) =
\begin{cases}
\sin(k x) \cos(k y) & \text{in } \Omega\\
\frac{e^{i k r}}{r} & \text{in } \Omegap \cup \Gamma,\\
\end{cases}
\end{equation}
where we recall that~$\Omegap:= \mathbb R^3 \setminus \overline\Omega$.

We remark that the function~$u$ in~\eqref{testcase1:solution} does not solve~\eqref{eq:helmholtz} due to the nonzero jumps across $\Gamma$. Instead, we considered
a modified problem and discretization scheme allowing for known jumps across the interface. This only incurs slight changes in the right-hand sides.
  
Here, we are interested in the $\h$-version of the method.
We depict the errors~\eqref{computed:quantities} for three different choices of wavenumber~$k$, namely, $k=\mathfrak k\sqrt{3}  \pi $ with $\mathfrak k = 1.5$, $3$, and~$6$.
Note that~$k$ is an eigenvalue of the Dirichlet-Laplacian in the second and third case.
We begin with the case~$k=1.5\sqrt 3  $, see Figure~\ref{fig:testcase_1}.

\begin{figure}[h]
\begin{center}
\includegraphics[width=0.495\textwidth]{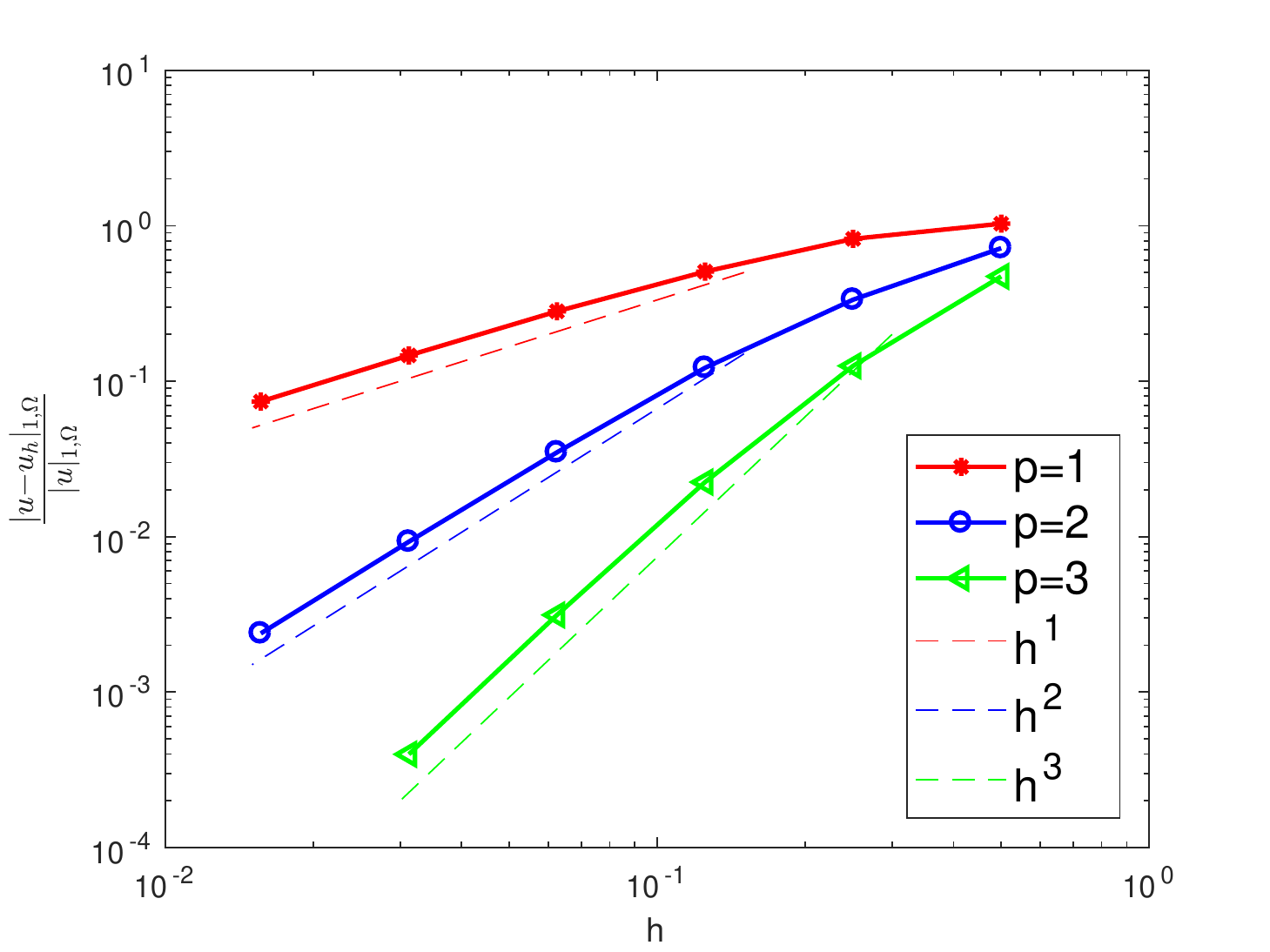}
\includegraphics[width=0.495\textwidth]{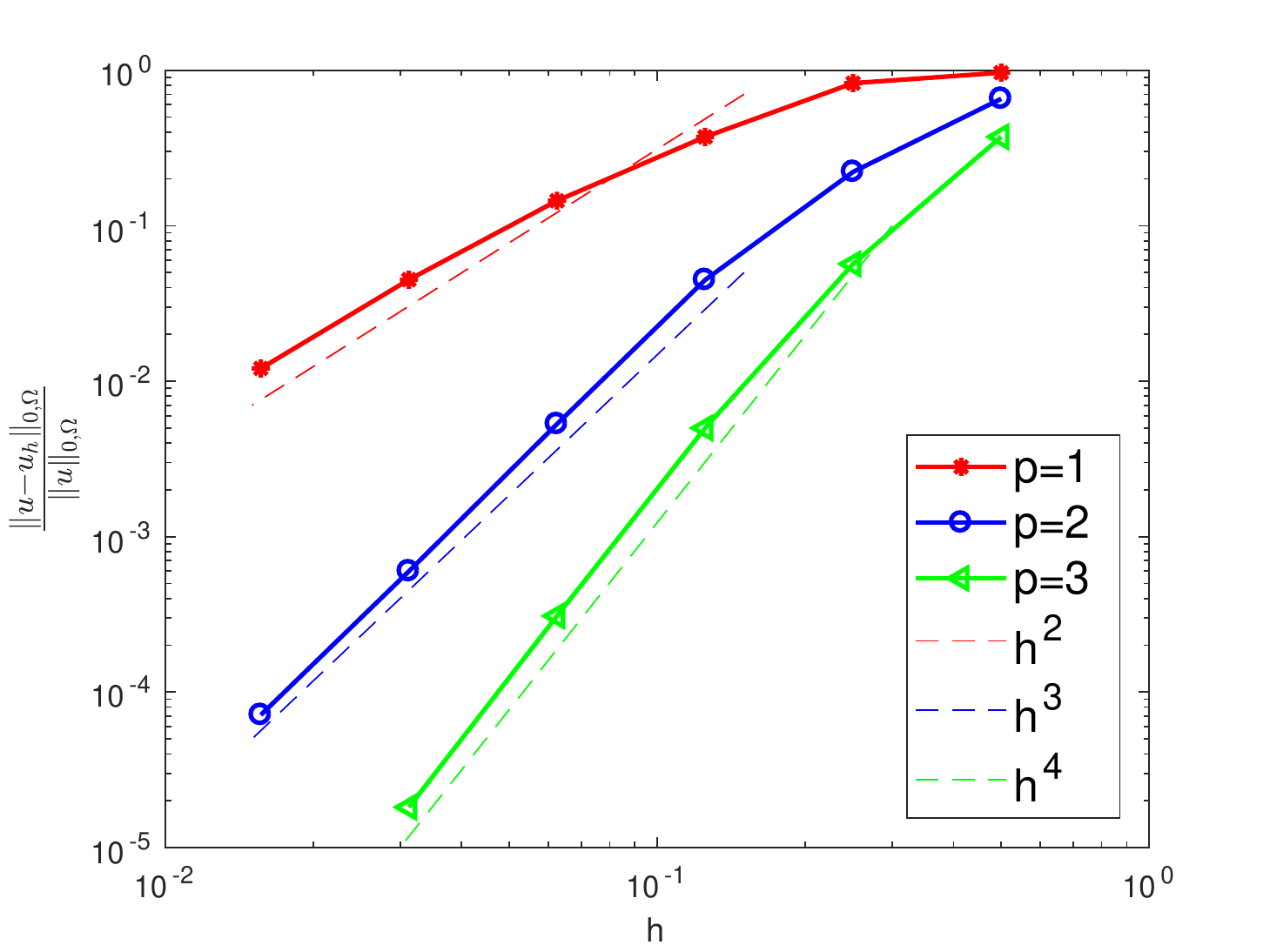}
\includegraphics[width=0.495\textwidth]{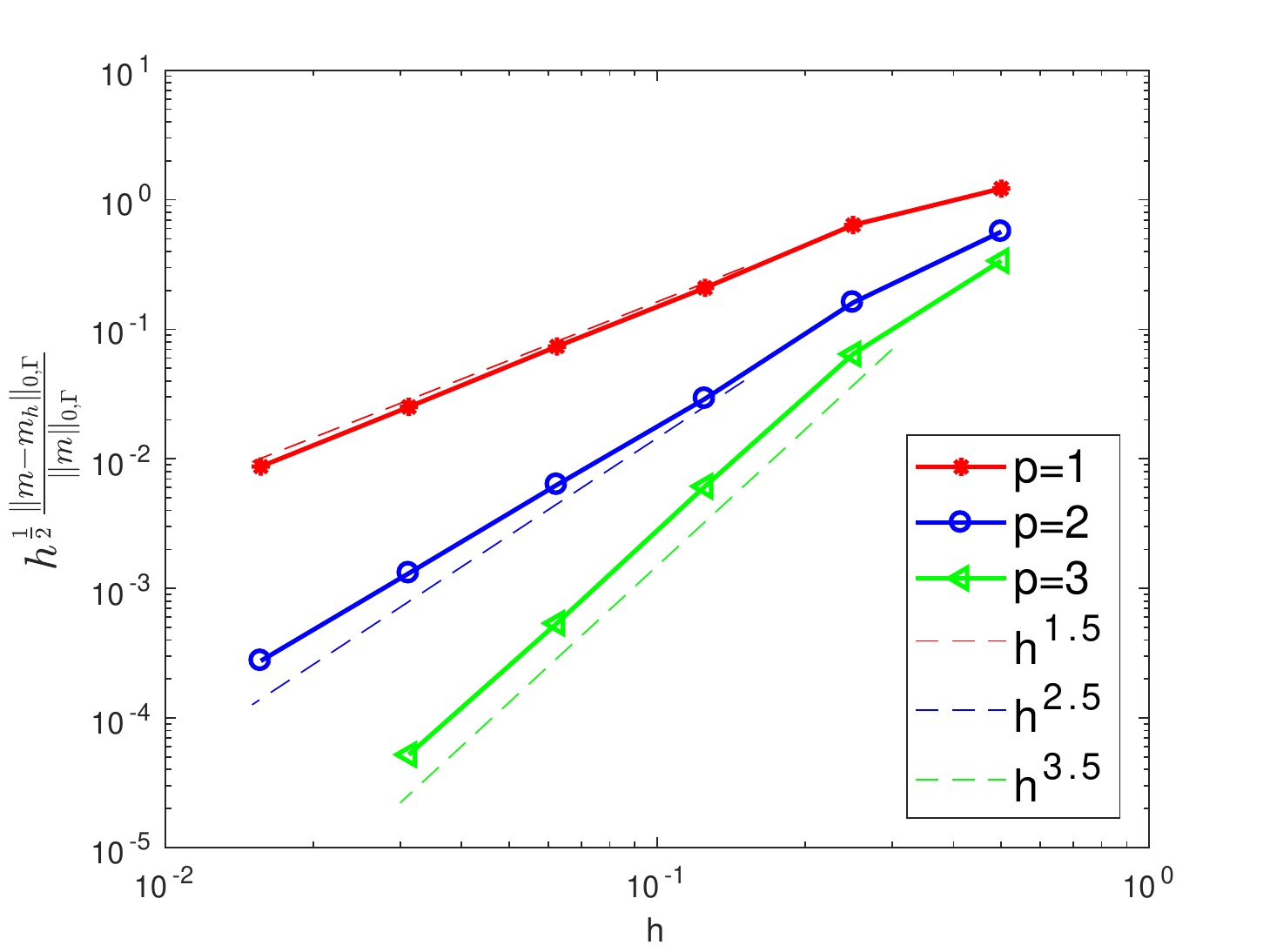}
\includegraphics[width=0.495\textwidth]{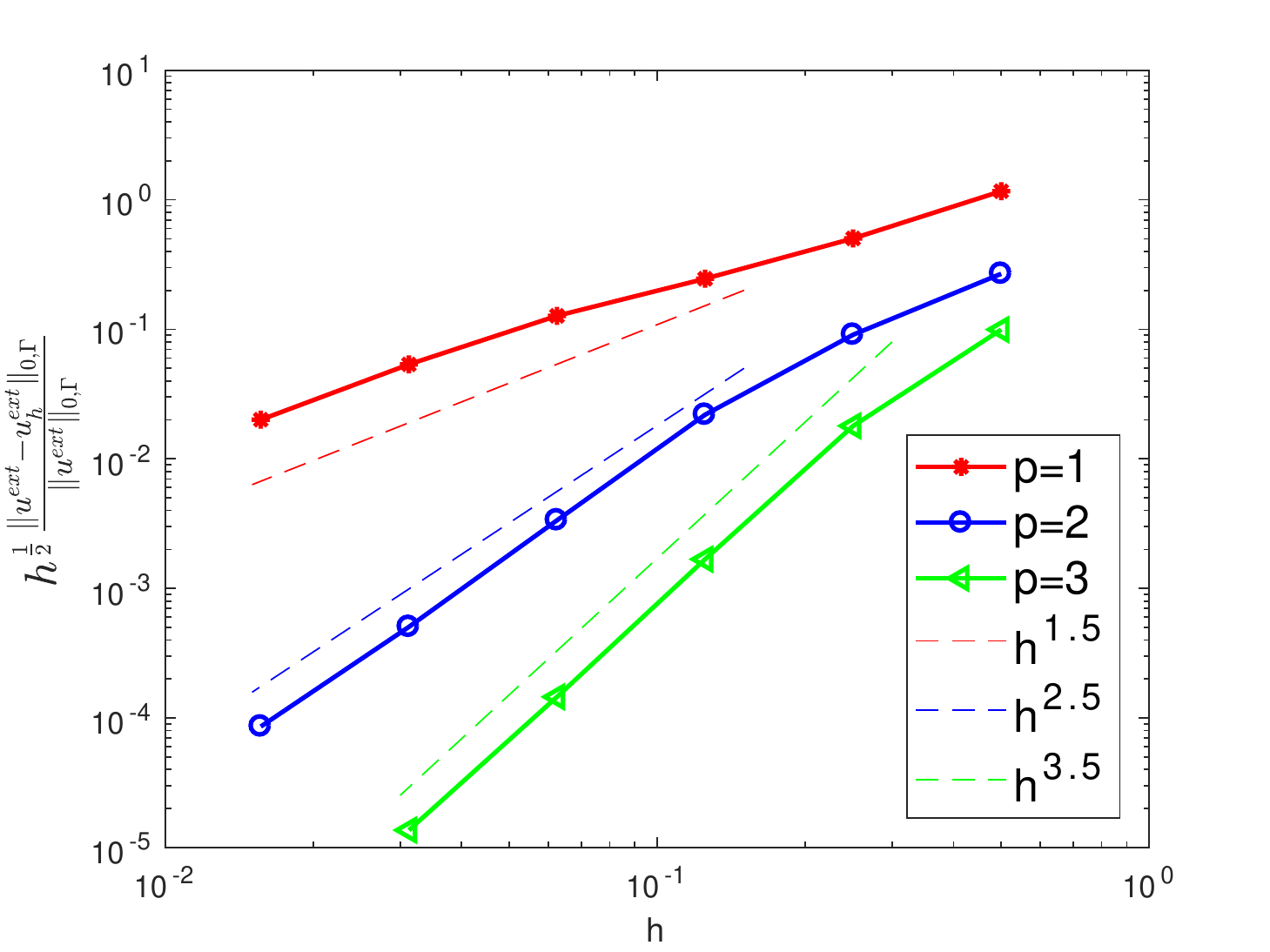}
\end{center}
\caption{The four panels depict the four errors in~\eqref{computed:quantities} for $\p=1,2,3$ in~\eqref{approximation:spaces} versus the mesh size $h$.
The wavenumber~$k=1.5\sqrt 3 \pi$ is \emph{neither} an interior Dirichlet nor a Neumann eigenvalue. $\Ad$ is constant. The solution is given in~\eqref{testcase1:solution}.
\textbf{Top-left panel:} $H^1$ error in~$\Omega$.
\textbf{Top-right panel:} $L^2$ error in~$\Omega$.
\textbf{Bottom-left panel:} $L^2$ error on~$\Gamma$ of the mortar variable times $\h^{\frac{1}{2}}$.
\textbf{Bottom-right panel:} $L^2$ error on~$\Gamma$ times $\h^{-\frac{1}{2}}$.}
\label{fig:testcase_1}
\end{figure}

We observe the optimal $O(h^p)$-convergence in the $H^1$-error. 
However, all the other errors converge faster. A similar ``superconvergence''  
behavior in FEM-BEM coupling has been analyzed in \cite{mpw2017simultaneous} by a refined duality technique.
\medskip

As a second experiment, we consider the wavenumber~$k=3\sqrt 3 \pi$, which is a Dirichlet-Laplace eigenvalue; see Figure~\ref{fig:testcase_2}.
\begin{figure}[h]
\begin{center}
\includegraphics[width=0.495\textwidth]{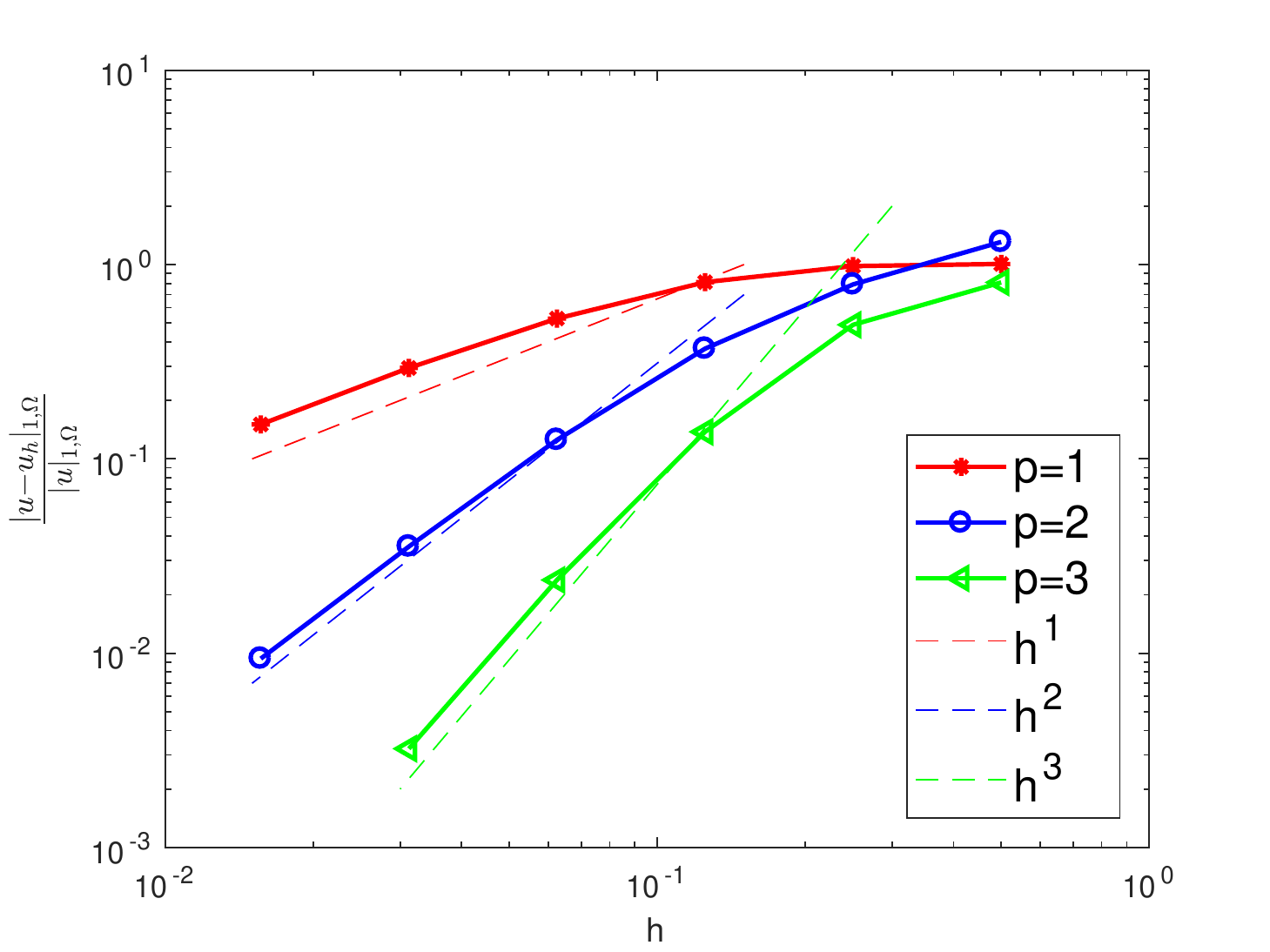}
\includegraphics[width=0.495\textwidth]{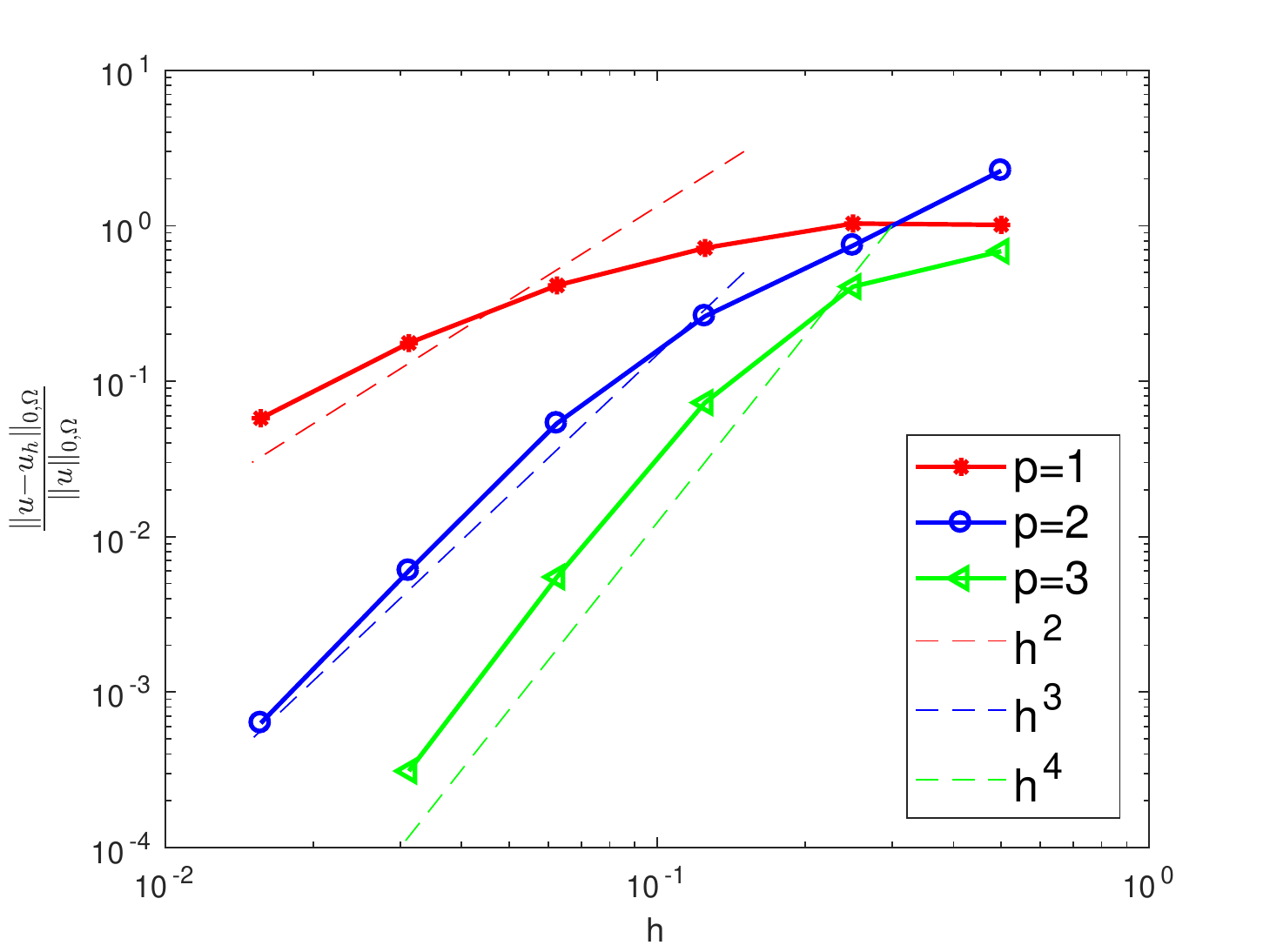}
\includegraphics[width=0.495\textwidth]{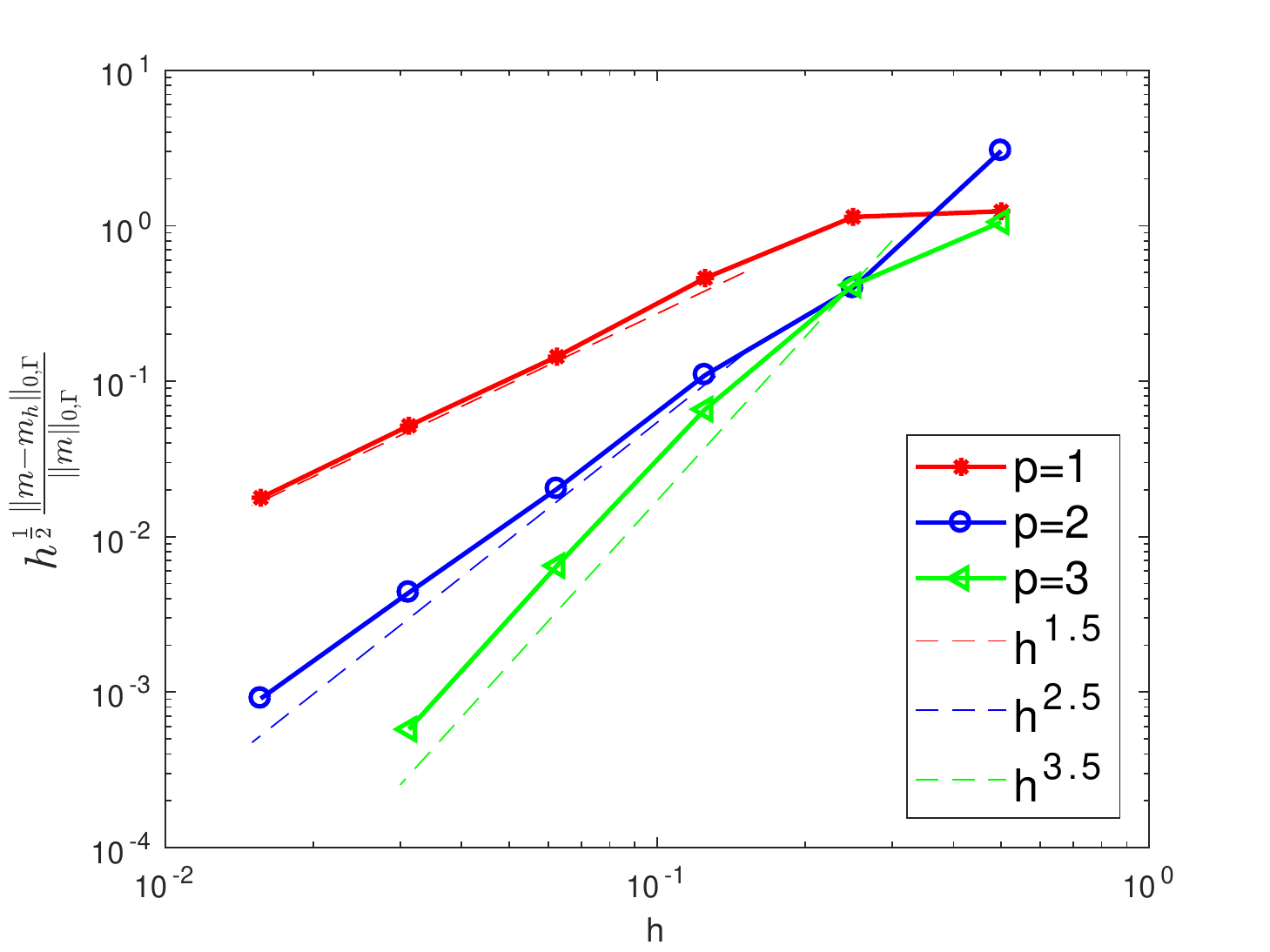}
\includegraphics[width=0.495\textwidth]{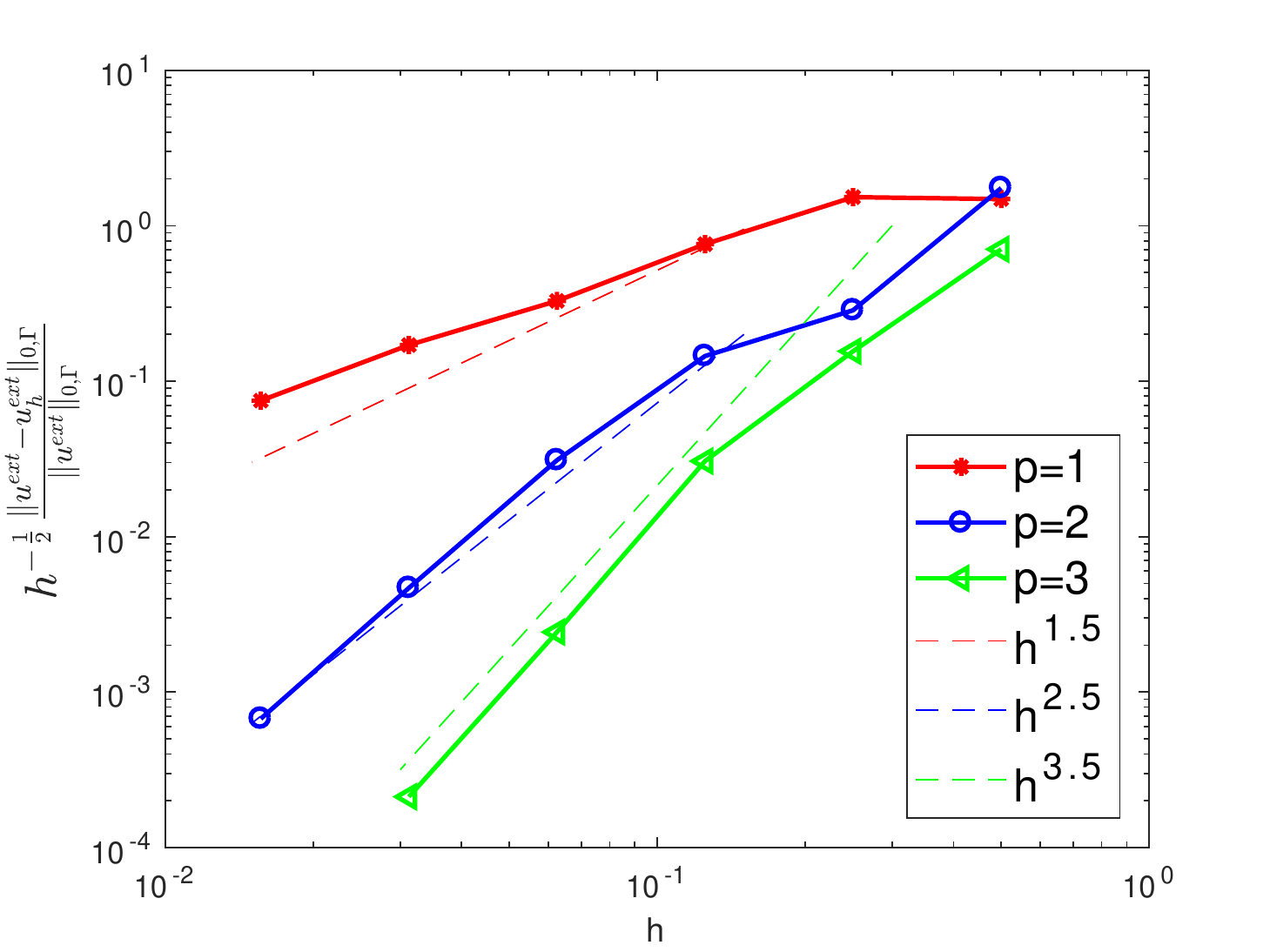}
\end{center}
\caption{The four panels show the four errors in~\eqref{computed:quantities} for $\p =1,2,3$ in~\eqref{approximation:spaces} versus the mesh size $h$.
The wavenumber $k=3 \sqrt 3  \pi$ is a Dirichlet-Laplace eigenvalue.  $\Ad$ is constant. The solution is provided in~\eqref{testcase1:solution}.
\textbf{Top-left panel:} $H^1$ error in~$\Omega$.
\textbf{Top-right panel:} $L^2$ error in~$\Omega$.
\textbf{Bottom-left panel:} $L^2$ error on~$\Gamma$ of the mortar variable times $\h^{\frac{1}{2}}$.
\textbf{Bottom-right panel:} $L^2$ error on~$\Gamma$ times $\h^{-\frac{1}{2}}$.}
\label{fig:testcase_2}
\end{figure}
We observe a very similar convergence behavior to the one observed in Figure~\ref{fig:testcase_1}.
\medskip

Finally, we consider the wavenumber~$k=6\sqrt 3 \pi$, which is again a Dirichlet-Laplace eigenvalue; see Figure~\ref{fig:testcase_3}.
\begin{figure}[h]
\begin{center}
\includegraphics[width=0.495\textwidth]{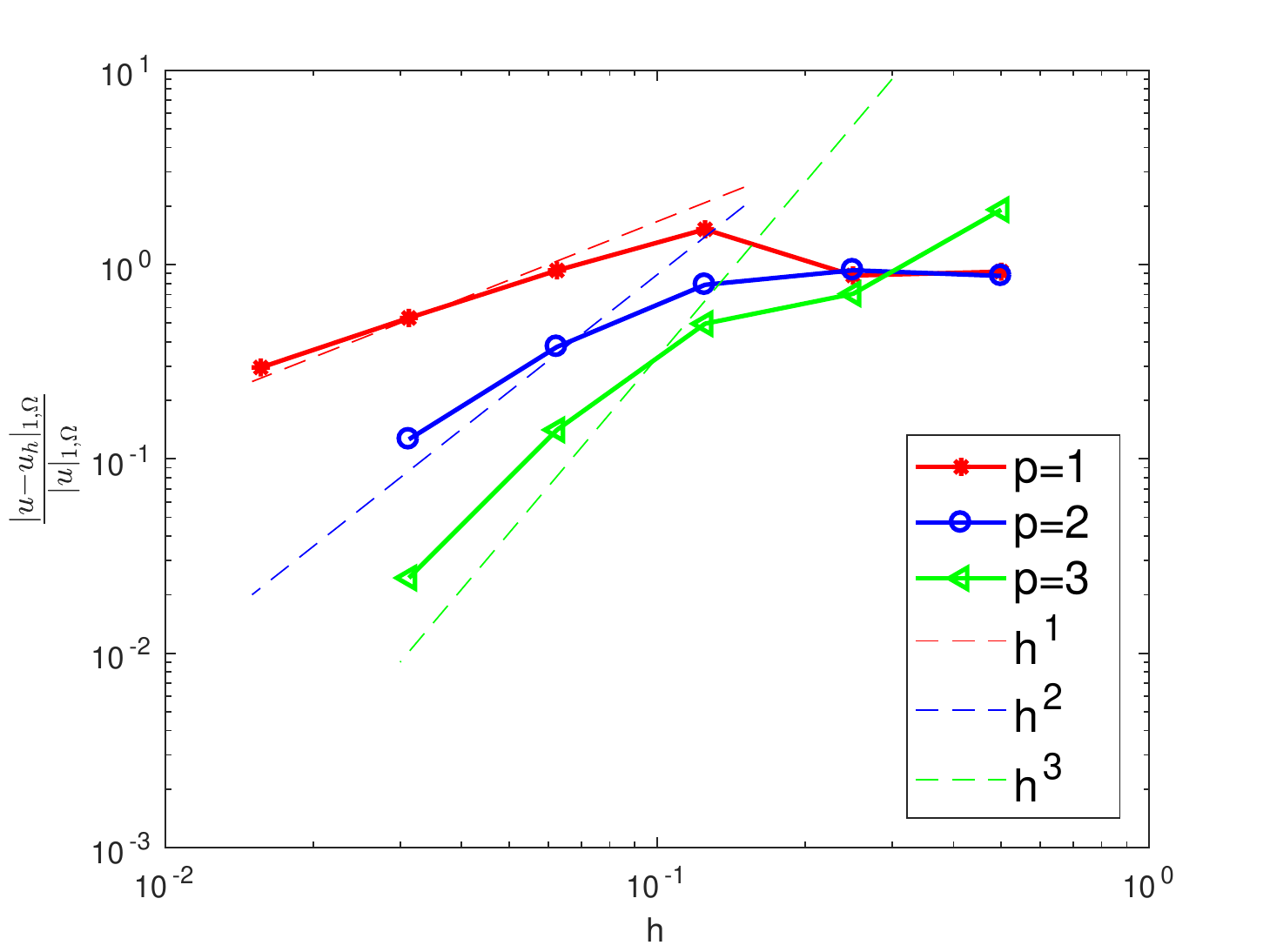}
\includegraphics[width=0.495\textwidth]{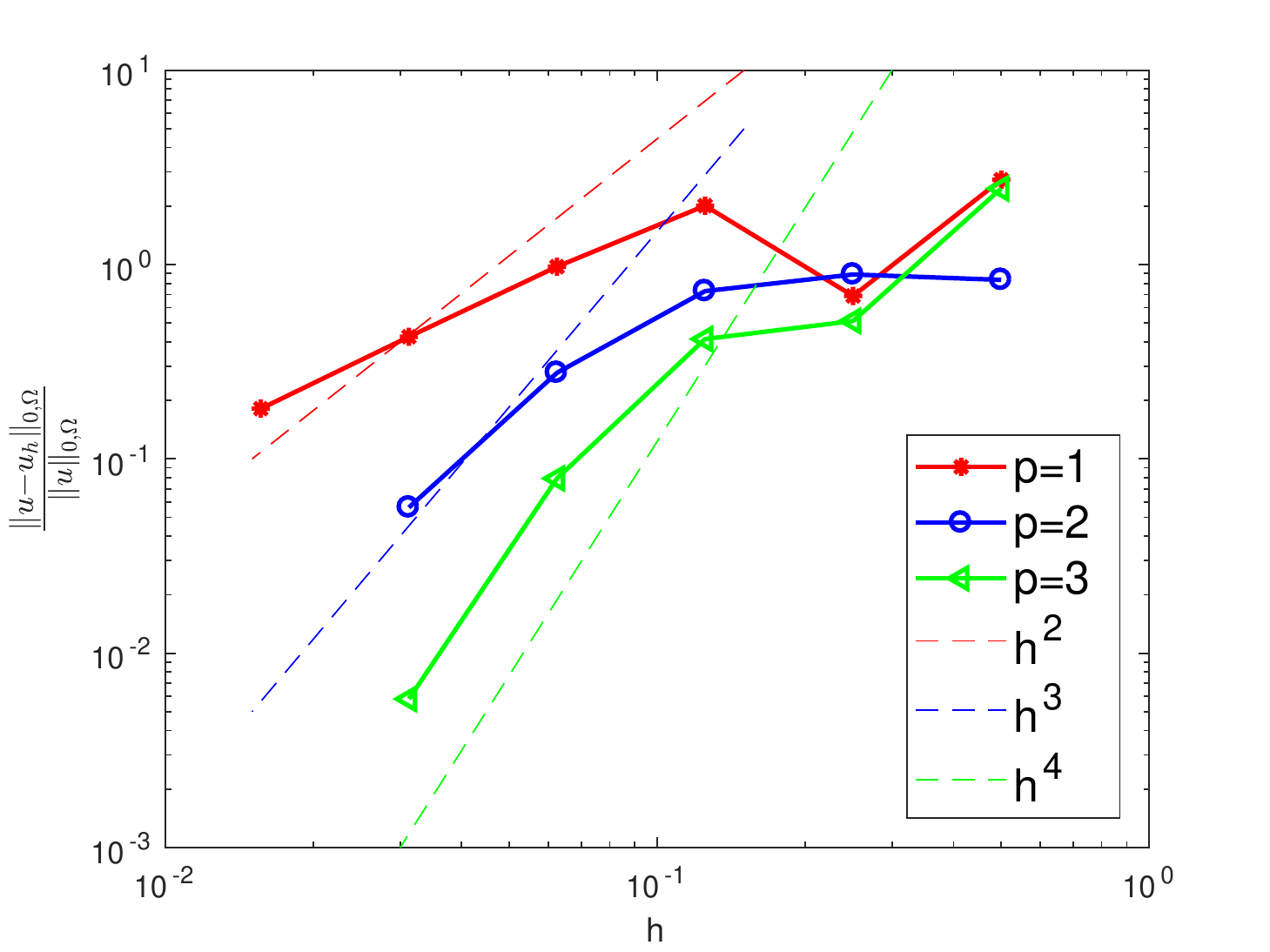}
\includegraphics[width=0.495\textwidth]{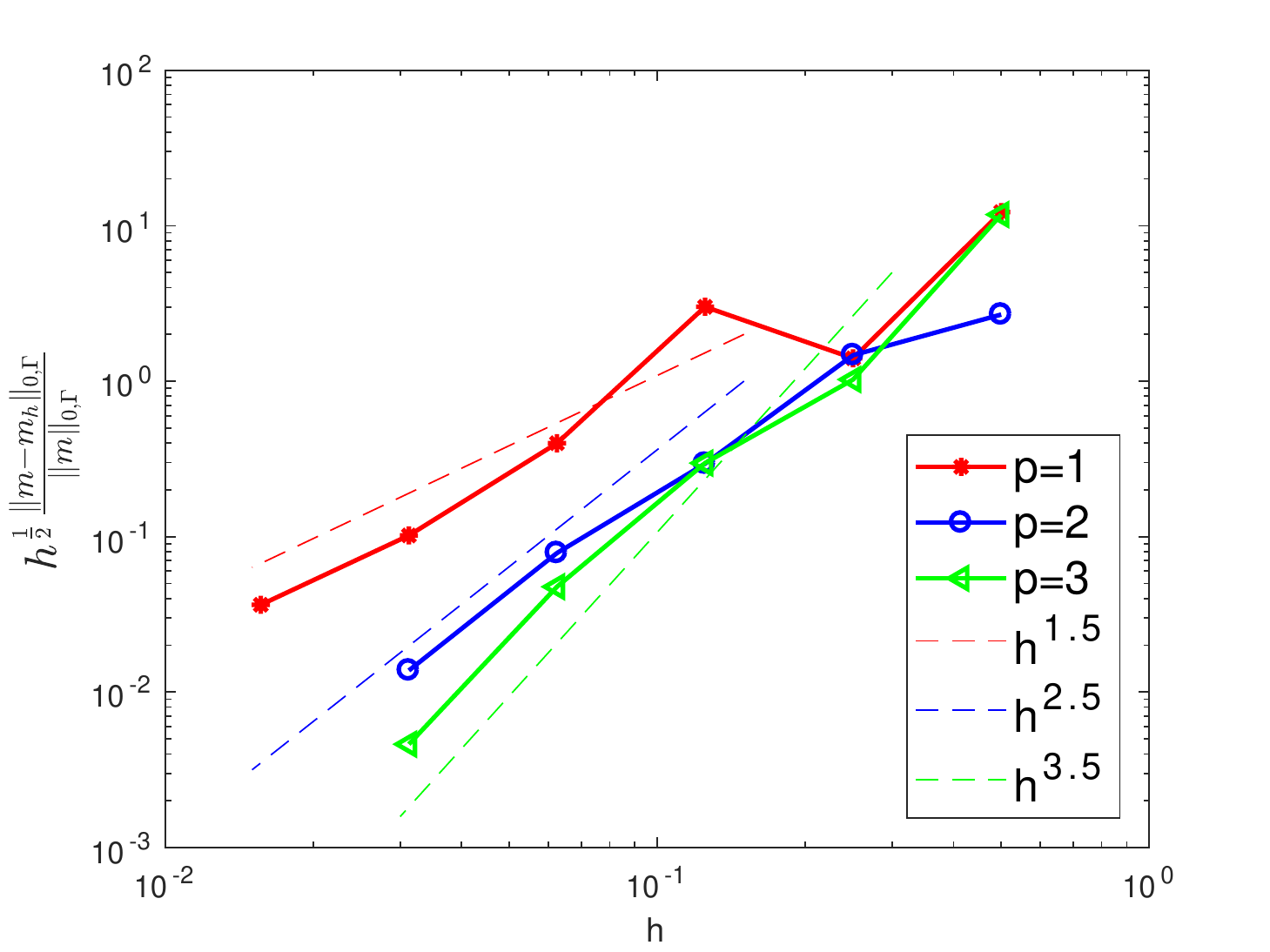}
\includegraphics[width=0.495\textwidth]{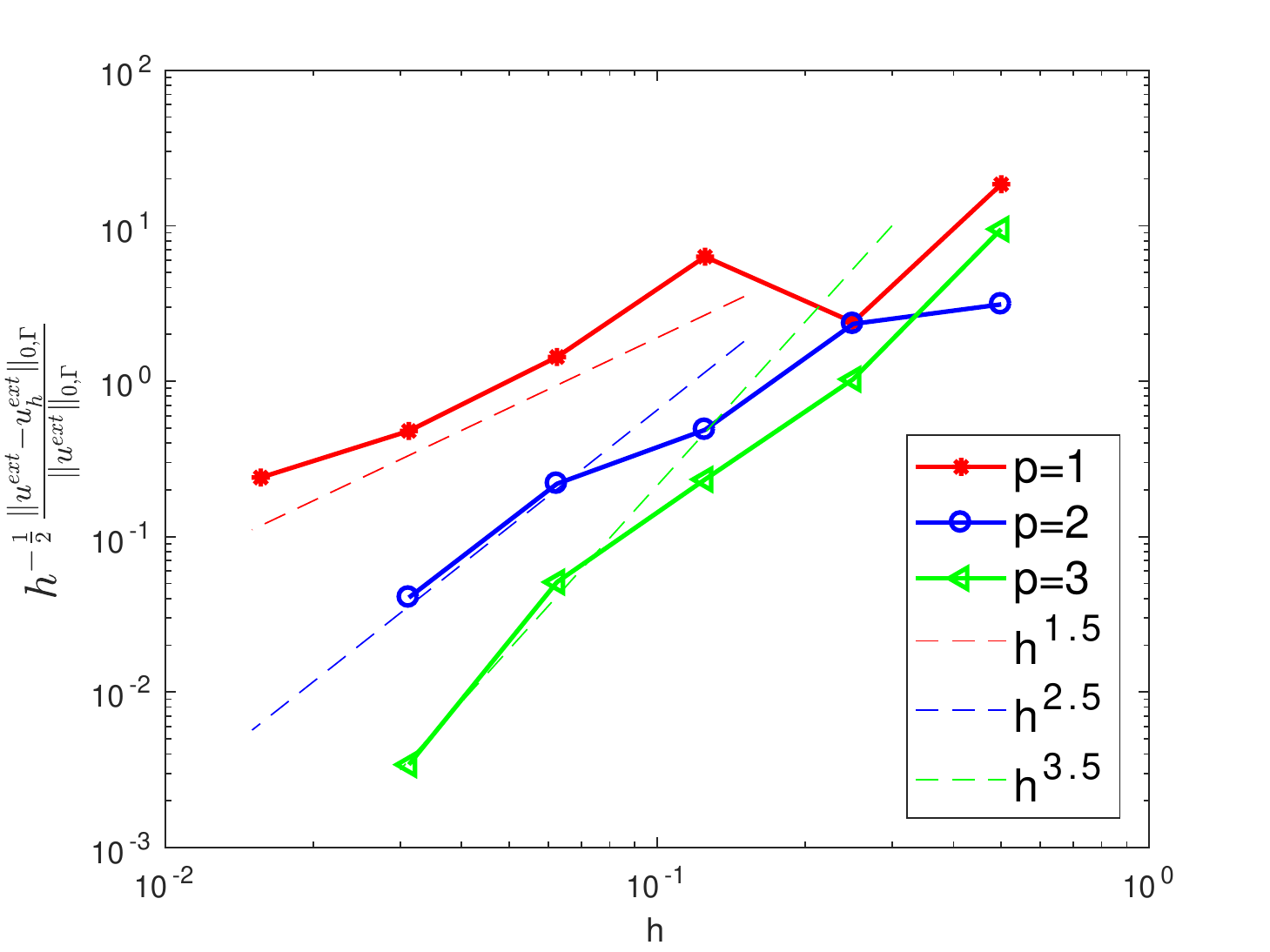}
\end{center}
\caption{The four panels show the four errors in~\eqref{computed:quantities} for $\p=1,2,3$ in~\eqref{approximation:spaces} versus the mesh size $h$. 
The wavenumber $k=6 \sqrt 3  \pi$ is a Dirichlet-Laplace eigenvalue. $\Ad$ is constant. The solution is provided in~\eqref{testcase1:solution}.
\textbf{Top-left panel:} $H^1$ error in~$\Omega$.
\textbf{Top-right panel:} $L^2$ error in~$\Omega$.
\textbf{Bottom-left panel:} $L^2$ error on~$\Gamma$ of the mortar variable times $\h^{\frac{1}{2}}$.
\textbf{Bottom-right panel:} $L^2$ error on~$\Gamma$ times $\h^{-\frac{1}{2}}$.}
\label{fig:testcase_3}
\end{figure}

Here, the initial rate of convergence is degraded by the pollution effect, due to the fact that~$k$ is larger than in the previous two cases. However, with the exception of the~$L^2(\Omega)$ error, the optimal convergence $O(h^p)$ is visible. 
Importantly, it does not matter whether~$k$ is an eigenvalue. The method converges for all choices of the wavenumber, as theoretically predicted.

\paragraph*{Test case 1: $\p$-version.}
Here, we are interested in the $\p$-version of method~\eqref{FEM-BEM}. We consider the test case with explicit solution given in~\eqref{testcase1:solution}. Note that the exact solution is piecewise analytic.
We consider three meshes, namely, a coarse mesh that is then uniformly refined once and twice. 
Figure~\ref{fig:p:version} shows the performance of the $p$-version on these three meshes. 
\begin{figure}[h]
\begin{center}
\includegraphics[width=0.495\textwidth]{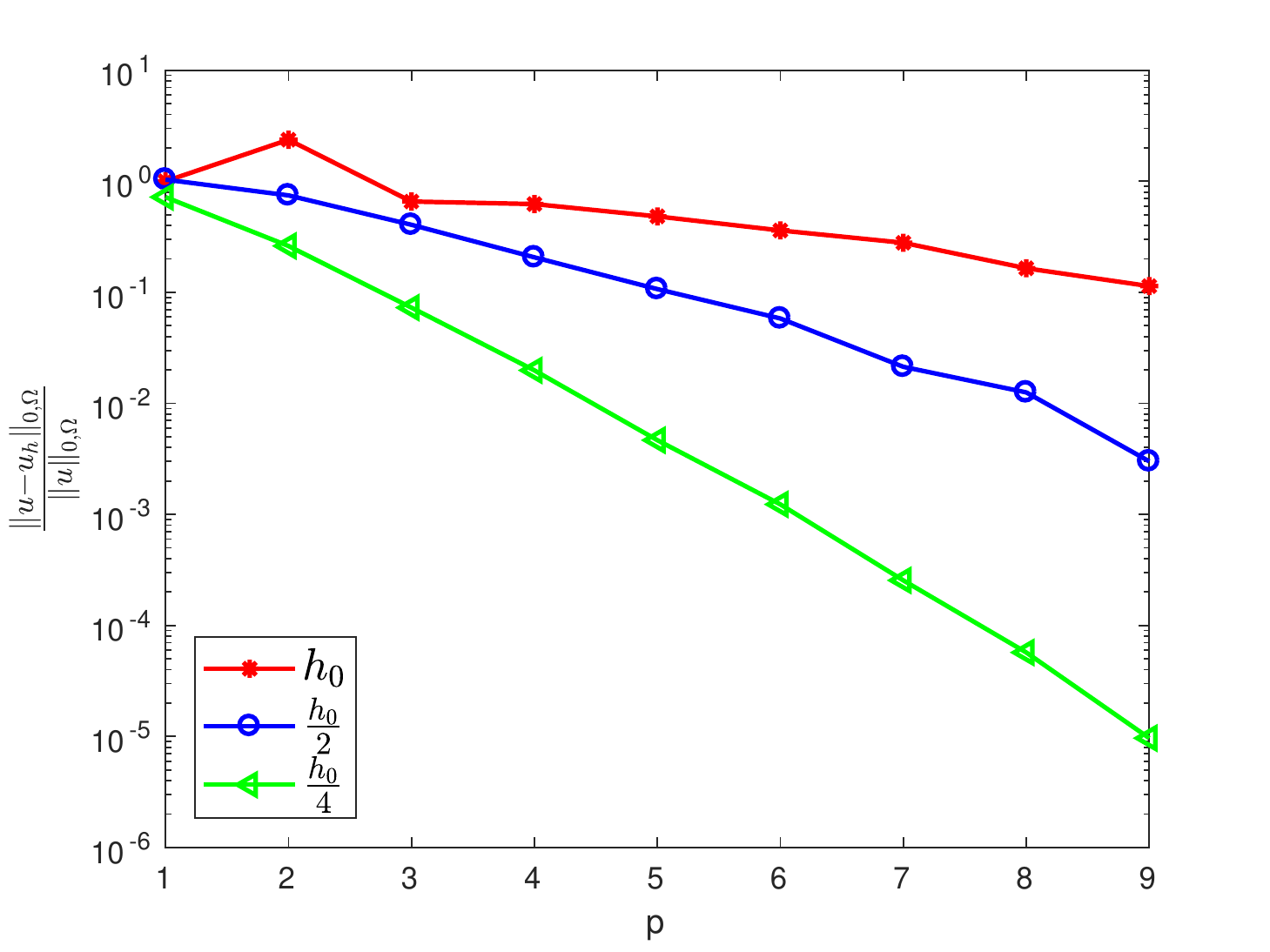}
\includegraphics[width=0.495\textwidth]{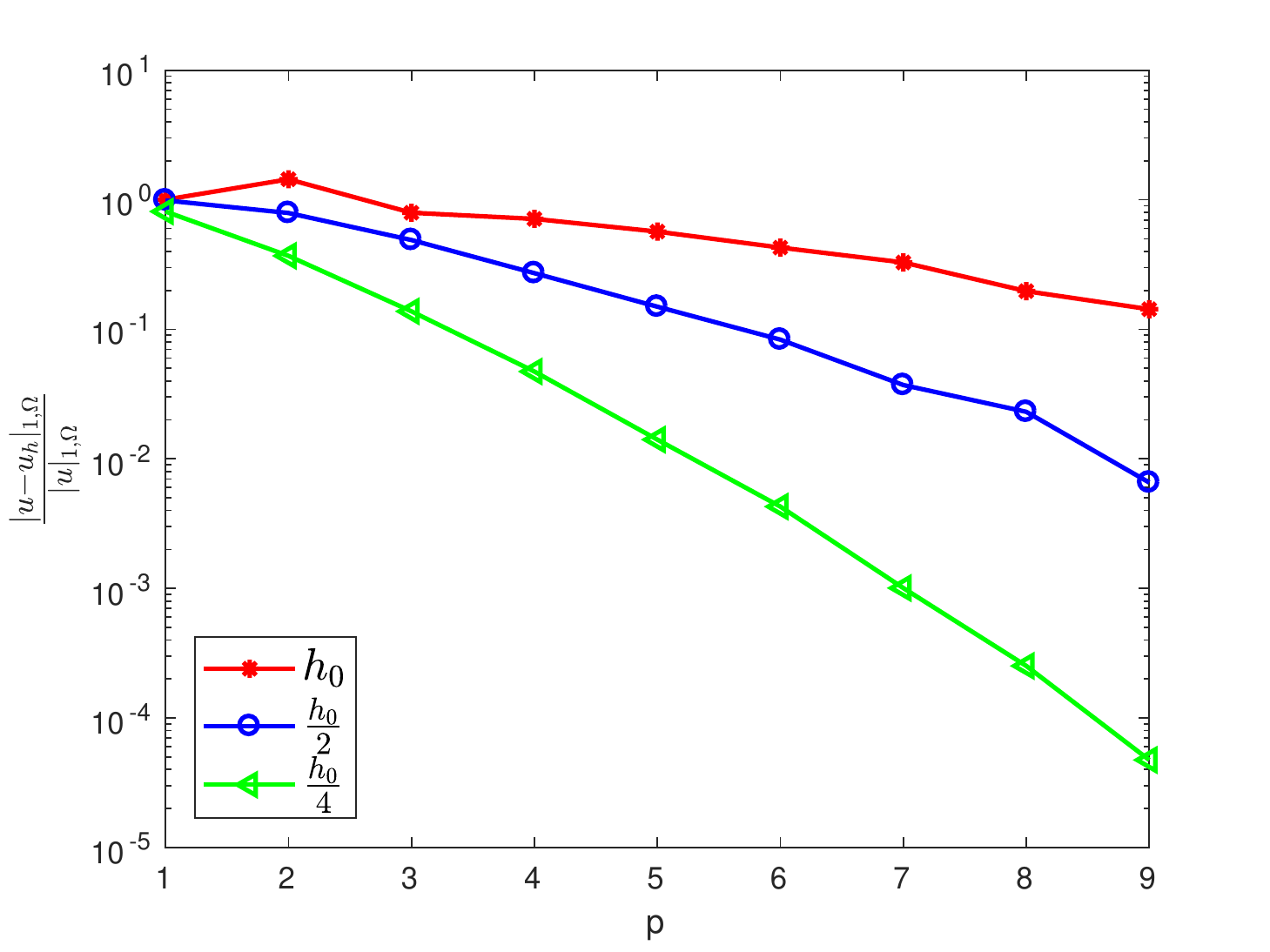}
\includegraphics[width=0.495\textwidth]{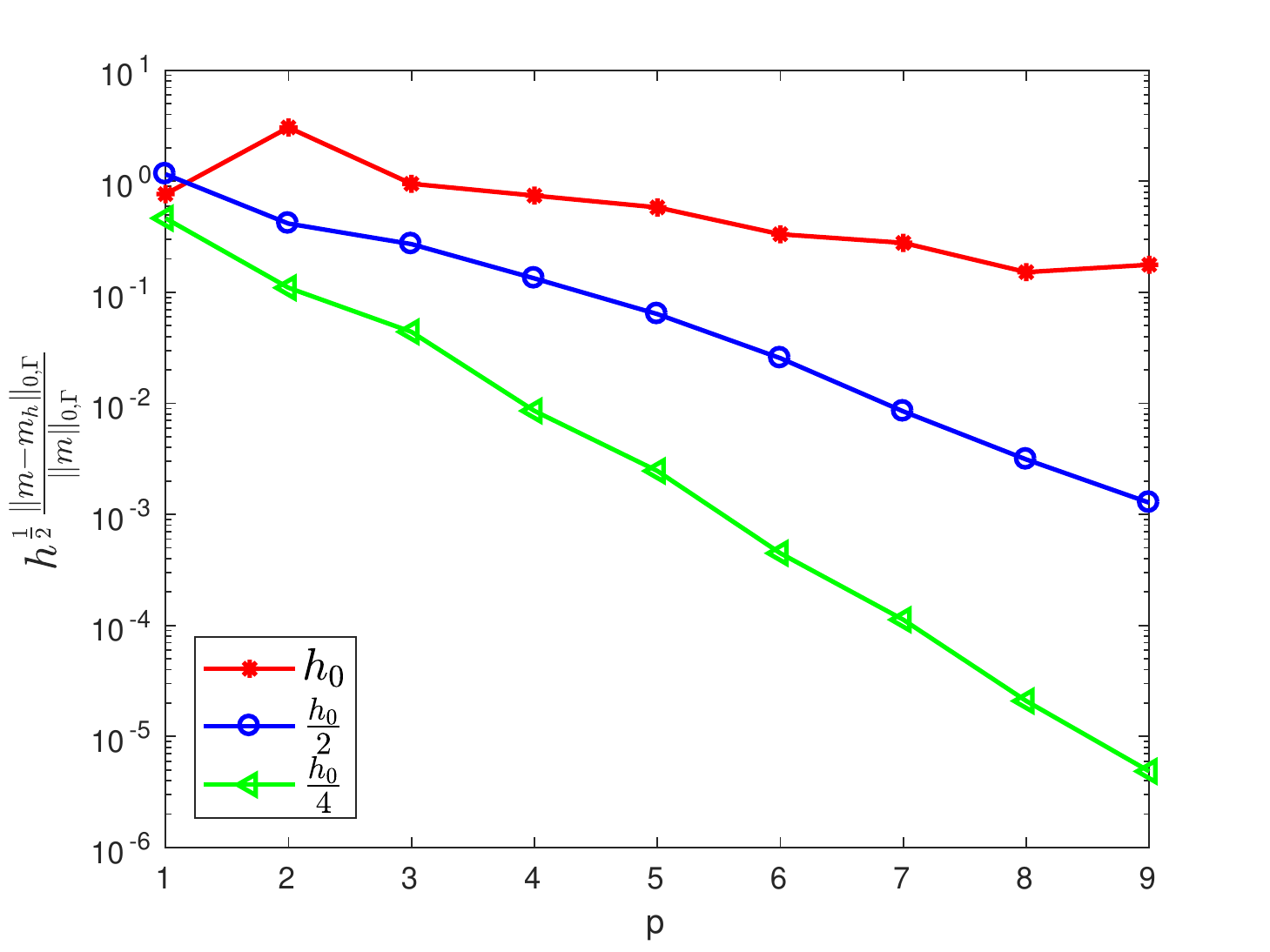}
\includegraphics[width=0.495\textwidth]{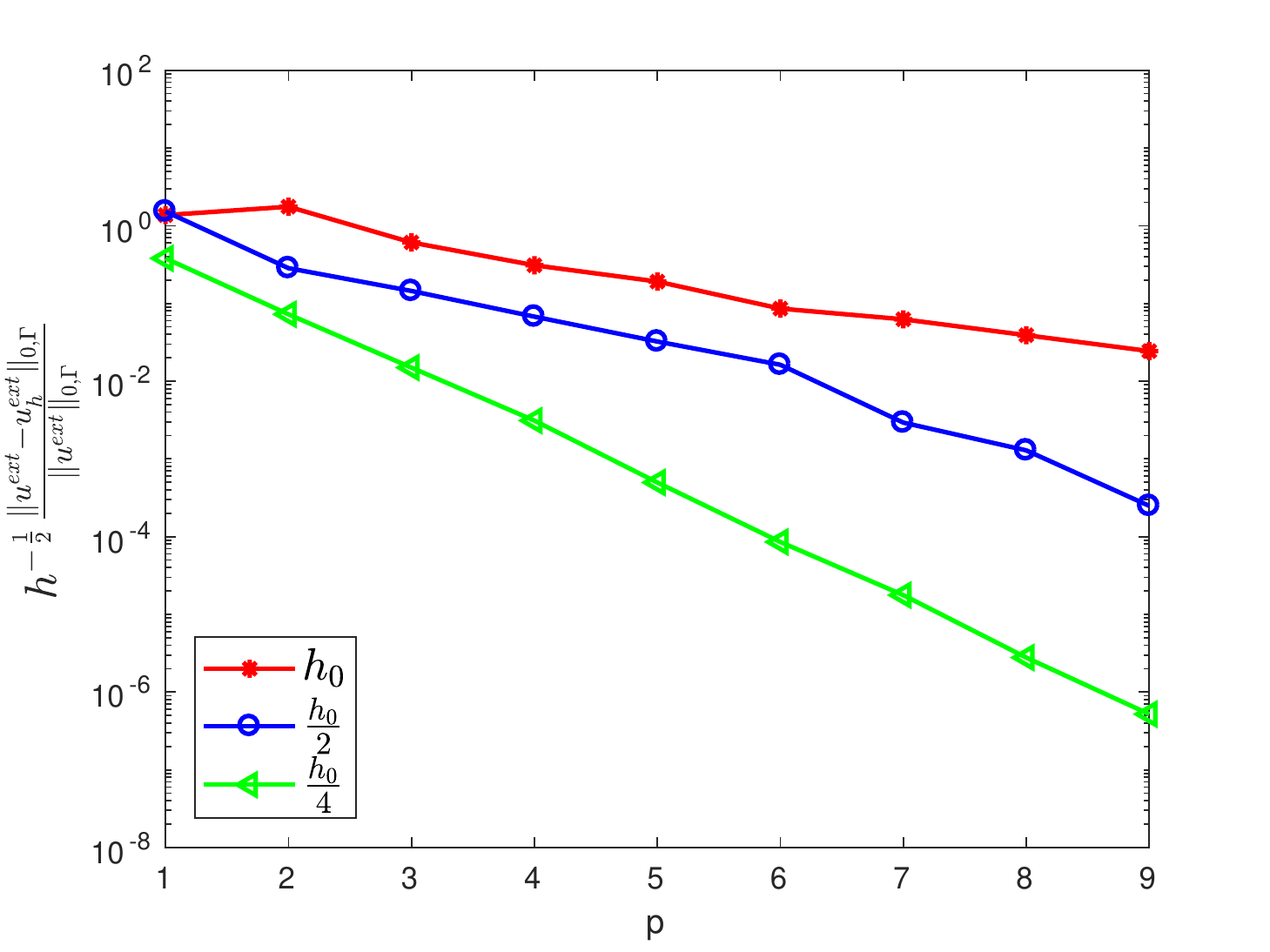}
\end{center}
\caption{The four panels show the four errors in~\eqref{computed:quantities} for different meshes versus the polynomial degree $p$.
The wavenumber $\k=3 \sqrt 3  \pi$ is a Dirichlet-Laplace eigenvalue. $\Ad$ is constant.  The solution is provided in~\eqref{testcase1:solution}.
\textbf{Top-left panel:} $H^1$ error in~$\Omega$.
\textbf{Top-right panel:} $L^2$ error in~$\Omega$.
\textbf{Bottom-left panel:} $L^2$ error on~$\Gamma$ of the mortar variable times $\h^{\frac{1}{2}}$.
\textbf{Bottom-right panel:} $L^2$ error on~$\Gamma$ times $\h^{-\frac{1}{2}}$.}
\label{fig:p:version}
\end{figure}

For all choices of the mesh, the method converge exponentially in terms of the polynomial degree~$\p$. This is reasonable, due to the piecewise smoothness of the solution~\eqref{testcase2:solution}.
We stress that the decay of the error is extremely slow when employing the coarsest mesh.

\paragraph*{Test case 2: $\h$-version.}
Here, we investigate the performance of the $\h$-version of the method for the case of piecewise smooth diffusion coefficient~$\Ad$. In particular, we assume that
\[
\Ad = 
\begin{cases}
2 & \text{in } \widehat \Omega := (-0.2,0.2)^3\\
1 & \text{in } \Omega \setminus \widehat \Omega.\\
\end{cases}
\]
We fix $k = \sqrt{3}\pi$.
We prescribe the solution as 
\begin{equation} \label{testcase2:solution}
u(x,y,z) = 
\begin{cases}
\sin^2 \left( \frac{5\pi}{2} (x-0.2)  \right) \sin^2 \left( \frac{5\pi}{2} (y-0.2)  \right) \sin^2 \left( \frac{5\pi}{2} (z-0.2)  \right) & \text{in } \Omega\\
\frac{e^{i k r}}{r} & \text{in } \Omegap \cup \Gamma. \\
\end{cases}
\end{equation}
We consider meshes that are conforming with respect to the diffusion parameter, i.e., on each element of the tetrahedral mesh, $\Ad$ is constant; see Figure~\ref{fig:discontinuous} for the results.
\begin{figure}[h]
\begin{center}
\includegraphics[width=0.495\textwidth]{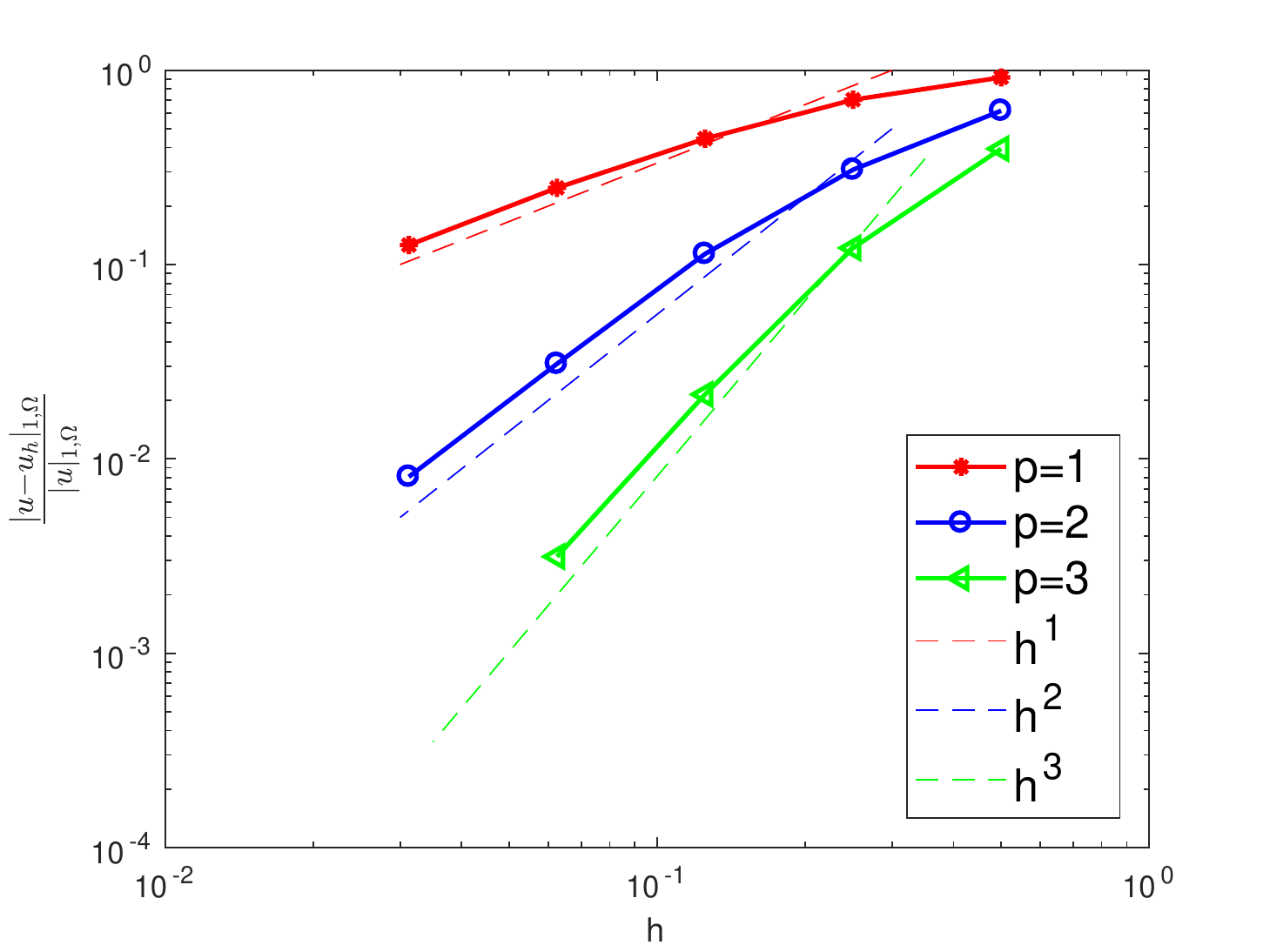}
\includegraphics[width=0.495\textwidth]{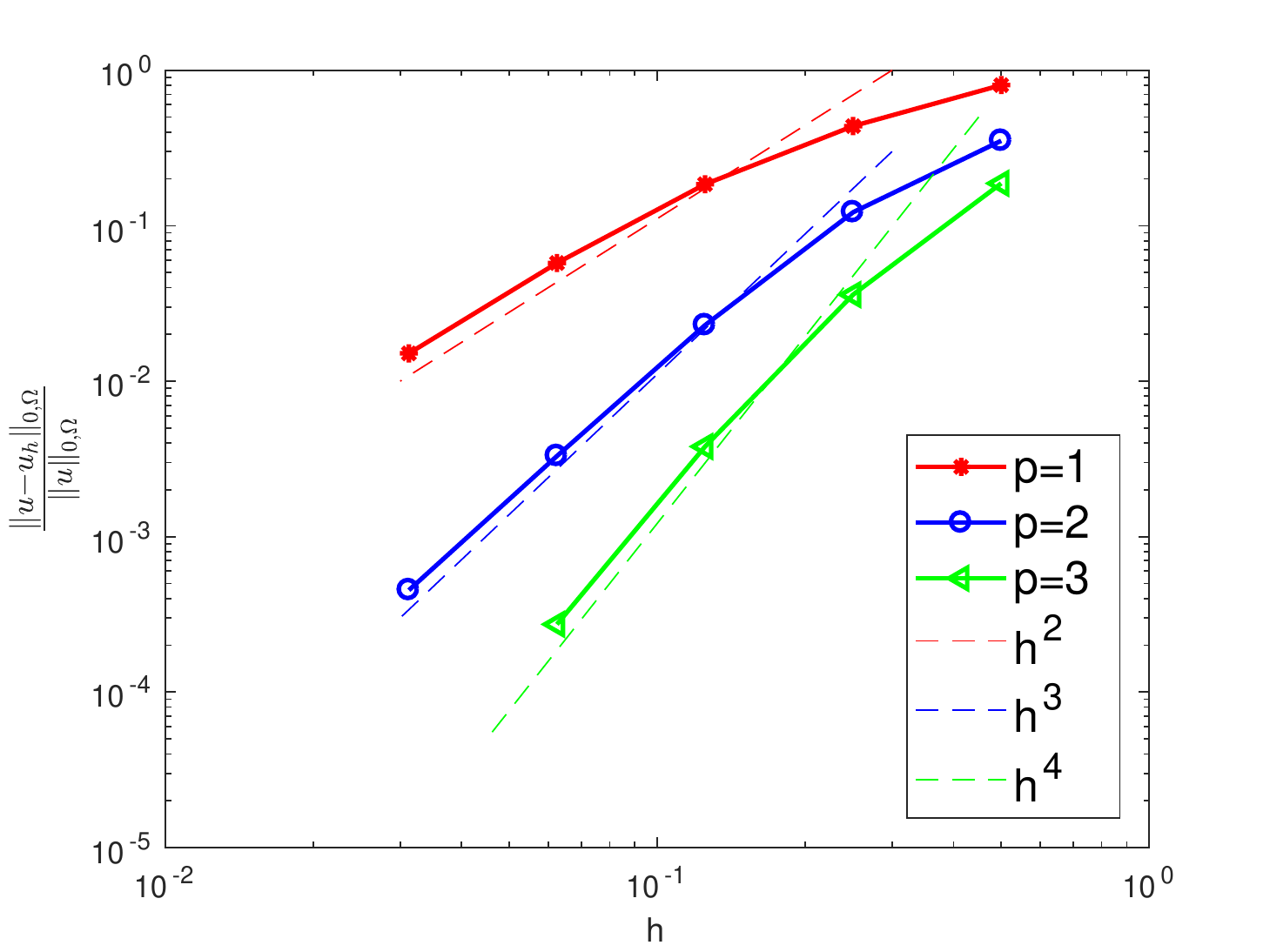}
\includegraphics[width=0.495\textwidth]{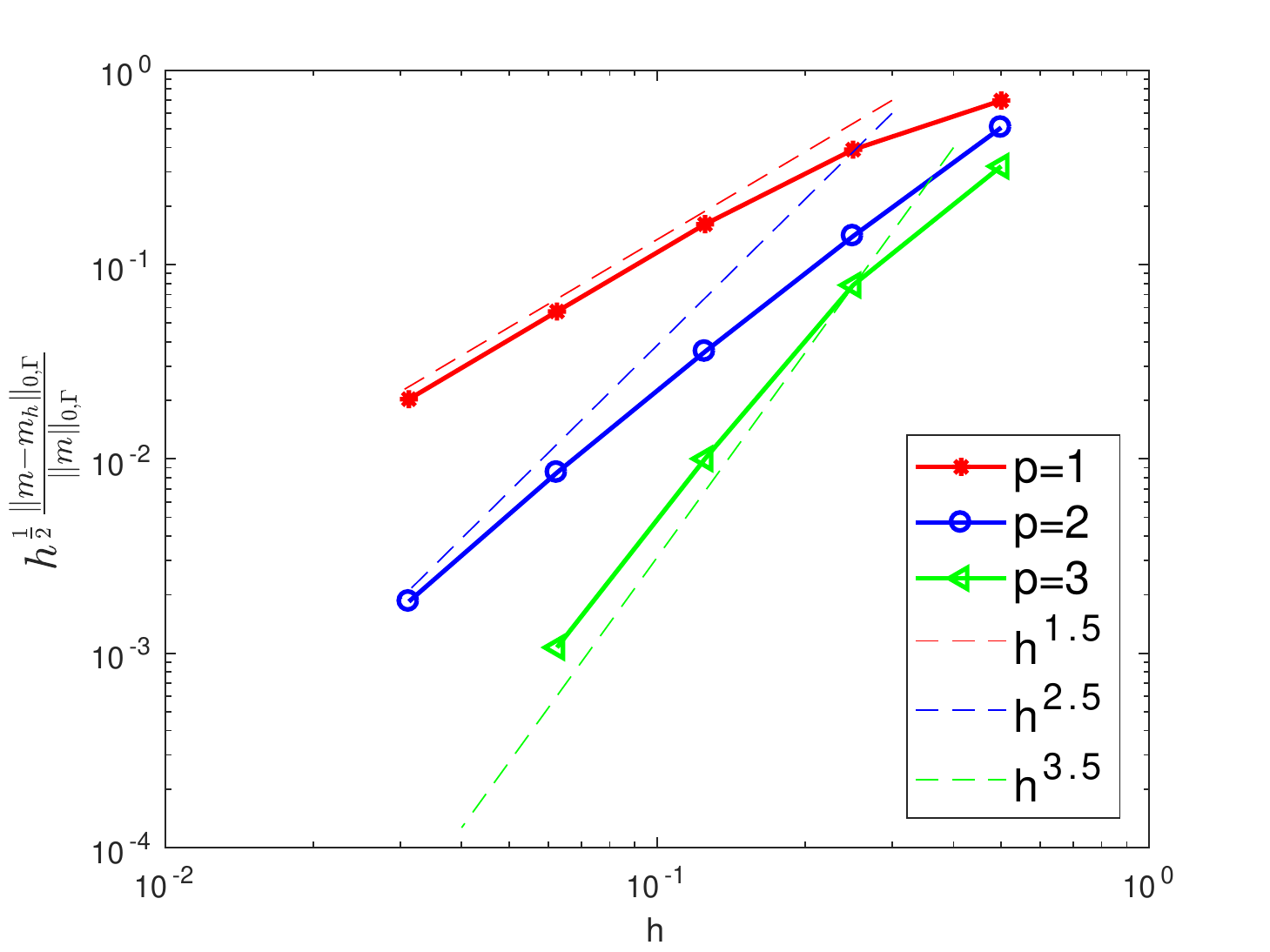}
\includegraphics[width=0.495\textwidth]{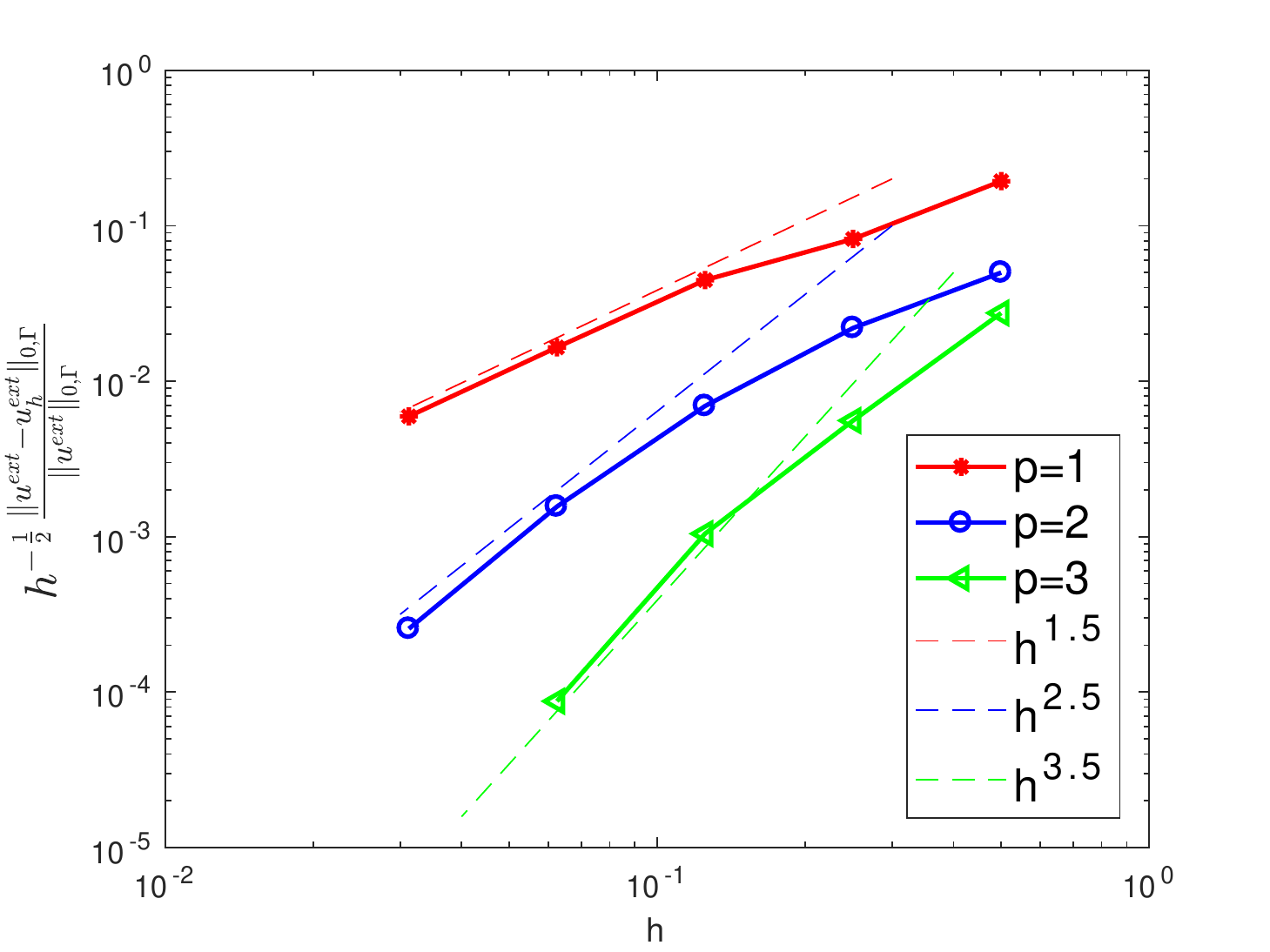}
\end{center}
\caption{We depict the four errors in~\eqref{computed:quantities} in the four panel, for different choices of the mesh, versus the polynomial degree.
The wavenumber $\k=3 \sqrt 3  \pi$ is a Dirichlet-Laplace eigenvalue. The diffusion parameter~$\Ad$ is \emph{piecewise} constant.  The solution is provided in~\eqref{testcase2:solution}.
\textbf{Top-left panel:} $H^1$ error in~$\Omega$.
\textbf{Top-right panel:} $L^2$ error in~$\Omega$.
\textbf{Bottom-left panel:} $L^2$ error on~$\Gamma$ of the mortar variable times $\h^{\frac{1}{2}}$.
\textbf{Bottom-right panel:} $L^2$ error on~$\Gamma$ times $\h^{-\frac{1}{2}}$.}
\label{fig:discontinuous}
\end{figure}

Owing to the fact that the tetrahedral mesh is conforming with respect to the discontinuity of~$\Ad$, the method converges optimally for polynomial degree~$\p=1$, $2$, and~$3$;
see Remark~\ref{remark:diffusion}.

\paragraph*{Implementation issues.}
We briefly discuss here some implementation issues.

First of all, we observe that the linear system stemming from method~\eqref{FEM-BEM} has the following form:
\begin{equation} \label{explicit:system}
  \begin{bmatrix}
    \mathbf A & \mathbf{B_1} & \mathbf{0}\\
    \mathbf{0} & \mathbf{B_2} & \mathbf{B_3} \\
    \mathbf{B_4} & \mathbf{B_5} & \mathbf{B_6}
  \end{bmatrix}
  \begin{bmatrix}
  \overrightarrow{\uh}\\
  \overrightarrow{\mh}\\
  \overrightarrow{\uhext}\\
  \end{bmatrix}
  = \begin{bmatrix}
  \overrightarrow{\f_\h}\\
  \overrightarrow{0}\\
  \overrightarrow{0}\\
  \end{bmatrix},
\end{equation}
where~$\overrightarrow{\uh}$, $\overrightarrow{\mh}$, and~$\overrightarrow{\uhext}$ denote the vector of degrees of freedom associated with~$\uh$, $\mh$, and~$\uhext$, respectively,
where the matrix on the left-hand side is defined by the sesquilinear forms on the left-hand side of~\eqref{FEM-BEM},
and where the vector~$\overrightarrow {\f_\h}$ is defined by the sesquilinear form on the right-hand side of~\eqref{FEM-BEM}.

Importantly, the matrix~$\mathbf A$ is associated with a sesquilinear form with entries in conforming finite element spaces.
All the other matrices, i.e. the $\mathbf{B_i}$ with $i=1,\dots,6$, are associated with sesquilinear form having at least one entry in boundary element spaces.
The assembly of the system is performed by combining the two libraries BEM++~\cite{BEM++paper} and NGSolve~\cite{NGSolve}.

Having at our disposal system~\eqref{explicit:system}, we first write~$\overrightarrow{\uh}$ in terms of~$\overrightarrow{\mh}$. This can be done by means of the LU decomposition that the solver Pardiso \cite{schenk2004solving} provides within NGSolve:
\begin{equation} \label{Schur}
\overrightarrow {\uh } = \mathbf A ^{-1} (\mathbf B_1 \overrightarrow {\mh}).
\end{equation}
We substitute~$\overrightarrow {\uh}$ in the second and third ``lines'' of the system.
Note that the resulting system is considerably smaller than the original one, especially for high polynomial degree. In fact, we have a system associated only with boundary element degrees of freedom. 

The solution of such system is successively computed with GMRES, preconditioned with an approximate LU decomposition based on the ${\mathcal H}$-matrix arithmetic provided by the library H2Lib~\cite{h2lib}. 
We have fixed the tolerance~$10^{-8}$ and a maximum number of~$2000$ iterations.

Since we showed that it is possible to proceed by Schur complement~\eqref{Schur}, it is clear that the computational effort needed to solve system~\eqref{explicit:system}
is comparable to that needed in a formulation without the additional mortar variable.

\begin{remark} \label{remark:transmission:discrete}
As a side remark, we observe that the system~\eqref{explicit:system} can be also rewritten by solving the second line in terms of~$\mh$ as
\[
\mh = \mathbf{B}_{\mathbf 2}^{-1} \mathbf{B_1} \uhext,
\]
and substitute~$\mh$ in the first and third equations, getting the discrete version of the transmission problem discussed in Remark~\ref{remark:transmission:continuous}.
\eremk
\end{remark}

\section{Conclusions} \label{section:conclusions}
We have presented a FEM-BEM coupling strategy for time harmonic
acoustic scattering in media with variable speed of sound.
The continuous problem has been formulated with the aid of an auxiliary mortar variable representing an impedance trace.
The novelty of this approach relies in this choice of the mortar variable, which leads to a block-structured system, with subblocks that are invertible, for any arbitrary choice of the coupling boundary.
The flexibility in the choice of the coupling boundary can be exploited to facilitate the meshing or the realization of relevant boundary integral operators.
The invertibility of the FEM and the BEM subblocks allows for the use of existing computationally tools for their numerical realization.
Stability and convergence of this FEM-BEM mortar method have been investigated theoretically and numerically.

\section*{Acknowledgements}
All authors have been funded by the Austrian Science Fund (FWF)
through the project F 65. M.~Melenk has been funded by the FWF
also through the project W1245. I. Perugia and A. Rieder have been
funded by the FWF also through the project P 29197-N32.

\appendix
\section{$k$-explicit continuity and G{\aa}rding inequality for analytic $\Gamma$}
\label{appendix:garding}

In this appendix, we consider the case of 
  analytic boundary $\Gamma$. For this case, based on the
  characterization of the difference
  operators $\Vk - \Vz$, $\Kk - \Kz$, $\Kprimek - \Kz'$, and $\Wk -
  \Wz$ established in~\cite{melenk2012mapping},
  we prove a $k$-explicit continuity assertion as well as a 
G{\aa}rding inequality in Theorem~\ref{thm:k-explicit-garding}.

Introduce the class of analytic functions 
\begin{align*}
\boldsymbol{\mathcal A}(C_1,\gamma_1,\Omega) &:= \{v \in C^\infty(\Omega)\,|\,
                                    \|\nabla^n v\|_{L^2(\Omega)} \leq
                                    C_1 \gamma_1^n \max\{n+1,k\}^{n+1}
\quad \forall n \in {\mathbb N}_0\},  
\end{align*}
where $|\nabla^n v|^2 = \sum_{\alpha \in {\mathbb N}_0^3: |\alpha| = n} 
\frac{n!}{\alpha!} 
|D^\alpha v|^2$. 

The following lemma decomposes the operators $\Vk - \Vz$, $\Kk-\Kz$, $\Kprimek-\Kz^\prime$, $\Wk-\Wz$ into 
a part that has a finite shift property and a part that maps into the
class of analytic functions. 

\begin{lem}
\label{lemma:decomposition}
Let $\Gamma$ be analytic and $k \ge k_0 > 0$.
Then there are
bounded linear operators 
$\SV$, $\SK$, $\SKprime$, $\SW$ and linear maps 
$\AVtilde: H^{-\frac32}(\Gamma) \rightarrow C^\infty(\overline\Omega)$, 
$\AKtilde: H^{-\frac12}(\Gamma) \rightarrow C^\infty(\overline\Omega)$
such that 
\begin{subequations}
\label{eq:differences}
\begin{align}
\label{eq:differences-a}
\Vk - \Vz & = \SV+ \gammazint \AVtilde, \\
\label{eq:differences-b}
\Kprimek - \Kz' & = \SKprime + \gammaoint \AVtilde, \\ 
\label{eq:differences-c}
\Kk - \Kz & = \SK + \gammazint \AKtilde,  \\
\label{eq:differences-d}
\Wk - \Wz & = \SW - \gammaoint \AKtilde.
\end{align}
\end{subequations}
For $s \ge -1$ 
the operators $\SV$, $\SK$, $\SKprime$, $\SW$, $\AVtilde$, $\AKtilde$ 
have, for constants $C_{s,s'}$, $C_{\V}$, $C_{\K}$, $\gamma_{\V}$,
  $\gamma_{\K}> 0$ independent of $k \ge k_0$, the 
mapping properties 
\begin{subequations}
\label{eq:estimates-S}
\begin{align}
\label{eq:estimates-S-a}
\|\SV \|_{H^{-\frac12+s'}(\Gamma)\leftarrow H^{-\frac12+s}(\Gamma)} 
\leq C_{s,s'} k^{-(1+s-s')}, \qquad 
1/2 < s' \leq  s+3, \\
\label{eq:estimates-S-b}
\|\SKprime \|_{H^{-\frac32+s'}(\Gamma)\leftarrow H^{-\frac12+s}(\Gamma)}
\leq C_{s,s'} k^{-(1+s-s')},  \qquad 
3/2 < s' \leq  s +3, \\
\label{eq:estimates-S-c}
\|\SK \|_{H^{-\frac12+s'}(\Gamma)\leftarrow H^{+\frac12+s}(\Gamma)} 
\leq C_{s,s'} k^{-(1+s-s')}, \qquad 
1/2 < s' \leq  s+3, \\
\label{eq:estimates-S-d}
\|\SW \|_{H^{-\frac32+s'}(\Gamma)\leftarrow H^{+\frac12+s}(\Gamma)} 
\leq C_{s,s'} k^{-(1+s-s')}, \qquad 
3/2 < s' \leq  s +3, \\ 
\label{eq:estimates-S-e}
\AVtilde \varphi   \in \boldsymbol {\mathcal A}(C_{\V} \|\varphi\|_{H^{-\frac{3}{2}}(\Gamma)}, \gamma_{\V}, \Omega) 
\qquad \forall \varphi \in H^{-\frac{3}{2}}(\Gamma), \\
\label{eq:estimates-S-f}
\AKtilde \psi   \in \boldsymbol {\mathcal A}(C_{\K} \|\psi\|_{H^{-\frac{1}{2}}(\Gamma)}, \gamma_{\K}, \Omega)
\qquad \forall \psi \in H^{-\frac{1}{2}}(\Gamma). 
\end{align}
\end{subequations}
\end{lem}
\begin{proof}
We stress that the functions 
$\AVtilde \varphi$ and 
$\AKtilde \psi$  are analytic in $\overline{\Omega}$ so that, when taking 
their traces in~\eqref{eq:differences}, the traces are analytic on~$\Gamma$. 

\cite[Thms.~{5.3}, {5.4}]{melenk2012mapping} assert for the potentials 
the representations 
$\Vtildek-\widetilde \V_0 = \SVtilde + \AVtilde$ and 
$\Ktildek-\widetilde \K_0 = \SVtilde + \AKtilde$, where 
the linear operators 
$\SVtilde: H^{-\frac32}(\Gamma) \rightarrow H^{3}(\Omega)$,  
$\SKtilde: H^{-\frac12}(\Gamma) \rightarrow H^{3}(\Omega)$, 
$\AVtilde:H^{-\frac32}(\Gamma) \rightarrow C^\infty(\overline\Omega)$, 
$\AKtilde:H^{-\frac12}(\Gamma) \rightarrow C^\infty(\overline\Omega)$
have the following mapping properties: 
for $s \ge -1$ and for $0 \leq s' \leq s+3$ 
and a constant $C > 0$ independent of $k$ there holds 
\begin{equation}
\|\SVtilde\|_{H^{s'}(\Omega) \leftarrow H^{-\frac12+s}(\Gamma)} 
\leq C k^{-1(1+s-s')}, 
\qquad   
\|\SKtilde\|_{H^{s'}(\Omega) \leftarrow H^{\frac12+s}(\Gamma)} 
\leq C k^{-1(1+s-s')}.
\end{equation}
For constants $C_{\V}$, $C_{\K}$, $\gamma_{\V}$,
  $\gamma_{\K}$ 
independent of $k$, one has 
$\AVtilde \varphi \in \boldsymbol {\mathcal A}(C_{\V} \|\varphi\|_{-\frac{3}{2},\Gamma}, \gamma_{\V},\Omega)$ and 
$\AKtilde \psi \in \boldsymbol {\mathcal A}(C_{\K} \|\psi\|_{-\frac{1}{2},\Gamma}, \gamma_{\K},\Omega)$. 
Applying the trace operator $\gammazint$ one obtains the representations
\eqref{eq:differences-a} and \eqref{eq:differences-c},
where the linear operators $\SV$, $\SKprime$, $\SK$, $\SW$ and the 
operators $\AVtilde$, $\AKtilde$ satisfy~\eqref{eq:estimates-S}.
The representations \eqref{eq:differences-b} and
  \eqref{eq:differences-d} are obtained similarly with the aid of $\gammaoint$.
\end{proof}

Based on the representations~\eqref{eq:differences} of the
  difference operators, we can prove the following $k$-explicit continuity and G{\aa}rding inequality.

\begin{thm}[$k$-explicit continuity and G{\aa}rding inequality for analytic $\Gamma$]
\label{thm:k-explicit-garding} 
Let $\Gamma$ be analytic and $k \ge k_0 > 0$. 
Then there are bounded linear operators
\begin{align*}
 \Theta_{f,\uext,\vext}\colon& H^{-\frac{1}{2}}(\Gamma) \rightarrow H^{\frac{1}{2}}(\Gamma)
& \Theta_{f,m,\vext}\colon & H^{-\frac{3}{2}}(\Gamma) \rightarrow H^{\frac{1}{2}}(\Gamma), \\
\Theta_{f,\uext,\lambda}\colon& H^{-\frac{1}{2}}(\Gamma) \rightarrow H^{\frac{3}{2}}(\Gamma), 
 & \Theta_{f,m,\lambda} \colon & H^{-\frac{3}{2}}(\Gamma) \rightarrow H^{\frac{3}{2}}(\Gamma) 
\end{align*}
(the subscript $f$ stands for ``finite shift properties'')
and
linear operators $\AVtilde':H^{-\frac{3}{2}}(\Gamma)
\rightarrow C^\infty(\overline{\Omega})$
$\AKtilde':H^{-\frac{1}{2}}(\Gamma)
\rightarrow C^\infty(\overline{\Omega})$,
such that, setting
\begin{subequations}
\label{eq:theta_A}
\begin{align}
\Theta_{{\mathcal A},\uext,\vext} & :=  -\gammaoint \AKtilde' +
i k \left(\gammazint \AKtilde' + \gammaoint \AVtilde'\right)  
+ k^2 \gammazint \AVtilde', 
\\
\Theta_{{\mathcal A},m,\vext} & :=  \gammaoint\AVtilde' - ik \gammazint \AVtilde', 
\\
\Theta_{{\mathcal A},\uext,\lambda} & := - (\gammazint \AKtilde' - ik \gammazint \AVtilde'), 
\\
\Theta_{{\mathcal A},m,\lambda} & := \gammazint \AVtilde '
\end{align}
\end{subequations}
(the subscript $\A$ stands for ``analytic'')
and
  \begin{align*}
\Theta_{\uext,\vext} & := 
\Theta_{f,\uext,\vext}  + \Theta_{{\mathcal A},\uext,\vext}, 
&\Theta_{m,\vext} & := 
\Theta_{f,m,\vext} + \Theta_{{\mathcal A},m,\vext}, \\
\Theta_{\uext,\lambda} & := 
\Theta_{f,\uext,\lambda} + \Theta_{{\mathcal A},\uext,\lambda}, 
&\Theta_{m,\lambda} & := 
\Theta_{f,m,\lambda}+ \Theta_{{\mathcal A},m,\lambda},
  \end{align*}
the following holds true: 

\begin{enumerate}[(i)]
\item 
\label{item:thm:k-explicit-garding-i}
For each $s \ge -1$, there holds 
\begin{align*}
\|\Theta_{f,\uext,\vext} \|_{H^{-\frac{3}{2}+s'}(\Gamma)\leftarrow H^{\frac{1}{2}+s}(\Gamma)} 
\leq C k^{-(1+s-s')}, \qquad 
3/2 < s' \leq s+3, \\
\|\Theta_{f,m,\vext} \|_{H^{-\frac{3}{2}+s'}(\Gamma)\leftarrow H^{-\frac{1}{2}+s}(\Gamma)}
\leq C k^{-(1+s-s')},  \qquad 
3/2 < s' \leq s +3, \\
\|\Theta_{f,\uext,\lambda}\|_{H^{-\frac{1}{2}+s'}(\Gamma)\leftarrow H^{\frac{1}{2} +s}(\Gamma)} 
\leq C k^{-(1+s-s')}, \qquad 
1/2 < s' \leq s+3, \\
\|\Theta_{f,m,\lambda} \|_{H^{-\frac{1}{2}+s'}(\Gamma)\leftarrow H^{-\frac{1}{2}+s}(\Gamma)} 
\leq C k^{-(1+s-s')}, \qquad 
1/2 < s' \leq  s +3,\\ 
\AVtilde' \varphi \in \boldsymbol {\mathcal A} \bigl(C_{\V} \|\varphi\|_{H^{-\frac{3}{2}}(\Gamma)},\gamma_{\V},\Omega\bigr) 
\quad\qquad \forall \varphi \in H^{-\frac{3}{2}}(\Gamma), \\
\AKtilde' \psi \in \boldsymbol {\mathcal A}\bigl(C_{\K} \|\psi\|_{H^{-\frac{1}{2}}(\Gamma)},\gamma_{\K},\Omega\bigr) 
\quad\qquad \forall \psi \in H^{-\frac{1}{2}}(\Gamma). 
\end{align*}
The constant $C$ depends only on $s$, $s'$, $\Gamma$, and $k_0$. The constants 
$C_{\V}$, $C_{\K}$, $\gamma_{\V}$, $\gamma_{\K}$ depend only on $k_0$ and $\Gamma$. 
\item 
\label{item:thm:k-explicit-garding-ii}
For a constant $c > 0$ depending only on $k_0$ and $\Gamma$, 
the sesquilinear form $\T(\cdot,\cdot)$ defined in~\eqref{form:T:for:Helmholtz} satisfies the G{\aa}rding inequality 
\begin{align*}
&\Re \left( \T\bigl( (v,\lambda,\vext),(v,\lambda,\vext)\bigr) 
+ \langle (v,\lambda,\vext),\Theta (v,\lambda,\vext)\rangle 
\right)\\
&\qquad\qquad\qquad \ge c  \|\Ad^{\frac12} \nabla u\|^2_{0,\Omega} + 
k^2\| n\, u\|^2_{0,\Omega} + 
\|\lambda\|^2_{-\frac{1}{2},\Gamma} + 
\|\vext\|^2_{\frac{1}{2},\Gamma}, 
\end{align*}
where the linear operator $\Theta$ is given by 
\begin{align*}
 & \langle (u,m,\uext),\Theta (v,\lambda,\vext)\rangle
  =  2 ((\k)^2 u,v)_{0,\Omega} \\
  \nonumber
  &\qquad\qquad\qquad
\mbox{} -
\langle \uext, \Theta_{\uext,\vext}\vext\rangle -
\langle m, \Theta_{m,\vext} \vext\rangle -
\langle \uext, \Theta_{\uext,\lambda} \lambda \rangle - 
\langle m, \Theta_{m,\lambda}\lambda \rangle.
\end{align*}
\item 
\label{item:thm:k-explicit-garding-iii}
For a constant $C_{cont} > 0$ depending only on $k_0$ and $\Gamma$, the 
sesquilinear form $\T(\cdot,\cdot)$ defined
  in~\eqref{form:T:for:Helmholtz} satisfies the continuity estimate
\begin{align*}
& \Bigl| \T\bigl( (u,m,\uext),(v,\lambda,\vext)\bigr) 
+ \langle (u,m,\uext),\Theta (v,\lambda,\vext)\rangle 
-
\langle (u,m,\uext),\widetilde\Theta (v,\lambda,\vext)\rangle \Big|\\
  &\qquad\qquad\qquad
\leq C_{cont} 
\Bigl\{ \|\Ad^{\frac12} \nabla u\|^2_{0,\Omega} + k \|u\|^2_{0,\Gamma} + \|m\|^2_{-\frac{1}{2},\Gamma} + \|\uext\|^2_{\frac{1}{2},\Gamma} 
\Bigr\}^{\frac12} \\
  &\qquad\qquad\qquad\qquad
    \times\Bigl\{ \|\Ad^{\frac12} \nabla v\|^2_{0,\Omega} + k \|v\|^2_{0,\Gamma} + \|\lambda\|^2_{-\frac{1}{2},\Gamma} + \|\vext\|^2_{\frac{1}{2},\Gamma} 
\Bigr\}^{\frac12},
\end{align*}
with a linear operator $\widetilde\Theta$ given by
  \begin{align*}
  &  \langle (u,m,\uext),\widetilde\Theta (v,\lambda,\vext)\rangle=
    \langle \uext, \widetilde{\widetilde \Theta}_{f,\uext,\vext} \vext\rangle\\
    &\qquad\qquad\qquad
      +\langle \uext, \widetilde \Theta_{f,\uext,\vext} \vext\rangle
      +\langle m, \widetilde\Theta_{f,m,\vext} \vext\rangle
      +\langle \uext, \widetilde\Theta_{f,\uext,\lambda} \lambda\rangle,
\end{align*}
where, for $s \ge 0$ and a constant $C>0$ independent  of $k$,
\begin{subequations}
\label{eq:tildetheta}
\begin{align}
\label{eq:tildetheta-a}
\|\widetilde{\widetilde\Theta}_{f,\uext,\vext}\|_{H^{\frac{1}{2} + s}(\Gamma)\leftarrow
H^{-\frac{1}{2} + s}(\Gamma)} & \leq C k^{2}, \\
\label{eq:tildetheta-b}
\|\widetilde\Theta_{f,\uext,\vext} \|_{H^{\frac{1}{2}+s}(\Gamma) \leftarrow H^{\frac{1}{2}+s}(\Gamma)}  & \leq C k, \\
\label{eq:tildetheta-c}
\|\widetilde\Theta_{f,m,\vext}\|_{H^{\frac{1}{2}+s}(\Gamma)\leftarrow H^{-\frac{1}{2}+s}(\Gamma)}   & \leq C k,\\
\label{eq:tildetheta-d}
\|\widetilde\Theta_{f,\uext,\lambda} \|_{H^{\frac{1}{2}+s}(\Gamma)\leftarrow H^{-\frac{1}{2}+s}(\Gamma)} &\leq C k. 
\end{align}
\end{subequations}
\end{enumerate}
\end{thm}
\begin{proof}
The proof of Theorem~\ref{theorem:Garding:inequality} 
shows in~\eqref{eq:T1T2T3} that the sesquilinear form $\T(\cdot,\cdot)$ can be written as the 
sum of the sesquilinear form $T_1(\cdot,\cdot)$, $T_2(\cdot,\cdot)$,
$T_3(\cdot,\cdot)$.\medskip

\emph{Proof of~\eqref{item:thm:k-explicit-garding-ii}:}
As in the proof Theorem~\ref{theorem:Garding:inequality}, 
we have
\begin{equation*}
  \begin{split}
    \Re \left(T_1\big((v,\lambda,\vext),(v,\lambda,\vext)\bigr)\right) &\gtrsim 
\|\Ad^{\frac{1}{2}} \nabla v\|^2_{0,\Omega} + \|\lambda\|^2_{-\frac{1}{2},\Gamma}
+ \|\vext\|^2_{\frac{1}{2},\Gamma},\\
\Re \left(T_2\big((v,\lambda,\vext),(v,\lambda,\vext)\bigr)\right) &= 0.
  \end{split}
\end{equation*}
We therefore 
focus on $T_3(\cdot,\cdot)$. We rewrite $T_3(\cdot,\cdot)$ as 
\begin{align*}
  & T_3\bigl((u,m,\uext),(v,\lambda,\vext)\bigr)
     = 
\Bigl\{ 
-  ( (\k)^2 u, v)_{0,\Omega}  
\Bigr\}
\\ & 
\qquad+ \Bigl\{ 
 \langle \uext, (\Wk-\Wz)^* \vext\rangle 
- ik \langle   \uext, (  \Kk- \Kz  )^*\vext\rangle \\
&\qquad\qquad - ik \langle  \uext, (\Kprimek -\Kz^\prime)^*\vext\rangle  
+ k^2 \langle \uext, (\Vk -\Vz)^* \vext \rangle 
\Bigr\}
\\ & 
\qquad + \Bigl\{ 
 \langle m, (\Kprimek -\Kz^\prime)^* \vext\rangle 
+ i k \langle  \m, (\Vk -\Vz)^*, \vtilde   \rangle   
\Bigr\}
\\ & 
\qquad + \Bigl\{ 
-  \langle \uext, (\Kk-\Kz)^* \lambda \rangle 
- ik \langle  \uext,  (\Vk - \Vz)^* \lambda \rangle  
\Bigr\}
+ 
\Bigl\{ 
\langle \m, (\Vk - \Vz)^*\lambda\rangle
\Bigr\}  \\
& 
=: - ((\k)^2 u, v)_{0,\Omega} + 
\langle \uext, \Theta_{\uext,\vext} \vext\rangle + 
\langle m, \Theta_{m,\vext} \vext\rangle + 
\langle \uext, \Theta_{\uext,\lambda} \lambda \rangle + 
              \langle m, \Theta_{m,\lambda} \lambda \rangle\\
  &
   =((\k)^2 u, v)_{0,\Omega}-\langle (u,m,\uext),\Theta (v,\lambda,\vext)\rangle.
\end{align*}
Notice that
\begin{subequations}
\label{eq:th}
\begin{align}
 \label{eq:th-a} 
  \Theta_{\uext,\vext}
   &= (\Wk-\Wz)^*+ik (  \Kk- \Kz  )^*
                         +ik  (\Kprimek-\Kz^\prime)^*+k^2 (\Vk -
                         \Vz)^*, \\
  \label{eq:th-b} 
\Theta_{m,\vext} &=  (\Kprimek-\Kz^\prime)^*-ik (\Vk - \Vz)^*, \\
  \label{eq:th-c} 
 \Theta_{\uext,\lambda} &= -(  \Kk- \Kz  )^*+ik (\Vk - \Vz)^*,\\
  \label{eq:th-d} 
\Theta_{m,\lambda} &= (\Vk - \Vz)^*.
\end{align}
\end{subequations}
Lemma~\ref{lemma:on:adjoint:operators} and the representations~\eqref{eq:differences}
give 
\begin{subequations}
\label{eq:differences-adjoint}
\begin{align}
\label{eq:differences-adjoint-a}
(\Vk - \Vz)^* & = \SV' + \gammazint \AVtilde', & 
(\Kprimek - \Kz')^* & = \SKprime' + \gammaoint \AVtilde', \\ 
\label{eq:differences-adjoint-c}
(\Kk - \Kz)^* & = \SK' + \gammazint \AKtilde',  & 
(\Wk - \Wz)^* & = \SW' - \gammaoint \AKtilde',
\end{align}
\end{subequations}
where the superscript ${\ }^\prime$ indicates that, for a linear 
operator $\varphi \mapsto A \varphi$, the linear operator $A'$ is understood as 
$\varphi \mapsto \overline{A \overline{\varphi}}$.
Inserting~\eqref{eq:differences-adjoint} into~\eqref{eq:th}, and
  taking into account the definitions~\eqref{eq:theta_A} of the
  operators
  $\Theta_{\A,\cdot,\cdot} $ mapping into classes of analytic functions,
we have the decompositions $\Theta_{\cdot,\cdot}
=\Theta_{f,\cdot,\cdot} +\Theta_{\A,\cdot,\cdot} $ with
\begin{align*}
\Theta_{f,\uext,\vext} & =  \SW' +
i k \left(\SK' + \SKprime'\right)  + k^2 \SV', 
\\
\Theta_{f,m,\vext} & = \SKprime' - ik \SV', 
\\
\Theta_{f,\uext,\lambda} & = - (\SK' - ik \SV')  
\\
\Theta_{f,m,\lambda} &= \SV'.
\end{align*}
The G{\aa}rding inequality stated in~\eqref{item:thm:k-explicit-garding-ii} is shown.
\medskip

\emph{Proof of~\ref{item:thm:k-explicit-garding-i}:} The
  expressions for the operators $\Theta_{f,\cdot,\cdot}$ obtained above and the
mapping properties of the operators 
$\SV$, $\SK$, $\SKprime$, and $\SW$ given in 
Lemma~\ref{lemma:decomposition} imply  
\begin{align*}
\|
\Theta_{f,\uext,\vext}\|_{H^{-\frac{3}{2}+s'}(\Gamma)\leftarrow
H^{\frac{1}{2}+s}(\Gamma)} & \lesssim 
k^{-(1+s-s')} + k k^{-(1+s-(s'-1))} +  k k^{-(1+(s+1)-s')} 
+ k^2 k^{-(1+(s+1)-(s'-1))} \\
& \lesssim k^{-(1+s-s')}, 
\qquad 3/2 < s' \leq s+3, \\
\|
\Theta_{f,m,\vext}\|_{H^{-\frac{3}{2}+s'}(\Gamma)\leftarrow
H^{-\frac{1}{2}+s}(\Gamma)} & \lesssim 
k^{-(1+s-s')} + k k^{-(1+s-(s'-1))} \lesssim  k^{-(1+s-s')}, 
\qquad 3/2 < s' \leq s+3, \\
\|\Theta_{f,\uext,\lambda}\|_{H^{-\frac{1}{2}+s'}(\Gamma)\leftarrow
H^{\frac{1}{2}+s}(\Gamma)} & \lesssim 
k^{-(1+s-s')} + k k ^{-(1+(s+1)-s')} \lesssim k^{-(1+s-s')} 
\qquad 1/2 < s' \leq s+3, \\
\|\Theta_{f,m,\lambda}\|_{H^{-\frac{1}{2}+s'}(\Gamma)\leftarrow
H^{-\frac{1}{2}+s}(\Gamma)} & \lesssim 
k^{-(1+s-s')}, 
\qquad 1/2 < s' \leq s+3. 
\end{align*}
This shows that the functions 
$\Theta_{f,\uext,\vext}$, 
$\Theta_{f,m,\vext}$, 
$\Theta_{f,\uext,\lambda}$, 
$\Theta_{f,m,\lambda}$ have the mapping properties
stated in~\eqref{item:thm:k-explicit-garding-i}. \medskip

\emph{Proof of~\eqref{item:thm:k-explicit-garding-iii}:} 
From the representation~\eqref{eq:T1T2T3}
of $\T(\cdot,\cdot)$ in terms of the sesquilinear forms 
$T_1(\cdot,\cdot)$, $T_2(\cdot,\cdot)$, $T_3(\cdot,\cdot)$
and the definition of $\Theta$, we get  
  \[
    \begin{split}
    &\T\left((u,m,\uext),(v,\lambda,\vext)\right)
    +\langle (u,m,\uext), \Theta(v,\lambda,\vext)\rangle\\
    &\qquad =T_1\left((u,m,\uext),(v,\lambda,\vext)\right)
    +T_2\left((u,m,\uext),(v,\lambda,\vext)\right)
    +((\k)^2u,v)_{0,\Omega}.
    \end{split}
    \]
We therefore
concentrate on the terms 
\begin{align*}
T_1\bigl((u,m,\uext),(v,\lambda,\vext)\bigl) & =\Bigl\{ (\Ad \nabla u, \nabla v )_{0,\Omega} 
+ \langle \Wz\uext,\vext\rangle 
+ k^2 \langle \Vz \uext,\vext\rangle  
+ \langle \Vz m, \lambda \rangle 
\Bigr\},  \\
T_2 \bigl((u,m,\uext),(v,\lambda,\vext)\bigr) & = 
\Bigl\{ ik( u, v)_{0,\Gamma} 
- i k \langle \Kz \uext,\vext\rangle 
- i k \langle \Kz^\prime \uext,\vext\rangle 
+\langle (\frac{1}{2}+\Kz^\prime) m,\vext\rangle 
\\ 
\nonumber 
& 
\quad+ i k \langle \Vz m,\vext\rangle 
- ik\langle \Vz \uext, \lambda\rangle 
- \langle (\frac{1}{2} + \Kz) \uext, \lambda\rangle 
- \langle \m, v \rangle 
 +  \langle u,\lambda \rangle 
\Bigr\}.  
\end{align*}
We start with estimating $T_1(\cdot,\cdot)$. 
We introduce $\widetilde{\widetilde\Theta}_{f,\uext,\vext}$ by 
\begin{equation}
\langle \uext, \widetilde{\widetilde\Theta}_{f,\uext,\vext} \vext \rangle 
 = k^2 \langle \uext, \Vz \vext\rangle
\end{equation}
and note that $\widetilde{\widetilde\Theta}_{f,\uext,\vext}$ has mapping property 
given in~\eqref{eq:tildetheta-a}. 
We bound 
\begin{align*}
| T_1\bigl((u,m,\uext),(v,\lambda,\vext)\bigl) 
- \langle \uext, \widetilde{\widetilde\Theta}_{f,\uext,\vext}\vext\rangle | &\lesssim
\|\Ad^{\frac12} \nabla u\|_{0,\Omega}  
\|\Ad^{\frac12} \nabla v\|_{0,\Omega}  \\
&\ + |\uext|_{\frac{1}{2},\Gamma} 
  |\vext|_{\frac{1}{2},\Gamma} 
+ \|m \|_{-\frac{1}{2},\Gamma} 
\|\lambda \|_{-\frac{1}{2},\Gamma}.  
\end{align*}
We turn to $T_2(\cdot,\cdot)$. We introduce 
$\widetilde\Theta_{f,\uext,\vext}$,  
$\widetilde\Theta_{f,m,\vext}$,  and
$\widetilde\Theta_{f,\uext,\lambda}$ by 
\begin{align*}
\langle \uext, \widetilde \Theta_{f,\uext,\vext} \vext \rangle 
 & = -i k \langle \uext, (\Kz + \Kz^\prime) \vext\rangle, \\
\langle m, \widetilde \Theta_{f,m,\vext} \vext\rangle & = 
ik \langle m, \Vz \vext \rangle, \\
\langle \uext, \widetilde \Theta_{f,\uext,\lambda} \uext\rangle & = 
- ik \langle \uext, \Vz \lambda\rangle,
\end{align*}
and note that 
$\widetilde\Theta_{f,\uext,\vext}$,  
$\widetilde\Theta_{f,m,\vext}$,  and
$\widetilde\Theta_{f,\uext,\lambda}$ have the mapping properties given in~\eqref{eq:tildetheta}. We have 
\begin{align*}
& | T_2\bigl((u,m,\uext),(v,\lambda,\vext)\bigl)  - 
\langle \uext, \widetilde\Theta_{f,\uext,\vext} \vext \rangle - 
\langle m, \widetilde\Theta_{f,m,\vext} \vext \rangle - 
\langle \uext, \widetilde\Theta_{f,\uext,\lambda} \lambda\rangle 
| \\
& \lesssim 
k\|u\|_{0,\Gamma} \|v\|_{0,\Gamma} + \|m\|_{-\frac{1}{2},\Gamma}\|\vext\|_{\frac{1}{2},\Gamma} 
+ \|\uext\|_{\frac{1}{2},\Gamma}\|\lambda\|_{-\frac{1}{2},\Gamma} + 
\|m\|_{-\frac{1}{2},\Gamma} \|v\|_{\frac{1}{2},\Gamma}+ 
\|\lambda\|_{-\frac{1}{2},\Gamma} \|u\|_{\frac{1}{2},\Gamma}. 
\end{align*}
The estimate given in~\eqref{item:thm:k-explicit-garding-iii} follows. 
\end{proof}
The G\aa rding inequality and the continuity estimate
of Theorem~\ref{thm:k-explicit-garding}, \eqref{item:thm:k-explicit-garding-ii} and~\eqref{item:thm:k-explicit-garding-iii},
can be 
simplified to be get some fully $k$-explicit bounds, as shown in
  the following corollary.
\begin{cor}
\label{cor:k-explicit-garding}
Let $\Gamma$ be analytic. Then there are $c$, $ C > 0$ independent of $k \ge k_0> 0$ such that 
for all $(u,m,\uext)$ and $(v,\lambda,\vext) \in H^1(\Omega) \times H^{-\frac{1}{2}}(\Gamma) \times H^{\frac{1}{2}}(\Gamma)$ 
there holds the G{\aa}rding inequality 
\begin{align*}
& \Re \left( \T\bigl( (v,\lambda,\vext),(v,\lambda,\vext)\bigr) 
+ \langle (v,\lambda,\vext),\Theta (v,\lambda,\vext)\rangle 
\right)\\
&\qquad\ge c  \bigl\{ \|\Ad^{\frac12} \nabla v\|^2_{0,\Omega} + 
k^2\| n\, v\|^2_{0,\Omega} + 
\|\lambda\|^2_{-\frac{1}{2},\Gamma} + 
\|\vext\|^2_{\frac{1}{2},\Gamma}
\bigr\} \\
&\qquad\quad - C \bigl\{k^2\|n\, v\|^2_{0,\Omega} +  k^4 \|\lambda\|^2_{-\frac{3}{2},\Gamma}  + 
             k^4 
             \|\vext\|^2_{-\frac{1}{2},\Gamma}
             + k^6 \|\vext\|^2_{-\frac{3}{2},\Gamma}
    \bigr\},
\end{align*}
as well as the continuity estimate
\begin{align*}
  & |\T\bigl((u,m,\uext),(v,\lambda,\vext) \bigr)| \\
  &\leq C 
\bigl\{ 
\|\Ad^{\frac12} \nabla u\|^2_{0,\Omega} + k^2\| n\, u\|^2_{0,\Omega}+ k \|u\|^2_{0,\Gamma} + 
\|m\|^2_{-\frac{1}{2},\Gamma} + \|\uext\|^2_{\frac{1}{2},\Gamma} + 
k^2 \|\uext\|^2_{-\frac{1}{2},\Gamma} 
\bigr\}^{\frac12} \\
& \qquad\times 
\bigl\{ 
\|\Ad^{\frac12} \nabla v\|^2_{0,\Omega} +k^2\| n\, v\|^2_{0,\Omega}  + k \|v\|^2_{0,\Gamma} + 
\|\vext\|^2_{\frac{1}{2},\Gamma} + 
                      k^4 
                       \|\vext\|^2_{-\frac{1}{2},\Gamma} + 
                       k^{6} \|\vext\|^2_{-\frac{3}{2},\Gamma}\\
&\qquad\qquad  + 
k^2 \|\lambda \|^2_{-\frac{1}{2},\Gamma} + 
k^4 \|\lambda \|^2_{-\frac{3}{2},\Gamma}  
\bigr\}^{\frac12}.
\end{align*}
\end{cor}
\begin{proof}
In view of Theorem~\ref{thm:k-explicit-garding}, \eqref{item:thm:k-explicit-garding-ii} and~\eqref{item:thm:k-explicit-garding-iii}, we have 
to estimate the terms in $\Theta$ and in $\widetilde\Theta$.
From the definitions of the operators $\Theta_{\A,\cdot,\cdot}$ in~\eqref{eq:theta_A} and
Theorem~\ref{thm:k-explicit-garding}, \eqref{item:thm:k-explicit-garding-i}, with
the (multiplicative) trace inequality, we get
\begin{subequations}
\label{eq:cor:Theta_A}
\begin{align}
  | \langle \uext, \Theta_{{\mathcal A},\uext,\vext} \vext\rangle| & 
\leq C k\|\uext\|_{-\frac12,\Gamma} \left[ 
k^{3} \|\vext\|_{-\frac{3}{2},\Gamma}  + 
k^{2} \|\vext\|_{-\frac{1}{2},\Gamma} \right], \\
  | \langle m, \Theta_{{\mathcal A},m,\vext} \vext\rangle| & 
\leq C \|m\|_{-\frac{1}{2},\Gamma} 
 k^3\|\vext\|_{-\frac{3}{2},\Gamma}, \\
| \langle \uext, \Theta_{{\mathcal A},\uext,\lambda} \lambda \rangle| & 
\leq C k \|\uext\|_{-\frac12,\Gamma} \left[ 
k^2 \|\lambda \|_{-\frac{3}{2},\Gamma}  +
                                                                       k \|\lambda \|_{-\frac{1}{2},\Gamma} \right], \\
| \langle m, \Theta_{{\mathcal A},m,\lambda} \lambda \rangle| & 
\leq C \|m\|_{-\frac{1}{2},\Gamma} 
k^{2} \|\lambda \|_{-\frac{3}{2},\Gamma} . 
\end{align}
\end{subequations}
Furthermore, from Theorem~\ref{thm:k-explicit-garding}, \eqref{item:thm:k-explicit-garding-i}, we get 
\begin{subequations}
\label{eq:cor:Theta_f}
\begin{align}
  \label{eq:cor:Theta_f_a}
| \langle \uext, \Theta_{f,\uext,\vext} \vext\rangle| & 
\leq C k^{-(1-s_1)} \|\uext\|_{\frac{3}{2}-s_1,\Gamma} \|\vext\|_{\frac{1}{2},\Gamma}, 
                                                       \qquad 3/2  < s_1 \leq 3, \\
  \label{eq:cor:Theta_f_b}
| \langle m, \Theta_{f,m,\vext} \vext\rangle| & 
\leq C k^{-(1+1-s_2)}\|m\|_{\frac{3}{2}-s_2,\Gamma} \|\vext\|_{\frac{1}{2},\Gamma}, 
\qquad 3/2 <  s_2 \leq 4,
  \\
  \label{eq:cor:Theta_f_c}
| \langle \uext, \Theta_{f,\uext,\lambda} \lambda \rangle| & 
\leq C k^{-(1-1-s_3)} \|\uext\|_{\frac{1}{2}-s_3,\Gamma} \|\lambda\|_{-\frac{1}{2},\Gamma}, 
                                                            \qquad 1/2 < s_3 \leq 2, \\
  \label{eq:cor:Theta_f_d}
| \langle m, \Theta_{f,m,\lambda} \lambda \rangle| & 
\leq C k^{-(1-s_4)} \|m\|_{\frac{1}{2}-s_4,\Gamma} \|\lambda\|_{-\frac{1}{2},\Gamma}, 
\qquad 1/2 < s_4 \leq 3 
\end{align}
\end{subequations}
(in Theorem~\ref{thm:k-explicit-garding}, \eqref{item:thm:k-explicit-garding-i},
  we have chosen
$s'=s_1$ and $s=0$ for~\eqref{eq:cor:Theta_f_a},
$s'=s_2$ and $s=1$  for~\eqref{eq:cor:Theta_f_b},
$s'=s_3$ and $s=-1$  for~\eqref{eq:cor:Theta_f_c},
$s'=s_4$ and $s=0$  for~\eqref{eq:cor:Theta_f_d},
and, by~\eqref{eq:tildetheta}
\begin{align*}
| \langle \uext, \widetilde{\widetilde\Theta}_{f,\uext,\vext} \vext\rangle| & 
\leq C k^{2} \|\uext\|_{-\frac{1}{2},\Gamma} \|\vext\|_{-\frac{1}{2},\Gamma} \\
| \langle \uext, \widetilde \Theta_{f,\uext,\vext} \vext\rangle| & 
\leq C k \|\uext\|_{-\frac{1}{2},\Gamma} \|\vext\|_{\frac{1}{2},\Gamma} \\
| \langle m, \widetilde \Theta_{f,m,\vext} \vext\rangle| & 
\leq C k \|m\|_{-\frac12,\Gamma} \|\vext\|_{-\frac{1}{2},\Gamma}, \\
| \langle \uext, \widetilde\Theta_{f,\uext,\lambda} \lambda \rangle| & 
\leq C k\|\uext\|_{\frac12,\Gamma} \|\lambda\|_{-\frac{1}{2},\Gamma}. 
\end{align*}
By selecting $s_1 = 2$, $s_2  = 2$, $s_3 = 1$, and $s_4 = 1$ 
in~\eqref{eq:cor:Theta_f}, we obtain 
the stated continuity estimate.

For the G{\aa}rding inequality, we employ 
Theorem~\ref{thm:k-explicit-garding}, \eqref{item:thm:k-explicit-garding-ii}
and have to use 
the bounds~\eqref{eq:cor:Theta_A}, \eqref{eq:cor:Theta_f} with 
$(u,m,\uext) = (v,\lambda,\vext)$.  
In~\eqref{eq:cor:Theta_A} we rearrange the powers of $k$ as follows:  
\begin{subequations}
\label{eq:cor:Theta_A-alternative}
\begin{align}
  | \langle \vext, \Theta_{{\mathcal A},\uext,\vext} \vext\rangle| & 
\leq C k^{2} \|\vext\|_{-\frac12,\Gamma} 
k^{2 } \|\vext\|_{-\frac{3}{2},\Gamma}  + 
k^{3} \|\vext\|^2_{-\frac{1}{2},\Gamma},\\ 
  | \langle \lambda, \Theta_{{\mathcal A},m,\vext} \vext\rangle| & 
\leq C \|\lambda\|_{-\frac{1}{2},\Gamma} 
 k^3\|\vext\|_{-\frac{3}{2},\Gamma}, \\
| \langle \vext, \Theta_{{\mathcal A},\uext,\lambda} \lambda \rangle| & 
\leq C k^2 \|\vext\|_{-\frac12,\Gamma} 
k \|\lambda \|_{-\frac{3}{2},\Gamma}  +
                                                                       k^2 \|\vext\|_{-\frac12,\Gamma} \|\lambda \|_{-\frac{1}{2},\Gamma} , \\
| \langle \lambda, \Theta_{{\mathcal A},m,\lambda} \lambda \rangle| & 
\leq C \|\lambda \|_{-\frac{1}{2},\Gamma} 
k^{2} \|\lambda \|_{-\frac{3}{2},\Gamma} . 
\end{align}
\end{subequations}
Combining these estimates~\eqref{eq:cor:Theta_A-alternative}
with the bounds~\eqref{eq:cor:Theta_f} for the choices 
$s_1 = 2$, $s_2  = 3$, $s_3 = 1$, and $s_4 = 2$ 
we obtain the stated G{\aa}rding inequality using Young's inequality. 
\end{proof}
{\footnotesize
\bibliography{bibliogr,biblio}

\begin{thebibliography}{10}

\bibitem{h2lib}
{H2Lib}.
\newblock Available at \url{www.h2lib.org/}.

\bibitem{NGSolve}
{Netgen/NGSolve}.
\newblock Available at \url{https://ngsolve.org/}.

\bibitem{aurada-feischl-fuehrer-karkulik-melenk-praetorius13}
M.~Aurada, M.~Feischl, T.~F{\"u}hrer, M.~Karkulik, J.M. Melenk, and
  D.~Praetorius.
\newblock Classical {FEM}-{BEM} coupling methods: nonlinearities,
  well-posedness, and adaptivity.
\newblock {\em Comput. Mech.}, 51:399--419, 2013.

\bibitem{baskin-spence-wunsch16}
D.~Baskin, E.~A. Spence, and J.~Wunsch.
\newblock Sharp high-frequency estimates for the {H}elmholtz equation and
  applications to boundary integral equations.
\newblock {\em SIAM J. Math. Anal.}, 48(1):229--267, 2016.

\bibitem{berenger94}
J.-P. Berenger.
\newblock A perfectly matched layer for the absorption of electromagnetic
  waves.
\newblock {\em J. Comput. Phys.}, 114(2):185--200, 1994.

\bibitem{BrakhageWerner}
H.~Brakhage and P.~Werner.
\newblock \"{U}ber das {D}irichletsche {A}ussenraumproblem f\"{u}r die
  {H}elmholtzsche {S}chwingungsgleichung.
\newblock {\em Arch. Math.}, 16:325--329, 1965.

\bibitem{bramble-pasciak07}
J.~H. Bramble and J.~E. Pasciak.
\newblock Analysis of a finite {PML} approximation for the three dimensional
  time-harmonic {M}axwell and acoustic scattering problems.
\newblock {\em Math. Comp.}, 76(258):597--614, 2007.

\bibitem{BuffaHiptmair2005regularized}
A.~Buffa and R.~Hiptmair.
\newblock Regularized combined field integral equations.
\newblock {\em Numer. Math.}, 100(1):1--19, 2005.

\bibitem{burq02}
N.~Burq.
\newblock Semi-classical estimates for the resolvent in nontrapping geometries.
\newblock {\em Int. Math. Res. Not.}, 5:221--241, 2002.

\bibitem{BurtonMiller}
A.~J. Burton and G.~F. Miller.
\newblock The application of integral equation methods to the numerical
  solution of some exterior boundary-value problems.
\newblock {\em Proc. Roy. Soc. London Ser. A}, 323:201--210, 1971.

\bibitem{carstensen-funken99}
C.~Carstensen and S.~A. Funken.
\newblock Coupling of nonconforming finite elements and boundary elements. {I}.
  {A} priori estimates.
\newblock {\em Computing}, 62(3):229--241, 1999.

\bibitem{casati-hiptmair-smajic18}
D.~Casati, R.~Hiptmair, and J.~Smajic.
\newblock Coupling {F}inite {E}lements and {A}uxiliary {S}ources for
  {M}axwell's {E}quations.
\newblock {\em International Journal of Numerical Modeling Electronic Networks,
  Devices and Fields}, 2019.
\newblock doi: \url{https://doi.org/10.1002/jnm.2534}.

\bibitem{actaBEMhelmholtz}
S.~N. Chandler-Wilde, I.~G. Graham, S.~Langdon, and E.~A. Spence.
\newblock Numerical-asymptotic boundary integral methods in high-frequency
  acoustic scattering.
\newblock {\em Acta Numer.}, 21:89--305, 2012.

\bibitem{chaumont2018wavenumber}
T.~Chaumont-Frelet and S.~Nicaise.
\newblock Wavenumber explicit convergence analysis for finite element
  discretizations of general wave propagation problem.
\newblock {\em IMA J. Numer. Anal.}, 2019.
\newblock doi: \url{https://doi.org/10.1093/imanum/drz020}.

\bibitem{claeys-hiptmair12}
X.~Claeys and R.~Hiptmair.
\newblock Electromagnetic scattering at composite objects: a novel multi-trace
  boundary integral formulation.
\newblock {\em ESAIM Math. Model. Numer. Anal.}, 46(6):1421--1445, 2012.

\bibitem{claeys-hiptmair13}
X.~Claeys and R.~Hiptmair.
\newblock Multi-trace boundary integral formulation for acoustic scattering by
  composite structures.
\newblock {\em Comm. Pure Appl. Math.}, 66(8):1163--1201, 2013.

\bibitem{claeys-hiptmair-jerez-hanckes13}
X.~Claeys, R.~Hiptmair, and C.~Jerez-Hanckes.
\newblock Multitrace boundary integral equations.
\newblock In {\em Direct and inverse problems in wave propagation and
  applications}, volume~14 of {\em Radon Ser. Comput. Appl. Math.}, pages
  51--100. De Gruyter, Berlin, 2013.

\bibitem{collino-monk98}
F.~Collino and P.~Monk.
\newblock The perfectly matched layer in curvilinear coordinates.
\newblock {\em SIAM J. Sci. Comput.}, 19(6):2061--2090, 1998.

\bibitem{colton-kress98}
D.~Colton and R.~Kress.
\newblock {\em Inverse acoustic and electromagnetic scattering theory},
  volume~93 of {\em Applied Mathematical Sciences}.
\newblock Springer-Verlag, Berlin, second edition, 1998.

\bibitem{costabel}
M.~Costabel.
\newblock Boundary integral operators on {L}ipschitz domains: elementary
  results.
\newblock {\em SIAM J. Math. Anal.}, 19(3):613--626, 1988.

\bibitem{costabel88a}
M.~Costabel.
\newblock A symmetric method for the coupling of finite elements and boundary
  elements.
\newblock In {\em The mathematics of finite elements and applications, {VI}
  ({U}xbridge, 1987)}, pages 281--288. Academic Press, London, 1988.

\bibitem{demkowicz-ihlenburg01}
L.~Demkowicz and F.~Ihlenburg.
\newblock Analysis of a coupled finite-infinite element method for exterior
  {H}elmholtz problems.
\newblock {\em Numer. Math.}, 88(1):43--73, 2001.

\bibitem{dominguez-ganesh-sayas19}
V.~{Dom{\'\i}nguez}, M.~Ganesh, and F.J. Sayas.
\newblock An overlapping decomposition framework for wave propagation in
  heterogeneous and unbounded media: Formulation, analysis, algorithm, and
  simulation.
\newblock {\em J. Comput. Phys.}, 403, 2020.

\bibitem{erath12}
C.~Erath.
\newblock Coupling of the finite volume element method and the boundary element
  method: an a priori convergence result.
\newblock {\em SIAM J. Numer. Anal.}, 50(2):574--594, 2012.

\bibitem{evansPDE}
L.~C. Evans.
\newblock {\em Partial {D}ifferential {E}quations}.
\newblock American {M}athematical {S}ociety, 2010.

\bibitem{GalkowskiPhDthesis}
J.~Galkowski.
\newblock {\em Distribution of Resonances in Scattering by Thin Barriers}.
\newblock PhD thesis, University of California, Berkeley, 2011.
\newblock Available at
  \url{http://www.homepages.ucl.ac.uk/~ucahalk/thesis.pdf}.

\bibitem{ganesh-morgenstern16}
M.~Ganesh and C.~Morgenstern.
\newblock High-order {FEM}-{BEM} computer models for wave propagation in
  unbounded and heterogeneous media: application to time-harmonic acoustic horn
  problem.
\newblock {\em J. Comput. Appl. Math.}, 307:183--203, 2016.

\bibitem{Gatica_VEMBEM}
G.~N. Gatica and S.~Meddahi.
\newblock Coupling of virtual element and boundary element methods for the
  solution of acoustic scattering problems.
\newblock
  \url{https://www.ci2ma.udec.cl/publicaciones/prepublicaciones/prepublicacion.php?id=372},
  2019.

\bibitem{Gordon71}
W.J. Gordon.
\newblock Blending-function methods of bivariate and multivariate interpolation
  and approximation.
\newblock {\em SIAM J. Numer. Anal.}, 8:158--177, 1973.

\bibitem{GordonHall73b}
W.J. Gordon and Ch.A. Hall.
\newblock Construction of curvilinear co-ordinate systems and applications to
  mesh generation.
\newblock {\em Int. J. Numer. Methods. Eng.}, 7:461--477, 1973.

\bibitem{GordonHall73}
W.J. Gordon and Ch.A. Hall.
\newblock Transfinite element methods: Blending function interpolation over
  arbitrary curved element domains.
\newblock {\em Numer. Math.}, 21:109--129, 1973.

\bibitem{graham-loehndorf-melenk-spence15}
I.~G. Graham, M.~L\"{o}hndorf, J.~M. Melenk, and E.~A. Spence.
\newblock When is the error in the {$h$}-{BEM} for solving the {H}elmholtz
  equation bounded independently of {$k$}?
\newblock {\em BIT}, 55(1):171--214, 2015.

\bibitem{GPS_Helmholtz_het_wellposedness}
I.~G. Graham, O.~R. Pembery, and E.~A. Spence.
\newblock The {H}elmholtz equation in heterogeneous media: a priori bounds,
  well-posedness, and resonances.
\newblock {\em J. Differential Equations}, 266(6):2869--2923, 2019.

\bibitem{grahamSauter_Helmholtz}
I.~G. Graham and S.~A. Sauter.
\newblock Stability and error analysis for the {H}elmholtz equation with
  variable coefficients.
\newblock {\em Math. Comp.}, 89(321):105,138, 2020.

\bibitem{hiptmair-meury06}
R.~Hiptmair and P.~Meury.
\newblock Stabilized {FEM}-{BEM} coupling for {H}elmholtz transmission
  problems.
\newblock {\em SIAM J. Numer. Anal.}, 44(5):2107--2130, 2006.

\bibitem{hiptmair-meury08}
R.~Hiptmair and P.~Meury.
\newblock Stabilized {FEM}-{BEM} coupling for {M}axwell transmission problems.
\newblock In {\em Modeling and computations in electromagnetics}, volume~59 of
  {\em Lect. Notes Comput. Sci. Eng.}, pages 1--38. Springer, Berlin, 2008.

\bibitem{hohage-nannen09}
T.~Hohage and L.~Nannen.
\newblock Hardy space infinite elements for scattering and resonance problems.
\newblock {\em SIAM J. Numer. Anal.}, 47(2):972--996, 2009.

\bibitem{hohage-nannen15}
T.~Hohage and L.~Nannen.
\newblock Convergence of infinite element methods for scalar waveguide
  problems.
\newblock {\em BIT}, 55(1):215--254, 2015.

\bibitem{hohage-schmidt-zschiedrich-I}
T.~Hohage, F.~Schmidt, and L.~Zschiedrich.
\newblock Solving time-harmonic scattering problems based on the pole
  condition. {I}. {T}heory.
\newblock {\em SIAM J. Math. Anal.}, 35(1):183--210, 2003.

\bibitem{hohage-schmidt-zschiedrich-II}
T.~Hohage, F.~Schmidt, and L.~Zschiedrich.
\newblock Solving time-harmonic scattering problems based on the pole
  condition. {II}. {C}onvergence of the {PML} method.
\newblock {\em SIAM J. Math. Anal.}, 35(3):547--560, 2003.

\bibitem{kirsch-monk90}
A.~Kirsch and P.~Monk.
\newblock Convergence analysis of a coupled finite element and spectral method
  in acoustic scattering.
\newblock {\em IMA J. Numer. Anal.}, 10(3):425--447, 1990.

\bibitem{kirsch-monk94}
A.~Kirsch and P.~Monk.
\newblock An analysis of the coupling of finite-element and {N}ystr\"{o}m
  methods in acoustic scattering.
\newblock {\em IMA J. Numer. Anal.}, 14(4):523--544, 1994.

\bibitem{kirsch-monk95}
A.~Kirsch and P.~Monk.
\newblock A finite element/spectral method for approximating the time-harmonic
  {M}axwell system in {${\mathbb R}^3$}.
\newblock {\em SIAM J. Appl. Math.}, 55(5):1324--1344, 1995.

\bibitem{loehndorf-melenk11}
Maike L\"{o}hndorf and Jens~Markus Melenk.
\newblock Wavenumber-explicit {$hp$}-{BEM} for high frequency scattering.
\newblock {\em SIAM J. Numer. Anal.}, 49(6):2340--2363, 2011.

\bibitem{mckay-bruno05}
E.~McKay~Hyde and O.~P. Bruno.
\newblock A fast, higher-order solver for scattering by penetrable bodies in
  three dimensions.
\newblock {\em J. Comput. Phys.}, 202(1):236--261, 2005.

\bibitem{mclean2000strongly}
W.~C.~H. McLean.
\newblock {\em Strongly {E}lliptic {S}ystems and {B}oundary {I}ntegral
  {E}quations}.
\newblock Cambridge {U}niversity {P}ress, 2000.

\bibitem{melenk2012mapping}
J.~M. Melenk.
\newblock Mapping properties of combined field {H}elmholtz boundary integral
  operators.
\newblock {\em SIAM J. Math. Anal.}, 44(4):2599--2636, 2012.

\bibitem{MelenkParsaniaSauter_generalDGHelmoltz}
J.~M. Melenk, A.~Parsania, and S.~Sauter.
\newblock General {D}{G}-methods for highly indefinite {H}elmholtz problems.
\newblock {\em J. Sci. Comput.}, 57(3):536--581, 2013.

\bibitem{mpw2017simultaneous}
J.~M. Melenk, D.~Praetorius, and B.~Wohlmuth.
\newblock Simultaneous quasi-optimal convergence rates in {FEM}-{BEM} coupling.
\newblock {\em Math. Methods Appl. Sci.}, 40(2):463--485, 2017.

\bibitem{melenk-sauter10}
J.~M. Melenk and S.~Sauter.
\newblock Convergence analysis for finite element discretizations of the
  {H}elmholtz equation with {D}irichlet-to-{N}eumann boundary conditions.
\newblock {\em Math. Comp.}, 79(272):1871--1914, 2010.

\bibitem{melenk-sauter20}
J.~M. {Melenk} and S.~{Sauter}.
\newblock {Wavenumber-explicit $hp$-FEM analysis for Maxwell's equations with
  transparent boundary conditions}, 2018.
\newblock arXiv:1803.01619.

\bibitem{melenk-sauter11}
J.M. Melenk and S.~Sauter.
\newblock Wavenumber explicit convergence analysis for finite element
  discretizations of the {H}elmholtz equation.
\newblock {\em {SIAM} J. Numer. Anal.}, 49:1210--1243, 2011.

\bibitem{melenk_phdthesis}
M.~Melenk.
\newblock {\em On {G}eneralized {F}inite {E}lement {M}ethods}.
\newblock PhD thesis, University of Maryland, 1995.

\bibitem{MoiolaSpence_acoustictrans_wavenumberexplicit}
A.~Moiola and E.~A. Spence.
\newblock Acoustic transmission problems: {W}avenumber-explicit bounds and
  resonance-free regions.
\newblock {\em Math. Models Methods Appl. Sci.}, 29(2):317--354, 2019.

\bibitem{nedelec01}
J.~C. N{\'e}d{\'e}lec.
\newblock {\em Acoustic and {E}lectromagnetic {E}quations}.
\newblock Springer, New York, 2001.

\bibitem{SauterSchwab_BEMbook}
S.~A. Sauter and C.~Schwab.
\newblock Boundary {E}lement {M}ethods.
\newblock In {\em Boundary Element Methods}, pages 183--287. Springer, 2010.

\bibitem{sayas09}
F.-J. Sayas.
\newblock The validity of {J}ohnson-{N}\'{e}d\'{e}lec's {BEM}-{FEM} coupling on
  polygonal interfaces.
\newblock {\em SIAM J. Numer. Anal.}, 47(5):3451--3463, 2009.

\bibitem{schatz1974}
A.~H. Schatz.
\newblock An observation concerning {R}itz-{G}alerkin methods with indefinite
  bilinear forms.
\newblock {\em Math. Comp.}, 28(128):959--962, 1974.

\bibitem{schenk2004solving}
O.~Schenk and K.~G{\"a}rtner.
\newblock Solving unsymmetric sparse systems of linear equations with
  {PARDISO}.
\newblock {\em Future Gener. Comp. Sy.}, 20(3):475--487, 2004.

\bibitem{schmidt-deuflhard95}
F.~Schmidt and P.~Deuflhard.
\newblock Discrete transparent boundary conditions for the numerical solution
  of {F}resnel's equation.
\newblock {\em Comput. Math. Appl.}, 29(9):53--76, 1995.

\bibitem{BEM++paper}
W.~{\'S}migaj, T.~Betcke, S.~Arridge, J.~Phillips, and M.~Schweiger.
\newblock Solving boundary integral problems with {BEM++}.
\newblock {\em ACM Transactions on Mathematical Software (TOMS)}, 41(2):6,
  2015.

\bibitem{steinbach_BEMbook}
O.~Steinbach.
\newblock {\em Numerical approximation methods for elliptic boundary value
  problems: {F}inite and {B}oundary {E}lements}.
\newblock Springer Science \& Business Media, 2007.

\bibitem{steinbach11}
O.~Steinbach.
\newblock A note on the stable one-equation coupling of finite and boundary
  elements.
\newblock {\em SIAM J. Numer. Anal.}, 49(4):1521--1531, 2011.

\bibitem{steinbach13}
O.~Steinbach.
\newblock Boundary integral equations for {H}elmholtz boundary value and
  transmission problems.
\newblock In {\em Direct and inverse problems in wave propagation and
  applications}, volume~14 of {\em Radon Ser. Comput. Appl. Math.}, pages
  253--292. De Gruyter, Berlin, 2013.

\bibitem{steinbach-windisch11}
O.~Steinbach and M.~Windisch.
\newblock Stable boundary element domain decomposition methods for the
  {H}elmholtz equation.
\newblock {\em Numer. Math.}, 118(1):171--195, 2011.

\end{thebibliography}
}
\bibliographystyle{plain}

\end{document}